\documentclass[11pt,a4paper]{amsart}

\usepackage{latexsym}
\usepackage{amsmath}
\usepackage{amsthm}
\usepackage{amssymb}
\usepackage{vmargin}
\usepackage{amscd}
\usepackage{stmaryrd}
\usepackage{euscript}
\usepackage{mathrsfs}
\usepackage{amscd}
\usepackage[all]{xy}
\usepackage{xr}

\DeclareMathAlphabet{\mathpzc}{OT1}{pzc}{m}{it}

\newtheorem{theorem}{Theorem}[section]
\newtheorem{proposition}[theorem]{Proposition}
\newtheorem{corollary}[theorem]{Corollary}

\newtheorem{lemma}[theorem]{Lemma}
\newtheorem*{theorem*}{Theorem}
\newtheorem*{proposition*}{Proposition}
\newtheorem*{corollary*}{Corollary}
\newtheorem*{lemma*}{Lemma}
\newtheorem*{conjecture*}{Conjecture}

\theoremstyle{definition}
\newtheorem{definition}[theorem]{Definition}

\newtheorem*{definition*}{Definition}

\theoremstyle{remark}
\newtheorem{example}[theorem]{Example}
\newtheorem{examples}[theorem]{Examples}
\newtheorem{remark}[theorem]{Remark}
\newtheorem{remarks}[theorem]{Remarks}

\newtheorem*{example*}{Example}
\newtheorem*{examples*}{Examples}
\newtheorem*{remark*}{Remark}
\newtheorem*{remarks*}{Remarks}
\newtheorem*{exercise*}{Exercise}

\newcommand\da{\!\downarrow\!}
\newcommand\ra{\rightarrow}
\newcommand\la{\leftarrow}

\newcommand\id{\mathrm{id}}

\newcommand\ten{\otimes}
\newcommand\vareps{\varepsilon}

\newcommand\CC{\mathrm{C}}
\newcommand\DD{\mathrm{D}}

\renewcommand\H{\mathrm{H}}
\newcommand\z{\mathrm{Z}}

\newcommand\N{\mathbb{N}}
\newcommand\Z{\mathbb{Z}}
\newcommand\Q{\mathbb{Q}}

\newcommand\Bu{\mathbb{B}}

\newcommand\bD{\mathbb{D}}

\newcommand\bF{\mathbb{F}}
\newcommand\bG{\mathbb{G}}

\newcommand\bI{\mathbb{I}}
\newcommand\bJ{\mathbb{J}}

\newcommand\bL{\mathbb{L}}

\newcommand\bP{\mathbb{P}}

\newcommand\bS{\mathbb{S}}

\newcommand\C{\mathcal{C}}

\newcommand\cA{\mathcal{A}}
\newcommand\cB{\mathcal{B}}

\newcommand\cD{\mathcal{D}}
\newcommand\cE{\mathcal{E}}

\newcommand\cG{\mathcal{G}}

\newcommand\cM{\mathcal{M}}
\newcommand\cN{\mathcal{N}}

\newcommand\cU{\mathcal{U}}

\renewcommand\O{\mathscr{O}}

\newcommand\sF{\mathscr{F}}

\newcommand\sM{\mathscr{M}}

\newcommand\Del{\mathfrak{Del}}

\newcommand\fG{\mathfrak{G}}

\newcommand\fM{\mathfrak{M}}

\newcommand\fX{\mathfrak{X}}

\renewcommand\L{\Lambda}

\newcommand\m{\mathfrak{m}}

\newcommand\g{\mathfrak{g}}

\newcommand\fm{\mathfrak{m}}

\newcommand\cHom{\mathcal{H}\!\mathit{om}}

\newcommand\Ring{\mathrm{Ring}}
\newcommand\Alg{\mathrm{Alg}}
\newcommand\NAlg{\mathrm{NAlg}}

\newcommand\hyp{\mathrm{hyp}}
\newcommand\Mod{\mathrm{Mod}}

\newcommand\Hom{\mathrm{Hom}}

\newcommand\HHom{\underline{\mathrm{Hom}}}

\newcommand\Ext{\mathrm{Ext}}
\newcommand\EExt{\mathbb{E}\mathrm{xt}}

\newcommand\cone{\mathrm{cone}}

\newcommand\Ob{\mathrm{Ob}\,}

\newcommand\Ab{\mathrm{Ab}}

\newcommand\Gp{\mathrm{Gp}}

\newcommand\hen{\mathrm{hen}}
\newcommand\loc{\mathrm{loc}}

\newcommand\Proj{\mathrm{Proj}\,}
\newcommand\Spec{\mathrm{Spec}\,}
\newcommand\Dec{\mathrm{Dec}\,}
\newcommand\DEC{\mathrm{DEC}\,}

\newcommand\Set{\mathrm{Set}}

\newcommand\Dat{\mathrm{Dat}}

\newcommand\ad{\mathrm{ad}}

\newcommand\Lim{\varprojlim}
\newcommand\LLim{\varinjlim}

\newcommand\into{\hookrightarrow}
\newcommand\onto{\twoheadrightarrow}
\newcommand\abuts{\implies}
\newcommand\xra{\xrightarrow}

\newcommand\pr{\mathrm{pr}}

\newcommand\alg{\mathrm{alg}}

\newcommand\bt{\bullet}
\newcommand\by{\times}

\newcommand\mc{\mathrm{MC}}
\newcommand\mmc{\underline{\mathrm{MC}}}

\newcommand\cMC{\mathcal{MC}}

\newcommand\ddel{\mathrm{Del}}

\newcommand\Symm{\mathrm{Symm}}
\newcommand\SL{\mathrm{SL}}
\newcommand\GL{\mathrm{GL}}

\newcommand\et{\acute{\mathrm{e}}\mathrm{t}}

\newcommand\Tot{\mathrm{Tot}\,}
\newcommand\diag{\mathrm{diag}\,}

\newcommand\ind{\mathrm{ind}}
\newcommand\pro{\mathrm{pro}}

\newcommand\pd{\partial}

\newcommand\half{\frac{1}{2}}

\newcommand\gp{\mathrm{Gp}}
\newcommand\gpd{\mathrm{Gpd}}
\newcommand\Gpd{\mathrm{Gpd}}

\renewcommand\alg{\mathrm{alg}}

\newcommand\cosk{\mathrm{cosk}}

\newcommand\op{\mathrm{opp}}

\newcommand\co{\colon\thinspace}

\newcommand\oR{\mathbf{R}}
\newcommand\oP{\mathbf{P}}
\newcommand\oL{\mathbf{L}}
\newcommand\oSpec{\mathbf{Spec}\,}

\newcommand\on{\mathbf{n}}

\newcommand\oO{\mathbf{0}}
\newcommand\oI{\mathbf{1}}

\newcommand\uleft\underleftarrow
\newcommand\uline\underline
\newcommand\uright\underrightarrow

\begin{document}

\begin{abstract}
We introduce  frameworks for constructing global derived moduli stacks associated to a broad range of problems, bridging the gap between the concrete and abstract conceptions of derived moduli.
Our three approaches  are via differential graded Lie algebras, via cosimplicial groups, and via quasi-comonoids, each more general than the last. 
Explicit examples of derived moduli problems addressed here are finite schemes, polarised projective schemes, torsors, coherent sheaves, and finite group schemes.

\end{abstract}

\title{Constructing derived moduli stacks}
\author{J.P.Pridham}
\thanks{The author was supported during this research by Trinity College, Cambridge; and   by the Engineering and Physical Sciences Research Council [grant number  EP/F043570/1].}

\maketitle

\section*{Introduction}

In \cite{dmsch},  representability was established  for many derived moduli  problems involving schemes and quasi-coherent sheaves. However, the derived stacks there were characterised as nerves of $\infty$-groupoids with very many objects, making it difficult to understand 
the derived stacks concretely.

By contrast to the indirect approach of satisfying a representability theorem, \cite{Hilb} and \cite{Quot} construct explicit derived Hilbert and Quot schemes as dg-schemes with the necessary  properties, but  give no universal family, so the derived moduli spaces lack functorial interpretations. 
In this paper, we will show how to reconcile these approaches, thereby giving explicit presentations for the derived moduli spaces of \cite{dmsch}. 

In fact, we go substantially beyond the problems considered in \cite{Hilb} and \cite{Quot}, and  give a framework valid in all characteristics (rather than just over $\Q$). This is done by working with quasi-comonoid-valued functors, which give a  global analogue of the simplicial deformation complexes of  \cite{paper2}. In broad terms, derived moduli constructions over $\Q$ tend to be based on differential graded Lie algebras (DGLAs), while  quasi-comonoids perform the same r\^ole in much greater generality.  Since quasi-comonoids 
arise naturally from algebraic theories, they are  much more general than DGLAs, even   in characteristic zero.
 
Beware that for the purposes of this paper, derived algebraic geometry will mean the theory of \cite{lurie} based on simplicial commutative rings, or on dg algebras when working over $\Q$, rather than the more exotic HAG contexts of \cite{hag2}. This enables us to apply Lurie's Representability Theorem in \S \ref{background}, but is also  needed in later sections. The key to \S \ref{dglamoduli} is that tensoring a commutative algebra with a Lie algebra gives a Lie algebra, but similar constructions could be made with any pair of algebras for Koszul-dual operads. Likewise, the constructions of \S \S\ref{cgpmoduli}--\ref{qmmoduli} adapt to give functors on any category of simplicial objects. However, they will not adapt to give functors on symmetric spectra, since they depend on the functor $A_{\bt} \leadsto A_0$.

The structure of the paper is as follows.
Section \ref{background} summarises various results from \cite{drep} concerning representability of derived stacks, and gives a few minor generalisations. Section \ref{sheafsn} develops some technical results on the pro-Zariski and pro-\'etale sites. Lemma \ref{fpsheafpro} shows that any finitely presented sheaf  is a sheaf for the associated pro site, 
 and our main results are Lemmas \ref{univZarcover} and \ref{univetcover}, concerning the existence of weakly universal coverings. These are applied in later sections to deal with  infinite sums of locally free sheaves, which feature when studying polarised projective varieties.  

In Section \ref{dglamoduli}, DGLAs are introduced, together with the Deligne groupoid $\Del(L)$ associated to any DGLA $L$ with a gauge action. By adapting the techniques of \cite{Quot},  DGLAs are used to construct derived moduli stacks for  pointed finite schemes (Proposition \ref{repfin}) and for polarised projective schemes (Proposition \ref{reppol}). The resulting functors are shown (in Propositions \ref{finconsistent} and \ref{polconsistent}, respectively) to be equivalent to the   corresponding functors in \cite{dmsch}, defined as nerves of $\infty$-groupoids of derived geometric stacks. 

DGLAs only tend to work in characteristic $0$, and  
Section \ref{cgpmoduli} shows how to construct derived moduli stacks using cosimplicial groups instead. For any simplicial cosimplicial group $G$, there is a derived Deligne groupoid $\uline{\Del}(G)$; 
Proposition \ref{delgpnice} shows that cosimplicial group-valued functors $G$ give rise to well-behaved derived moduli functors  $\uline{\Del}(\uline{G})$. For any DGLA $L$ with gauge $G_L$, there is an associated  cosimplicial group $D(\exp(L),G_L)$, and 
Corollary \ref{cfexp2} shows that the  Deligne groupoids associated to $L$ and $D(\exp(L),G_L)$ are isomorphic.
\S \ref{shfcgpsn} defines a kind of  sheafification $G^{\sharp}$  for cosimplicial group-valued functors $G$, removing the need to sheafify $\uline{\Del}(\uline{G})$; this gives an immediate advantage of cosimplicial groups over DGLAs.
 Proposition \ref{torsorconsistent} gives a  cosimplicial group governing derived moduli of torsors, a problem not easily accessible via DGLAs. 

Cosimplicial groups cannot handle all moduli problems, so  Section \ref{qmmoduli} begins by recalling the  quasi-comonoids of \cite{monad}, and the derived Deligne groupoid $\uline{\Del}(E)$ of a simplicial quasi-comonoid $E$. Corollary \ref{delqmnice} then shows that quasi-comonoid-valued functors $E$ give rise to well-behaved derived moduli functors  $\uline{\Del}(\uline{E})$.
 In \S \ref{monads}, we recall basic properties of monads, together with results from \cite{monad} showing how these give rise to quasi-comonoids. Monads are ubiquitous, arising whenever there is some kind of algebraic structure.
\S \ref{diagrams} goes further, by associating quasi-comonoids to diagrams. In particular, this allows derived moduli of morphisms  to be constructed for all the examples considered in \S \ref{qmegs}. \S \ref{shfqmsn} then  defines a kind of  sheafification $E^{\sharp}$  for quasi-comonoid-valued functors $E$, removing the need to sheafify $\uline{\Del}(\uline{E})$.

For every cosimplicial group $G$, there is a  quasi-comonoid $\cE(G)$, and Lemma \ref{cfdef} shows that $\uline{\Del}(\cE(G))\simeq \uline{\Del}(G)$, ensuring consistency between the various approaches. 
For moduli problems based on additive categories, the associated quasi-comonoid $E$ is linear. This means that its normalisation $NE$ is a DG associative non-commutative algebra (so \emph{a fortiori} a DGLA), so the techniques of this section give DGLAs for abelian moduli problems. Moreover,
Proposition \ref{mcqmdgla} gives an equivalence $\uline{\Del}(E)\simeq \uline{\Del}(NE)$, so quasi-comonoids and DGLAs give equivalent derived moduli.

Section \ref{qmegs} gives a selection of examples which can be tackled by quasi-comonoids. Derived moduli of 
finite schemes, of polarised projective schemes, and of finite group schemes  are constructed in  Propositions \ref{repfinqm}, \ref{reppolqm} and \ref{repfingpqm}, respectively. In Propositions \ref{finconsistentqm}, \ref{polconsistentqm} and \ref{fingpconsistentqm}, these are shown to be equivalent to the  corresponding functors in \cite{dmsch}, defined as nerves of $\infty$-groupoids of derived geometric stacks. Proposition \ref{cohmod} constructs derived moduli of coherent sheaves, and Proposition \ref{cohconsistent} shows that this is equivalent to the nerve of $\infty$-groupoids of hypersheaves considered in \cite{dmsch}.

\tableofcontents

\section{Background on representability}\label{background}

Let $\bS$ be the category of simplicial sets. Denote the  category of simplicial commutative rings by $s\Ring$, the category of simplicial commutative $R$-algebras by $s\Alg_R$, and the category of simplicial $R$-modules by $s\Mod_R$. 
If $\Q \subset R$, we let $dg_+\Alg_R$ be the category of differential graded-commutative $R$-algebras in non-negative chain degrees,
  and $ dg_+\Mod_R$ the category of $R$-modules in  chain complexes in non-negative chain degrees.  

\begin{definition}\label{normdef}
 Given a simplicial abelian group $A_{\bt}$, we denote the associated normalised chain complex  by $N^sA$ (or, when no confusion is likely, by $NA$). Recall that this is given by  $N(A)_n:=\bigcap_{i>0}\ker (\pd_i: A_n \to A_{n-1})$, with differential $\pd_0$. Then $\H_*(NA)\cong \pi_*(A)$.

When $\Q \subset R$, using the Eilenberg--Zilber shuffle product (\cite{W} 8.5.4),  normalisation  $N$ extends to a right Quillen equivalence
$$
N:s\Alg_R \to dg_+\Alg_R,
$$  
by \cite{QRat} \S I.4.
\end{definition}

\begin{definition}
Define $dg_+\cN_R$ (resp. $s\cN_R$)  to be the full subcategory of $dg_+\Alg_R$ (resp. $s\Alg_R$) consisting of objects $A$
 for which the map $A \to \H_0A$ (resp. $A \to \pi_0A$) has nilpotent kernel. Define $dg_+\cN_R^{\flat}$ (resp. $s\cN_R^{\flat}$) to be the full subcategory of $dg_+\cN_R$ (resp. $s\cN_R$) consisting of objects $A$
 for which $A_i=0$ (resp. $N_iA=0$) for all $i \gg 0$.
\end{definition}

From now on, we will write $d\cN^{\flat}$ (resp. $d\Alg_R$, resp. $d\Mod_R$) to mean either $s\cN_R^{\flat}$ (resp. $s\Alg_R$, resp. $s\Mod_R$) or $dg_+\cN_R^{\flat}$ (resp. $dg_+\Alg_R$, resp. $dg_+\Mod_R$), noting that we only use chain algebras   in characteristic $0$.

\begin{definition}
Say that a  surjection $A \to B$ in  $dg_+\Alg_R$ (resp. $s\Alg_R$) is a \emph{tiny acyclic extension} if the kernel $K$ satisfies $I_A\cdot K=0$, and  $K$ (resp. $NK$) is of the form $\cone(M)[-r]$ for some $\H_0A$-module  (resp. $\pi_0A$-module) $M$. In particular, $\H_*K=0$.
\end{definition}

\subsection{Formal quasi-smoothness and homogeneity}

The following definitions are mostly taken from \cite{dmsch}.

\begin{definition}
Say that a natural transformation $\eta: F \to G$ of functors $F, G:  d\cN^{\flat} \to \bS$  is homotopic (resp. pre-homotopic) if for all tiny acyclic extensions $A \to B$, the map
\[
F(A) \to F(B)\by_{G(B)}G(A)
\]
is a trivial fibration (resp. a surjective fibration). Say that $F$ is homotopic if $F \to \bt$ is so, where $\bt$ denotes the one-point set.
\end{definition}

\begin{definition}
Say that a natural transformation $\eta: F \to G$ of functors $F, G:  d\cN^{\flat} \to \bS$  is formally quasi-presmooth (resp. formally presmooth)  if for all square-zero extensions $A \to B$, the map
\[
F(A) \to F(B)\by_{G(B)}G(A)
\]
is a fibration (resp. a surjective fibration).

Say that $\eta$ is formally quasi-smooth (resp. formally smooth) if it is formally quasi-presmooth (resp. formally presmooth)  and homotopic. 
\end{definition}

\begin{definition}
Say that a natural transformation $\eta: F \to G$ of functors on $  d\cN^{\flat}$ is formally \'etale if 
for all square-zero extensions $A \to B$, the map
\[
F(A) \to F(B)\by_{G(B)}G(A)
\]
is an isomorphism.
\end{definition}


\begin{definition}
Say that a natural transformation $F \to G$ of functors on $ d\cN^{\flat}$ is (relatively) homogeneous if for all square-zero extensions $A \to B$, the map
\[
F(A\by_BC ) \to G(A\by_BC)\by_{[G(A)\by_{G(B)}G(C)]}[F(A)\by_{F(B)}F(C)]
\]
is an isomorphism. Say that $F$ is homogeneous if $F \to \bt$ is relatively homogeneous.
\end{definition}

\begin{proposition}\label{fethgs}
Let $\alpha:F \to G$ be a formally \'etale morphism of functors $F, G :  d\cN^{\flat} \to \Set$. If $G$ is homogeneous, then so is $F$. Conversely, if  $\alpha$ is surjective  and $F$ is  homogeneous, then so is $G$.
\end{proposition}
\begin{proof}
This is \cite{dmsch} Proposition \ref{dmsch-fethgs}.
\end{proof}

\subsection{Tangent complexes}

Given a category $\C$, write $\Ab(\C)$ for the category of abelian group objects in $\C$.

\begin{definition}
For a  homogeneous functor $F:d\cN^{\flat} \to \bS$, $A \in d\cN^{\flat} $ and $M \in d\Mod_A$, define the tangent space by
$$
T(F,M):= F(A \oplus M) \in \bS \da F(A),
$$
noting that this is an abelian group object in this category. Here, $\bS \da F(A) $ denotes the category of objects over $ F(A)$.

Given a natural transformation $\alpha:F \to G$ of  homogeneous functors $F,G:d\cN^{\flat}\to \bS$, define the relative tangent space by
$$
T(F/G, M):= \ker (T(F,M)\to T(G,M)\by_{G(A)}F(A)) \in \Ab(\bS \da F(A)).
$$

Given $x \in F(A)$, define $T_x(F/G,M):=T(F/G,M)\by_{F(A)}\{x\} \in \Ab(\bS)=s\Ab$. 
\end{definition}

 When $\alpha:F \to G$ is formally quasi-presmooth, note that this definition is compatible with \cite{drep} Definition \ref{drep-Tdef}, in the sense that for $x \in \pi_0F(A)$, the space $T_x(F/G)(M)$ of \cite{drep} is the homotopy fibre of  $T(F/G,M)\to F(A)$ over $x$, since $T(F/G,M) \to F(A)$ is a fibration.

\begin{definition}\label{pretotcohodef}
Given a pre-homotopic formally quasi-presmooth transformation $F\xra{\alpha} G$ of homogeneous functors $F,G:d\cN^{\flat} \to \bS$, an object $A \in  d\cN^{\flat}$, a point $x \in F_0(A)$, and  a module $M\in d\Mod_A$, define $\DD^{i}_x(F/G,M)$ as follows, following \cite{drep} Definition \ref{drep-pretotcohodef}.

For $i\le 0$, set 
\[
\DD^i_x(F/G,M):= \pi_{-i}( T_x(F/G,M) ). 
\]
For $i >0$, set 
\[
\DD^i_x(F,M):= \pi_0F( T_x(F/G,M[-i]) /\pi_0(T_x(F/G, \cone(M)[1-i])). 
\]

Note that homogeneity of $F$ ensures that these are abelian groups for all $i$, and that the multiplicative action of $A$ on $M$ gives them the structure of $A$-modules.
\end{definition}

If $F\xra{\alpha} G$ is  formally quasi-smooth, note that \cite{drep} Lemma \ref{drep-adf} gives
$$
\DD^{n-i}_x(F/G,M)= \pi_i (T_x(F/G,M[-n])).
$$

The following is immediate.
\begin{lemma}
If $X,Y,Z:d\cN^{\flat} \to \bS$ are homogeneous, and   $X \xra{\alpha} Y$ is formally quasi-presmooth, with $\beta:Z \to Y$ any map,  set $T:=X\by_YZ $, and observe that $T \to Z$ is quasi-presmooth. There is  an isomorphism
$$
\DD^*_t(T/Z,M) \cong \DD^*_{x}(X/Y,M),
$$ 
for $t \in T(A)$ with image $x \in X(A)$.
\end{lemma}

\begin{proposition}\label{longexact}
Let $X,Y, Z:d\cN^{\flat}\to \bS$ be homogeneous functors, with  $X \xra{\alpha} Y$ and $Y \xra{\beta} Z$ formally quasi-smooth.    For $x \in X(A)$, there is then a long exact sequence
$$
\ldots \xra{\pd} \DD^j_x(X/Y,M) \to \DD^j_x(X/Z,M) \to \DD^j_y(Y/Z,M) \xra{\pd} \DD^{j+1}_x(X/Y,M) \to \DD^{j+1}_x(X/Z,M) \to \ldots,
$$
where $y \in Y(A)$ is the image of $x$.
\end{proposition}
\begin{proof}
Since $T_x(X/Y,M)=\ker(\alpha:T_x (X/Z,M) \to T_y (Y/Z,M))$, we have fibration sequences
\[
\ldots \to \pi_iT_x(X/Y,M[-n]) \to \pi_iT_x(X/Z,M[-n]) \to  \pi_iT_x(Y/Z,M[-n])   \to \ldots  
\]
for all $i,n\ge 0$
so the result follows because $\pi_iT_x(X/Y,M[-n])= \DD^{n-i}(X/Y,M)$, and similarly for $X/Z, Y/Z$.
\end{proof}

\begin{definition}
Recall that a local coefficient system on $S \in \bS$ is an object $V$ of $\Ab(\bS \da S)$ for which the maps $\pd_i : V_s \to V_{\pd_is}$ are isomorphisms for all $ s \in S_n$, where $V_s:= V_n\by_{S_n}\{s\}$.
\end{definition}

\begin{lemma} \label{tanlocsys}
If  $\alpha: X \to  Y$ is a  formally quasi-smooth    morphism between homogeneous functors, take 
an object $A \in d\cN^{\flat}$ and  $M\in d\Mod_A$.
Then there is a local coefficient system 
$$
\DD^*(X/Y,M)
$$
on $X(A)$, whose stalk at $x \in X(A)$ is $\DD^*_x(X/Y,M)$. In particular, $\DD^*_x(X/Y,M)$ depends (up to non-canonical isomorphism) only on the image of $x$ in $\pi_0X(A)$.
\end{lemma}
\begin{proof}
As with \cite{drep} Lemma \ref{drep-tanlocsys}, this follows straightforwardly from the proof of  \cite{drep} Lemma \ref{drep-pi0tan}. 
\end{proof}

\subsubsection{Obstructions}

\begin{proposition}\label{robs}
If $F, G:d\cN^{\flat}\to \bS$ are homogeneous, with $G$  pre-homotopic and $\alpha:F \to G$  formally quasi-smooth, 
 then for any square-zero extension $e:I \to A \xra{f} B$ in $d\cN^{\flat}$, there is a sequence of sets
$$
\pi_0(FA)\xra{f_*} \pi_0(FB\by_{GB}GA) \xra{o_e} \Gamma(FB,\DD^1(F/G, I)), 
$$  
where $\Gamma(-)$ denotes the global section functor. 
This is exact in the sense that the fibre of $o_e$ over $0$ is the image of $f_*$ 
 Moreover,  there is a group action of $\DD^0_x(F/G, I)$ on the fibre of $\pi_0(FA) \to \pi_0(FB)$ over $x$, whose orbits are precisely the fibres of $f_*$. 

For any $y \in F_0A$, with $x=f_*y$, the fibre of $FA \to FB\by_{GB}GA$ over $x$ is isomorphic to $T_{x}(F/G,I)$, and the sequence above 
extends to a long exact sequence
$$\xymatrix@R=0ex{
\cdots \ar[r]^-{e_*} &\pi_n(FA,y) \ar[r]^-{f_*}&\pi_n(FB\by_{GB}GA,x) \ar[r]^-{o_e}& \DD^{1-n}_{y}(F/G,I) \ar[r]^-{e_*} &\pi_{n-1}(FA,y)\ar[r]^-{f_*}&\cdots\\ &\cdots \ar[r]^-{f_*}&\pi_1(FB\by_{GB}GA,x) \ar[r]^-{o_e}& \DD^0_{y}(F/G,I)  \ar[r]^-{-*y} &\pi_0(FA).
}
$$
\end{proposition}
\begin{proof}
The proof of \cite{drep} Proposition \ref{drep-obs} carries over to this generality.
\end{proof}

\begin{corollary}\label{cohosmoothchar}
If $F, G:d\cN^{\flat}\to \bS$ are homogeneous, with $G$  pre-homotopic and $\alpha:F \to G$  pre-homotopic and formally  quasi-presmooth, then $\alpha$ is formally presmooth if and only if $\DD^i_x(F/G,M)=0$ for all $i>0$, all discrete rings $A$, all $x \in \pi_0F(A)$ and all $A$-modules $M$.
\end{corollary}

\begin{definition}
 Given a functor $F$ on $d\cN^{\flat}_R$, define the functor $\pi^0F$ on $\Alg_{\H_0R}$ by $(\pi^0F)(A):= F(A)$.
\end{definition}

\begin{corollary}\label{detectweak}
Take a morphism $\alpha: F \to F'$ of homogeneous formally quasi-smooth functors  $F, F':d\cN^{\flat}\to \bS$. Then $\alpha$ is a weak equivalence if and only if 
\begin{enumerate}
        \item $\pi^0\alpha: \pi^0F \to \pi^0F'$ is a weak equivalence of functors $\Alg_{\H_0R}\to \bS$, and
\item the maps $\DD^i_x(F,M) \to \DD^i_{\alpha x}(F',M)$ are isomorphisms for all   $A \in \Alg_{\H_0R}$, all $A$-modules $M$, all $x \in F(A)_0$, and all $i>0$.
\end{enumerate}
\end{corollary}
\begin{proof}
For any $A\in d\cN^{\flat}$, we need to show that $\alpha_A:F(A) \to F'(A)$ is a weak equivalence. By hypothesis, we know that this holds if we replace $A$ with $\H_0A$.  Now, the map $A \to \H_0A$ is a nilpotent extension; let the kernel be $I_A$. The maps $A/I_A^{n+1} \to A/I_A^{n}$ are  square-zero extensions, and their kernels  $I_A^n/I_A^{n+1}$ are $\H_0A$-modules. This allows us to proceed inductively, using the long exact sequence of Proposition \ref{robs}
 to deduce that $F(A/I_A^{n+1}) \to F'(A/I_A^{n+1})$ is a weak equivalence whenever $F(A/I_A^{n}) \to F'(A/I_A^{n})$ is so.
\end{proof}

\subsection{Representability}

For the remainder of this section, $R$ will be a  derived G-ring admitting a dualising module (in the sense of \cite{lurie} Definition 3.6.1). In particular, this is satisfied if $R$ is a G-ring admitting  a dualising complex in the sense of \cite{HaRD} Ch. V. Examples are $\Z$, any field,  or any  Gorenstein local ring.

\begin{theorem}\label{mylurierep3}
Take a functor $F: d\cN_R^{\flat} \to \bS$ satisfying the following conditions:

\begin{enumerate}

\item $F$ is formally quasi-smooth.

\item\label{trunc} For all discrete rings $A$, $F(A)$ is $n$-truncated, i.e. $\pi_iF(A)=0$ for all $i>n$.

\item\label{cohesive} 
$F$ is homogeneous.

\item\label{stack} $\pi^0F:\Alg_{\H_0R} \to \bS$ is a hypersheaf for the \'etale topology. 

\item\label{afp1a} $\pi_0\pi^0F:  \Alg_{\H_0R} \to \Set$  preserves filtered colimits.

\item\label{afp1b} For all $A \in \Alg_{\H_0R}$ and all $x \in F(A)$, the functors $\pi_i(\pi^0F,x): \Alg_A \to \Set$  preserve filtered colimits for all $i>0$.

\item 
for all finitely generated integral domains $A \in \Alg_{\H_0R}$, all $x \in F(A)_0$ and all \'etale morphisms $f:A \to A'$, the maps
\[
\DD_x^*(F, A)\ten_AA' \to \DD_{fx}^*(F, A')
\]
are isomorphisms.

\item\label{afp2} for all finitely generated $A \in \Alg_{\H_0R}$  and all $x \in F(A)_0$, the functors $\DD^i_x((F/R), -): \Mod_A \to \Ab$ preserve filtered colimits for all $i>0$.

\item for all finitely generated integral domains $A \in \Alg_{\H_0R}$  and all $x \in F(A)_0$, the groups $\DD^i_x(F/R, A)$ are all  finitely generated $A$-modules.

\item\label{effective} for any complete discrete local Noetherian  $\H_0R$-algebra $A$, with maximal ideal $\m$, the map
$$
\pi^0F(A) \to {\Lim}F(A/\m^r)
$$
is a weak equivalence.
\end{enumerate}

Then $F$ is the restriction to $d\cN_R^{\flat}$ of  an almost finitely presented geometric derived $n$-stack $F':d\Alg_R \to \bS$. Moreover, $F'$ is uniquely determined by $F$ (up to weak equivalence).
\end{theorem}
\begin{proof}
This variant of Lurie's Representability Theorem  essentially  appears as \cite{drep} Theorem \ref{drep-lurierep3}, which takes a homotopy-preserving, homotopy-homogeneous functor instead of a formally quasi-smooth homogeneous functor. However,  
 every homotopic functor is homotopy-preserving (by \cite{dmsch} Lemma \ref{dmsch-2htpicgood}), while every formally quasi-presmooth homogeneous functor is homotopy-homogeneous (by \cite{dmsch} Lemma \ref{dmsch-2hgsgood}). Finally, note that formal quasi-presmoothness allows us to replace homotopy limits with limits.
\end{proof}

\begin{remark}\label{stackhyper}
 For the definition of hypersheaves featuring in (\ref{stack}) above, see \cite{drep} Definition \ref{drep-hypersheafdef}. For all the applications in this paper, the following observation suffices.  Given a groupoid-valued functor $\Gamma: \Alg_{\H_0R} \to \gpd$, the nerve $B\Gamma : \Alg_{\pi_0R} \to \bS$ is a hypersheaf if and only if $\Gamma$ is a stack (in the sense of \cite{champs}).
\end{remark}

\begin{remark}\label{cflurie}
Note that there are slight differences in terminology between \cite{hag2} and \cite{lurie}. In the former, only disjoint unions of affine schemes are $0$-representable, so arbitrary schemes are $2$-geometric stacks, and Artin stacks are $1$-geometric stacks if and only if they have affine diagonal. In the latter, algebraic spaces are $0$-stacks.  A geometric $n$-stack  is called $n$-truncated in \cite{hag2}, and it follows easily that every $n$-geometric stack in \cite{hag2} is $n$-truncated. A weak converse is that  every  geometric $n$-stack  is  $(n+2)$-geometric. 

Theorem \ref{mylurierep3} follows the convention from \cite{lurie}, so ``geometric derived $n$-stack'' means ``$n$-truncated derived geometric stack''.

Beware, however, that Condition (\ref{trunc}) of the theorem only applies to \emph{discrete} rings. In general, if $A \in d\cN_R^{\flat}$ with $\H_iA=0$ for $i>m$, then a geometric derived $n$-stack $F$ will have the property that $\pi_j F(A)=0$ for all $j>m+n$. 
\end{remark}

\subsection{Pre-representability}

Concrete approaches to derived moduli can naturally produce functors $F \co d\cN_R^{\flat} \to \bS$ with the property that $\pi_iF(A)=0$ for all $i>n$ and \emph{all} $A$. Such functors are not  geometric derived $n$-stacks, since they cannot be both homotopic and homogeneous. The purpose of this section is to establish weaker conditions which can be satisfied by such functors, and still allow us to  associate a geometric derived $n$-stack $\uline{F}$ to $F$. In particular, all of the examples in \S\S \ref{dglamoduli}--\ref{qmegs} will work by constructing derived geometric $1$-stacks $\uline{F}$ from groupoid-valued functors $F$.


\begin{definition}\label{salgstr}
Define a  simplicial enrichment of $s\cN_R^{\flat}$   as follows. For  $A\in s\cN_R^{\flat}$ and a finite simplicial set $K$,   $A^K \in s\cN_R^{\flat}$ is defined by 
$$
(A^K)_n:= \Hom_{\bS}(K \by \Delta^n, A).
$$ 

Spaces  $\HHom(A,B) \in \bS$ of morphisms are then given by
$$
\HHom_{s\cN_R^{\flat}}(A, B)_n:= \Hom_{s\cN_R^{\flat}}(A, B^{\Delta^n}).
$$
\end{definition}

\begin{definition}
Define a  simplicial enrichment of $dg_+\cN_R^{\flat}$   as follows.
        First set $\Omega_n$ to be the  differential graded algebra 
$$
\Q[t_0, t_1, \ldots, t_n,dt_0, dt_1, \ldots, dt_n ]/(\sum t_i -1, \sum dt_i)
$$  
of rational differential forms on the $n$-simplex $\Delta^n$, where $t_i$ is of degree $0$.
 These fit together to form a simplicial complex $\Omega_{\bt}$ of graded-commutative DG-algebras, and we define $A^{\Delta^n}$ as the good truncation $A^{\Delta^n}:= \tau_{\ge 0}(A \ten \Omega_n)$.  [Note that this construction  commutes with finite limits, so  extends to define $A^K$ for finite simplicial sets $K$.]

Spaces  $\HHom(A,B) \in \bS$ of morphisms are then given by
$$
\HHom_{dg_+\cN_R^{\flat}}(A, B)_n:= \Hom_{dg_+\cN_R^{\flat}}(A, B^{\Delta^n}).
$$
\end{definition}

\begin{definition}
Given
a functor $F: d\cN^{\flat} \to \bS$, we  define $\uline{F}: d\cN^{\flat}\to s\bS$, (for $s\bS$ the category of bisimplicial sets), by 
$$
\underline{F}(A)_{n} :=  F(A^{\Delta^n}).
$$
\end{definition}

\begin{definition}\label{barwdef}
Define  $\bar{W}:s\bS \to \bS$ to be the right adjoint to  Illusie's total $\Dec$ functor given by $\DEC(X)_{mn}= X_{m+n+1}$. Explicitly,
\[
 \bar{W}_p(X) = \{(x_0, x_1, \ldots, x_p) \in
 \prod^p_{i=0} X_{i,p-i} | \pd^v_0 x_i = \pd^h_{i+1}x_{i+1},\, \forall 0 \le i <p\}
\]
with operations
\begin{eqnarray*}
 \pd_i(x_0, \ldots, x_p) &=& (\pd^v_i x_0, \pd^v_{i-1}x_1, \ldots , \pd^v_1 x_{i-1}, \pd^h_ix_{i+1}, \pd^h_i x_{i+2}, \ldots, \pd^h_i x_p),\\
\sigma_i(x_0, \ldots, x_p) &=& (\sigma^v_i x_0,\sigma^v_{i-1}x_1, \ldots , \sigma^v_0 x_i, \sigma^h_i x_i, \sigma^h_i x_{i+1}, \ldots, \sigma^h_i x_p).
\end{eqnarray*}
\end{definition}

In \cite{CRdiag},  it is established that the canonical natural transformation
\[
\diag X \to \bar{W}X
\]
from the diagonal is a weak equivalence for all $X$.

\begin{lemma}\label{ulinec}
For a homotopic functor $F:d\cN^{\flat} \to \bS$, the natural transformation $F \to \bar{W}F$  is a weak equivalence.
\end{lemma}
\begin{proof}
 This is \cite{drep} Lemma \ref{drep-ulinec}.       
\end{proof}

\begin{proposition}\label{settotop}
If a formally quasi-presmooth homogeneous functor $F:d\cN^{\flat} \to \bS$ is  pre-homotopic, then the   functor  $\bar{W}\underline{F}:d\cN^{\flat} \to \bS$ is homogeneous and formally quasi-smooth.
\end{proposition}
\begin{proof}
 This is essentially the same as \cite{drep} Corollaries \ref{drep-settotop} and \ref{drep-settotopb} (replacing weak equivalences with isomorphisms, and homotopy fibre products with fibre products), using the result from \cite{CRfib} that diagonal fibrations are $\bar{W}$-fibrations.
\end{proof}

\begin{lemma}\label{totcohoc}
Given a formally quasi-presmooth pre-homotopic homogeneous functor $F: d\cN_R^{\flat} \to \bS$, an object $A \in  d\cN^{\flat}$, a  point $x \in F_0(A)$, and a module      $M \in d\Mod_A$, there are canonical isomorphisms
\[
\DD^i_x(F,M) \cong \DD^i_x(\bar{W}\uline{F},M).
\]
\end{lemma}
\begin{proof}
 The proof of \cite{drep} Lemma \ref{drep-totcohoc}, which deals with the case when $A$ and $M$ are discrete, carries over to this generality.       
\end{proof}

\begin{corollary}\label{predetectweak}
 If a formally quasi-presmooth homogeneous functor $F:d\cN^{\flat} \to \bS$ is  pre-homotopic, and admits a morphism $\alpha:F \to G$ to a formally quasi-smooth   homogeneous functor, then $\alpha$ induces a functorial weak equivalence
\[
 \bar{W}\uline{F} \simeq G       
\]
if and only if
\begin{enumerate}
        \item $\pi^0\alpha: \pi^0F \to \pi^0G$ is a weak equivalence of functors $\Alg_{\H_0R}\to \bS$, and
\item the maps $\DD^i_x(F,M) \to \DD^i_{\alpha x}(G,M)$ are isomorphisms for all   $A \in \Alg_{\H_0R}$, all $A$-modules $M$, all $x \in F(A)_0$, and all $i>0$.
\end{enumerate}
\end{corollary}
\begin{proof}
    By Lemma \ref{totcohoc}, the map from $F$ to $\bar{W}\uline{F}$ induces isomorphisms on $\DD^i$, so the maps
\[
\DD^i_{x}(\bar{W}\uline{F},M) \to \DD^i_{\alpha x}( \bar{W}\uline{G},M) 
\]
are isomorphisms. Proposition \ref{settotop} shows that $\bar{W}\uline{F}$ and   $\bar{W}\uline{G}$ are formally quasi-smooth   homogeneous functors.
Since $F \leadsto \bar{W}\uline{F}$ does not change $\pi^0F$, Lemma \ref{totcohoc} and Corollary \ref{detectweak} imply that the map
\[
   \bar{W}\uline{F} \to    \bar{W}\uline{G}                                                                                                                                                                                                                                                  \]
is a weak equivalence.

Since $G$ is  also homogeneous and formally quasi-smooth, Corollary \ref{detectweak} gives a weak equivalence $G \to \bar{W}\uline{G}$.
 Combining this with the weak equivalence above, we see that $\bar{W}\uline{F} $ and $G $ are canonically weakly equivalent.   
\end{proof}

\begin{remark}\label{predetectweakh}
 By replacing   Proposition \ref{robs} with \cite{drep} Proposition \ref{drep-obs},  the proof of Corollary \ref{predetectweak} works just as well if $F$ is homotopy-homogeneous and homotopy-surjecting, while $G$ is homotopy-homogeneous and homotopy-preserving. In particular, this holds if $G$ is any presentation of a derived geometric $n$-stack.   
\end{remark}


\begin{theorem}\label{mylurieprerep}
Take a functor $F: d\cN_R^{\flat} \to \bS$ satisfying the following conditions.

\begin{enumerate}
 
\item $F$ is 
pre-homotopic. 

\item $F$ is formally quasi-presmooth.

\item For all discrete rings $A$, $F(A)$ is $n$-truncated, i.e. $\pi_iF(A)=0$ for all $i>n$ .

\item
$F$ is homogeneous.

\item $\pi^0F:\Alg_{\H_0R} \to \bS$ is a hypersheaf for the \'etale topology.

\item
$\pi_0\pi^0F:  \Alg_{\H_0R} \to \Set$  preserves filtered colimits.

\item
For all $A \in \Alg_{\H_0R}$ and all $x \in F(A)$, the functors $\pi_i(\pi^0F,x): \Alg_A \to \Set$  preserve filtered colimits for all $i>0$.

\item 
for all finitely generated integral domains $A \in \Alg_{\H_0R}$, all $x \in F(A)_0$ and all \'etale morphisms $f:A \to A'$, the maps 
\[
\DD_x^*(F, A)\ten_AA' \to \DD_{fx}^*(F, A')
\]
are isomorphisms.

\item
 for all finitely generated $A \in \Alg_{\H_0R}$  and all $x \in F(A)_0$, the functors $\DD^i_x(F, -): \Mod_A \to \Ab$ preserve filtered colimits for all $i>0$.

\item for all finitely generated integral domains $A \in \Alg_{\H_0R}$  and all $x \in F(A)_0$, the groups $\DD^i_x(F, A)$ are all  finitely generated $A$-modules.

\item for all complete discrete local Noetherian  $\H_0R$-algebras $A$, with maximal ideal $\m$, the map
$$
\pi^0F(A) \to {\Lim} F(A/\m^r)
$$
is a weak equivalence. 
\end{enumerate}
 Then $\bar{W}\uline{F}$ is the restriction to $d\cN^{\flat}_R$ of  an almost finitely presented geometric derived $n$-stack $F':s\Alg_R \to \bS$ (resp. $F':dg_+\Alg_R \to \bS$). Moreover, $F'$ is uniquely determined by $F$ (up to weak equivalence).
\end{theorem}
\begin{proof}
 This essentially  appears as \cite{drep} Theorem \ref{drep-lurieprerep}, which takes a homotopy-surjecting, 
homotopy-homogeneous functor instead of a formally quasi-smooth pre-homotopic 
homogeneous functor. However,  
 every pre-homotopic functor is automatically homotopy-surjecting,
while 
every formally quasi-presmooth homogeneous functor is homotopy-homogeneous (by \cite{dmsch} Lemma \ref{dmsch-2hgsgood}).

Alternatively, note that Proposition \ref{settotop} ensures that $\bar{W}\uline{F}$ is homogeneous and formally quasi-smooth, so we may apply Theorem \ref{mylurierep3}.
\end{proof}

\section{Sheaves on the pro-Zariski and pro-\'etale sites}\label{sheafsn}

Our primary motivation for this section is the following. In general, an infinite  direct sum $M=\bigoplus_i M_i$ of locally free $A$-modules is not locally free for the \'etale topology, in the sense that there need not exist any faithfully flat \'etale morphism $A \to A'$ with $M\ten_AA'$ free. However, for all maximal ideals $\m$ of  $A$, the $A_{\m}$-module $M\ten_AA_{\fm}$ is free. Indeed, for any set $S$ of maximal ideals, the $\prod_{\m \in S} A_{\m}$-module $M\ten_A\prod_{\m \in S}A_{\fm}$ is free. As we will show below, this amounts to saying that $M$ is locally free for the pro-Zariski topology, and hence for the pro-\'etale topology.

\begin{definition}
 A presheaf $\sF:\Alg_R \to \Set$ is said to be locally of finite presentation if for any filtered direct system $\{A_i\}_i$, the map
\[
\LLim_i \sF(A_i) \to \sF(\LLim_i A_i)        
\]
is an isomorphism.      
\end{definition}

\begin{definition}
 Given a property $\oP$ of morphisms of affine schemes, we say that $f: X \to Y$ is pro-$\oP$ if it can be expressed as the limit $X = \Lim_i X_i$ of a filtered inverse system $\{X_i\}_i$ of $\oP$-morphisms   $X_i\to Y$, in which all structure maps $X_i \to X_j$ are   $\oP$-morphisms. Likewise, we say that a map $A \to B$ of rings is ind-$\oP$ if $\Spec B \to \Spec A$ is pro-$\oP$.
\end{definition}

\begin{lemma}\label{fpsheafpro}
 If $\sF:\Alg_R \to \Set$ is locally of finite presentation and a sheaf for a class $\oP$ of covering morphisms, then $\sF$ is also a sheaf for the class $\pro(\oP)$.        
\end{lemma}
\begin{proof}
Given any finite (possibly empty) set $\{A_s\}_{s\in S}$ of objects of $\Alg_R $, we automatically have an isomorphism
\[
 \sF(\prod_{s\in S} A_s) \to \prod_{s\in S} \sF(A_s),      
\]
 so we need only check that for any ring homomorphism $A \to B$ in $\ind(\oP)$, the diagram
\[
 \sF(A) \to \sF(B) \abuts \sF(B\ten_AB)       
\]
is an equaliser diagram.

Now, we can express $A \to B$ as a direct limit $B = \LLim_i B_i$ of $\oP$-morphisms $A \to B_i$, so
\[
 \sF(B) \cong \LLim_i \sF(B_i), \quad    \sF(B\ten_AB)\cong    \LLim_i\sF(B_i\ten_AB_i), 
\]
$\sF$ being locally of finite presentation. Since $\sF$ is a $\oP$-sheaf, the diagram
\[
 \sF(A) \to \sF(B_i) \abuts \sF(B_i\ten_AB_i)       
\]
is an equaliser, and the required result now follows from the observation that finite limits commute with filtered direct limits.      
\end{proof}

We will now construct weak universal covers for the topologies which concern us.
\subsection{The pro-Zariski topology}

\begin{definition}
 A morphism $A \to B$ of commutative rings is said to be conservative if the map
\[
 A^{\by} \to A\by_BB^{\by}        
\]
is an isomorphism, where $A^{\by}$ denotes units in $A$. Say that a morphism $\Spec B \to \Spec A$ of affine schemes is conservative if $A\to B$ is so.  
\end{definition}

\begin{definition}
 Say that a map $A \to B$ of commutative rings is a localisation if $B \cong A[S^{-1}]$, for some subset $S \subset A$.       
\end{definition}
Note that $\Spec D \to \Spec C$ is an open immersion if and only if $ D \cong C[S^{-1}]$ for some finite set $S$. Thus $A \to B$ is a localisation if and only if $\Spec B \to \Spec A$ is a pro-(open immersion).

\begin{lemma}\label{UFopen}
 Any commutative ring homomorphism $f:A \to B$ has a unique factorisation $A \to C \to B$ as a localisation followed by a conservative map.
\end{lemma}
\begin{proof}
 This is \cite{anel} Proposition 52. The factorisation is given by setting $S:= \{a \in A\,:\, f(a) \in B^{\by}\}$, then letting $C:= A[S^{-1}]$.
\end{proof}

In order to study the Zariski topology, we wish to use local isomorphisms rather than open immersions. Likewise, for the pro-Zariski topology, we want pro-(local isomorphisms) rather than pro-(open immersions).

\begin{definition}
 A morphism $A \to B$ of commutative rings is said to be strongly conservative if it is conservative, and 
the map
$
 \id(A)\to \id(B)        
$
on sets of idempotent elements
is an isomorphism. Say that a morphism $\Spec B \to \Spec A$ of affine schemes is strongly conservative if $A\to B$ is so.  
\end{definition}

\begin{remark}
 The set $\id(A)$  just consists of  ring homomorphisms $\Z^2 \to A$. If $A$ is finitely generated, then $\Spec A$  has a finite set $\pi(\Spec A)$ of components. Since an arbitrary ring $A$ can be expressed as a filtered colimit $A = \LLim_i A_i$ of finitely generated rings, we can then define $\pi(A)$ to be the profinite set $\LLim_i \pi(\Spec A_i)$. Thus a conservative morphism  $\Spec B \to \Spec A$ is strongly conservative if and only if $\pi(\Spec B) \to \pi(\Spec A)$ is an isomorphism of profinite sets. 
\end{remark}

\begin{lemma}\label{UFZar}
 Every morphism $f:X \to Y$ of affine schemes has a unique factorisation $X \to (X/Y)^{\loc} \to Y$ as  a strongly conservative map followed by a pro-(local isomorphism).
\end{lemma}
\begin{proof}
 This is remarked at the end of \cite{anel} \S 4.2, where strongly conservative maps are denoted by Conv, and pro-(local isomorphisms) by Zet. Explicitly, we first factorise $f$ as $X \to Y\by_{\pi(Y)}\pi(X)\to Y$, and then apply Lemma \ref{UFopen} to the first map, obtaining $X \to  (X/Y)^{\loc}\to  Y\by_{\pi(Y)}\pi(X)\to Y$. If $X=\Spec B$ and $Y= \Spec A$, note that 
\[
 Y\by_{\pi(Y)}\pi(X)= \Spec (A\ten_{\Z.\id(A)}\Z.\id(B)).
\]
Note that we would get the same construction by applying  Lemma \ref{UFopen} to $ X \to Y\by\pi(X)$.
 \end{proof}

\begin{lemma}\label{univZarcover}
 For any commutative ring $A$, the category of pro-Zariski covers of $\Spec A$ has a weakly initial object $\Spec C$. In other words, for any  covering pro-(local isomorphism) $Y \to \Spec A$, there exists a map $\Spec C \to Y$ over $\Spec A$, although the map need not be unique. 

Moreover, every   covering pro-(local isomorphism) $Z \to \Spec C$ has a section.
\end{lemma}
\begin{proof}
Let $S$ be the set of maximal ideals of $ A$, and set $C:= (A/\prod_{\m \in S} (A/\m))^{\loc}$, as constructed in Lemma \ref{UFZar}. Explicitly, we first form the subring $A'$ of $A^S$ consisting of functions $f:S \to A$ with finite image. To form $C$, we then invert any element $f \in A'$ whenever for all $s \in S$, $f(s) \notin \m_s$.

Now, given any covering  pro-(local isomorphism) $\Spec B \to \Spec A$, use the covering property to lift the closed points of $A$ to closed points of $B$; this gives us a map
\[
 g:B \to \prod_{\fm} A/{\fm}.
\]
Properties of unique factorisation systems then give a unique map
\[
B \to  (A/\prod_{\m \in S} (A/\m))^{\loc}
\]
compatible with $g$.

For the second part, take a covering pro-(local isomorphism) $Z= \Spec D \to \Spec C$, and choose a lift $D \to A/\m$ of each canonical map $C \to A/\m$. This gives a diagram $A \xra{h} D \to \prod_{\m \in S} A/\m$ with $h$ opposite to a pro-(local isomorphism), so the universal property of $C$ then gives  a unique factorisation $D \to C \to \prod_{\m \in S} A/\m$. The composition $C \to D \to C$ must then be the identity, since $C \to \prod_{\m \in S} A/\m$ is strongly conservative. 
\end{proof}


\subsection{The pro-\'etale topology}

\begin{definition}
 A morphism $f:A \to B$ is said to be Henselian if any factorisation $A \to A' \to B$, with $A \to A'$ \'etale, has a section $A' \to A$ over $B$. Say that a morphism  $\Spec B \to \Spec A$ of affine schemes is Henselian if  $A \to B$ is so.
\end{definition}

\begin{lemma}\label{UFet}
 Every morphism $f:X \to Y$ of affine schemes has a unique factorisation $X \to (X/Y)^{\hen} \to Y$ as  a Henselian map followed by a pro-\'etale morphism.
\end{lemma}
\begin{proof}
 This is an immediate consequence of \cite{anel} Proposition 64, which shows that ind-\'etale morphisms and Henselian morphisms form the left and right classes of a unique factorisation system on the category of commutative  rings. Explicitly, if $Y=\Spec A$ and $X= \Spec B$, then
\[
 (A/B)^{\hen}:= \LLim A_i,
\]
where $A_i$ runs over all factorisations $A \to A_i \to B$ of $f^{\sharp}$ with $A \to A_i$ \'etale. Then $(X/Y)^{\hen}:= \Spec (A/B)^{\hen}$.
 \end{proof}

\begin{lemma}\label{univetcover}
For any commutative ring $A$, there is a weakly initial object $\Spec C$ in the category of pro-\'etale coverings of $\Spec A$. 

Moreover, every  pro-\'etale covering  $Z \to \Spec C$ has a section.
\end{lemma}
\begin{proof}
 For each point $x$ of $\Spec A$, choose a geometric point $\bar{x}$ over $x$, so $k(\bar{x})$ is a separably closed field, and let the set of all these points be $S$.
Now, use Lemma \ref{UFet} to construct the unique factorisation
\[
A \to (A/ \prod_{\bar{x} \in S} k(\bar{x}))^{\hen} \to \prod_{\bar{x} \in S} k(\bar{x})
\]
of $A \to \prod_{\bar{x} \in S} k(\bar{x}) $. 
The arguments of Lemma \ref{univZarcover} now adapt to show that $\Spec C:= [(\Spec  \prod_{\bar{x} \in S} k(\bar{x}))/\Spec A]^{\hen}$ is weakly  initial in the category of pro-\'etale coverings of $\Spec A$, and that every covering of $\Spec C$ has a section.
\end{proof}


\subsection{Sheaves on derived rings}

\begin{definition}
 Given a subclass $\oP$ of flat morphisms of commutative rings, closed under pushouts and composition, say that a morphism $f:A \to B$ in $s\Ring$ is
\begin{enumerate}
 \item homotopy-$\oP$ if $\pi_0f: \pi_0A \to \pi_0B$ is in $\oP$, and the maps
\[
 \pi_n(A) \ten_{\pi_0A}\pi_0B \to \pi_nB
\]
are isomorphisms for all $n$;

\item strictly $\oP$ if  $f_0: A_0 \to B_0$ is in $\oP$, and the maps
\[
 A_n \ten_{A_0}B_0 \to B_n
\]
are isomorphisms for all $n$.
\end{enumerate}
 \end{definition}

\begin{definition}
 Given $\oP$ as above, say that a morphism $f:A \to B$ in $dg_+\Alg_{\Q}$ is
\begin{enumerate}
 \item homotopy-$\oP$ if $\H_0f: \H_0A \to \H_0B$ is in $\oP$, and the maps
\[
 \H_n(A) \ten_{\H_0A}\H_0B \to \H_nB
\]
are isomorphisms for all $n$;

\item strictly $\oP$ if  $f_0: A_0 \to B_0$ is in $\oP$, and the maps
\[
 A_n \ten_{A_0}B_0 \to B_n
\]
are isomorphisms for all $n$.
\end{enumerate}
 \end{definition}

\begin{lemma}
 Every strictly $\oP$ morphism in $s\Ring$ or $dg_+\Alg_{\Q}$ is homotopy-$\oP$.
\end{lemma}
\begin{proof}
 We first prove this in the simplicial case. Take a strictly $\oP$ morphism $f:A \to B$; taking homotopy groups gives $\pi_n(B)\cong \pi_n(A)\ten_{A_0}B_0$, by flat base change.
We then have isomorphisms
\begin{eqnarray*}
\pi_n(B) &\cong&  \pi_n(A)\ten_{A_0}B_0\\
&\cong& \pi_n(A)\ten_{\pi_0A} (\pi_0A\ten_{A_0}B_0)\\
&\cong& \pi_n(A)\ten_{\pi_0A}\pi_0B,
 \end{eqnarray*}
as required. For the chain algebra case, replace $\pi_n$ with $\H_n$.
\end{proof}

\begin{definition}\label{strictcover}
 On $s\Ring^{\op}$ and $dg_+\Alg_{\Q}^{\op}$, we define topologies for every class $\oP$ as above by setting $\oP_c$ to be the intersection of $\oP$ with faithfully flat morphisms, and saying that $f:A \to B$ is a homotopy-$\oP$ covering (resp. a strict $\oP$ covering) if $f$ is homotopy-$\oP_c$ (resp. strictly $\oP_c$).

In this way, we define both homotopy and strict sites for the \'etale, Zariski, pro-\'etale and pro-Zariski topologies.
\end{definition}

\section{Moduli from DGLAs}\label{dglamoduli}

\subsection{DGLAs}

\begin{definition}       
A \emph{differential graded Lie algebra (DGLA)}   is a  graded $\Q$-vector space  $L=\bigoplus_{i \in \N_0} L^i$, equipped with operators $[-,-]:L \by L \ra L$ bilinear and $d:L \ra L$ linear,  satisfying:

\begin{enumerate}
\item $[L^i,L^j] \subset L^{i+j}.$

\item $[a,b]+(-1)^{\bar{a}\bar{b}}[b,a]=0$.

\item $(-1)^{\bar{c}\bar{a}}[a,[b,c]]+ (-1)^{\bar{a}\bar{b}}[b,[c,a]]+ (-1)^{\bar{b}\bar{c}}[c,[a,b]]=0$.

\item $d(L^i) \subset L^{i+1}$.

\item $d \circ d =0.$

\item $d[a,b] = [da,b] +(-1)^{\bar{a}}[a,db]$
\end{enumerate}
Here $\bar{a}$ denotes the degree of $a$, mod $ 2$, for $a$ homogeneous.
\end{definition}

\subsubsection{Maurer--Cartan}
\begin{definition}
Given a  DGLA $L^{\bt}$, define the Maurer--Cartan set by 
$$
\mc(L):= \{\omega \in  L^{1}\ \,|\, d\omega + \half[\omega,\omega]=0 \in   L^{2}\}
$$
\end{definition}
\begin{lemma}\label{obsdgla}
If a map $e:L \onto M$ of DGLAs has kernel $K$, with $[K,K]=0$, then for any $\omega \in \mc(M)$, the obstruction to lifting $\omega$ to $\mc(L)$ lies in
\[
 \H^2(K, d+[\omega,-]).
\]
\end{lemma}
\begin{proof}
This is well-known. 
Given $\omega \in \mc(M)$, choose a lift $\tilde{\omega} \in L^1$, and look at $u(\tilde{\omega}):= d\tilde{\omega}+\half[\tilde{\omega},\tilde{\omega}]$. Since $[a,[a,a]]=0$ for any $a \in L^1$, we get
\[
 du+[\tilde{\omega},u(\tilde{\omega})]= [d\tilde{\omega},\tilde{\omega}]+[\tilde{\omega},d\tilde{\omega}]=0,
\]
so $u \in \z^2( K, d+[\omega,-])$. Another choice for $\tilde{\omega}$ is of the form $\tilde{\omega}+a$, for $a \in K^1$, and then 
\[
 u((\tilde{\omega}+a)=  u((\tilde{\omega})+ da + [\tilde{\omega},a],
\]
so the obstruction  is
\[
o_e(\omega):= [u(\tilde{\omega})] \in \z^2( K, d+[\omega,-])/(d+[\omega,-])K^1=  \H^2(K, d+[\omega,-]).
\]
\end{proof}

\subsubsection{The gauge action}

\begin{definition}
 Given a DGLA $L$, we  say that a group $G_L$ is a gauge group for $L$ if it is equipped with the following extra data:
  \begin{enumerate}
        \item group homomorphisms $\ad: G_L \to \GL(L^n)$ for all $n$, and 
\item a map $D: G_L \to L^1$,
\end{enumerate}
 satisfying the following conditions for $g,h \in G_L$, $v,w \in L$:
\begin{enumerate}
 \item  $\ad_g([v,w])= [\ad_gv, \ad_gw]$,
\item $D(gh)= Dg+ \ad_g(Dh)$,
\item $d(Dg)= \half[Dg,Dg]$,
\item $d(\ad_g(v))= [Dg, \ad_g(v)]+ \ad_g(dv)$.
\end{enumerate}
\end{definition}

\begin{examples}\label{nilpdglagauge}
 If the DGLA $L$ is nilpotent, then a canonical choice for $G_L$ is the group $\exp(L^0)$, with $D(g)= (dg)\cdot g^{-1}$. 

When $L^0$ is finite-dimensional, $G_L$ will typically be an algebraic group integrating $L^0$, again with  $D(g)= (dg)\cdot g^{-1}$.       
\end{examples}

\begin{definition}\label{gauge}
Given a gauge group $G_L$ for a DGLA $L$, define the \emph{gauge action} of $G_L$ on $\mc(L)$ by
\[
 g\star\omega := \ad_g(\omega) -Dg       
\]
for $g \in G_L$ and $\omega \in \mc(L)$, noting that the conditions on $\ad_g$ and $D$ ensure that this is well-defined and a group homomorphism. 
\end{definition}

\begin{definition}\label{deldef}
 Given a DGLA $L$ with gauge group $G_L$, define the Deligne groupoid by  
$
 \Del(L):= [\mc(L)/G_L]
$
In other words, $\Del(L)$ has  objects $\mc(L)$, and morphisms from $\omega$ to $\omega'$ consist of $\{g \in G_L\,:\, g\star\omega=\omega'\}$. 

Define  $\ddel(L)\in \bS$ to be the nerve $B\Del(L) $ of $\Del(L)$.
\end{definition}

\subsection{Moduli of pointed finite schemes}

For a fixed $r \in \N$, we now construct a DGLA governing moduli of pointed finite schemes of rank $r+1$. For any commutative  $\Q$-algebra $A$, our moduli groupoid consists of  non-unital commutative $A$-algebras $B$, with the $A$-module underlying $B$ being locally free of rank $r$.
Our approach is analogous to the treatment  of  finite subschemes in \cite{Hilb} \S 3.

\begin{definition}
Given a graded vector space $V$ over $\Q$, let $CL(V)$ be the free (ind-conilpotent) graded Lie coalgebra $\bigoplus_{n \ge 1}CL_n(V)$   cogenerated by $V$. Note that $CL_n(V) $ is a quotient of $V^{\ten n}$ by graded shuffle permutations.
 \end{definition}

\begin{definition}
 Given a graded-commutative chain algebra $A$, define $\beta(A)$ to be the dg Lie coalgebra $  CL(A[-1])$, with coderivation $d_C$ given on cogenerators   by
$$
d_C( a_1\ten a_2\ldots \ten a_n) = \left\{\begin{matrix} da_1 & n=1 \\ a_1a_2 & n=2 \\ 0 & n>2. \end{matrix} \right.
$$
\end{definition}

\begin{definition}\label{findgla}
 Define a DGLA $L$ by
\[
 L^n:= \Hom_{\Q}(CL_{n+1}(\Q^r[-1]), \Q^r[-1]);       
\]
this can be identified with the space of degree $-n$ Lie coalgebra derivations of $\beta(\Q^r)$, and this latter description allows us to define differential and bracket as
\[
 d_L(f)= d_{\beta}\circ f \pm f\circ d_{\beta} \quad [f,g]= f\circ g \mp g \circ f.       
\]
 
Define a gauge group for $L$ by setting $G_L= \GL(\Q^r)= \GL_r(\Q)$. This has a canonical action on $\beta(\Q^r)$, so we set  $\ad: G_L \to \GL(L^n)$ to be the adjoint action on derivations.  Finally, $D:G_L \to L^1$ is given by $D(g)= d_{\beta} - \beta(g) \circ d_{\beta} \circ  \beta(g)^{-1}$.
\end{definition}

\begin{definition}\label{betastar}
 Given a differential graded (chain) Lie coalgebra $C$, define the graded-commutative chain algebra $\beta^*(C)$ to be the free graded-commutative algebra on generators $C[1]$, with derivation given on generators by
\[
 d_{\beta^*(C)}= d_C+ \Delta: C[1] \to C \oplus S^2(C[1])[-1],
\]
where $\Delta: C[1] \to S^2(C[1])[-1]= CL_2(C)[1]$ is the cobracket.
\end{definition}

Note that $\beta^*$ is left adjoint to the functor $\beta$ from graded-commutative chain algebras to ind-conilpotent chain  Lie coalgebras.

 \begin{lemma}\label{localfin}
If we set $G_{L\ten A}:= \GL_r(A)$, then for any commutative $\Q$-algebra 
$A$,
$\Del(L\ten A)$ is canonically isomorphic to the groupoid of non-unital -commutative $A$-algebra structures on the $A$-module $A^r$.
 \end{lemma}
 \begin{proof}
This is standard. Square-zero  $A$-linear degree $-1$ derivations on $\beta(\Q^r)\ten_{\Q}A$ are all of the form $d_{\beta(\Q^r)} +\omega$, for $\omega \in \mc(L\ten A)$. Given $g \in \GL(A)$, the derivation $ d_{\beta(\Q^r)} +g\star\omega$ is then  $\beta(g) \circ (d_{\beta(\Q^r)} +\omega) \circ \beta(g)^{-1}$. 

An element $\omega \in \mc(L\ten A)$ is just  an associative multiplication $S^2(A^r) \to A^r$, so corresponds to a non-unital commutative $A$-algebra structure.
\end{proof}
      
\begin{definition}
 Given $A \in dg_+\cN^{\flat}_{\Q}$, define $L\ten A$ to be the DGLA
\[
(L\ten A)^n:= \bigoplus_i L^{n+i}\ten A_i,                                                                      
\]
with differential $d_L \pm d_A$ and bracket given by $[v\ten a, w\ten b] = \pm [v,w]\ten (ab)$, where signs follow the usual graded conventions.
\end{definition}

\begin{definition}
 For the DGLA $L$ of Definition \ref{findgla}, 
define  the  groupoid-valued functor $\cG:dg_+\cN^{\flat}_{\Q}\to \gpd$ to be the stackification of the groupoid presheaf
\[
 A \mapsto  [\mc(L\ten A)/\GL_r(A_0)] 
\]
 in the strict Zariski topology of Definition \ref{strictcover}. 

Explicitly, objects of $\cG(A)$ are pairs $(\omega, g) \in \mc(L\ten A\ten_{A_0}B)\by \GL_r(B\ten_{A_0}B)$, for $A_0 \to B$ a faithfully flat local isomorphism (so $\Spec B \to \Spec A_0$ is an open cover), satisfying the following conditions:
\begin{enumerate}
 \item $g\star (\pr_1^*\omega) = \pr_0^*\omega \in \mc(L\ten  A\ten_{A_0}B\ten_{A_0}B)$,
\item $\pr_{02}^*g = (\pr_{01}^*g)\cdot  (\pr_{12}^*g) \in \GL_r( B\ten_{A_0}B\ten_{A_0}B)$.
\end{enumerate}

An isomorphism from $(B,\omega, g)$ to $(C,\nu, h)$ is a local isomorphism $B\ten_{A_0}C \to D$ with $A_0 \to D$  faithfully flat, together with an element $\alpha \in \GL_r(D)$ such that $\alpha\star \omega = \nu \in \mc(L\ten  A\ten_{A_0}D)$, with $(\pr_0^*\alpha)\cdot g = h\cdot (\pr_1^*\alpha) \in \GL_r(D\ten_{A_0}D)$.
\end{definition}

\begin{definition}\label{barwdef2}
As in \cite{pathgpd}, given a simplicial object $\C$ in the category of categories,  we define the simplicial set $\bar{W}\C$ by first forming the nerve $B\C$ (a bisimplicial set), then applying the functor $\bar{W}$ of Definition \ref{barwdef}, giving 
\[  
  \bar{W}\C:= \bar{W}B\C.
 \]     
Explicitly,
\begin{eqnarray*}
(\bar{W}\Gamma)_n = \{(\uline{x},\uline{g})\,:\, \uline{x}&\in& \Ob\Gamma_n \by \Ob\Gamma_{n-1}\by \ldots\by \Ob\Gamma_0,\\
\uline{g} &\in& \Gamma_{n-1}(\pd_0x_n, x_{n-1}) \by\Gamma_{n-2}(\pd_0x_{n-1}, x_{n-2})\by \ldots \by \Gamma_0(\pd_0x_1, x_0) \},
\end{eqnarray*}
with operations giving $\pd_i(x_n,\ldots ,x_0; g_{n-1},\ldots ,g_0)$ as
\begin{eqnarray*}
\left\{ \begin{matrix} 
\left(\begin{matrix}x_{n-1},\ldots, x_0  ;\\
        g_{n-2},\ldots, g_0
      \end{matrix}\right) & i=0,\\
\left(\begin{matrix} \pd_ix_n,\pd_{i-1}x_{n-1},\ldots, \pd_1x_{n-i+1}, x_{n-i-1},\ldots, x_0  ;\\
\pd_{i-1}g_{n-1},\ldots, \pd_1g_{n-i+1}, (\pd_0g_{n-i})g_{n-i-1}, g_{n-i-2}, \ldots, g_0 \end{matrix}\right) & 0<i<n,\\ 
\left(\begin{matrix}\pd_nx_n, \ldots, \pd_1x_1;\\
\pd_{n-1}g_{n-1}, \ldots, \pd_1g_1\end{matrix}\right) & i=n, 
\end{matrix} \right.\\
\end{eqnarray*}
and $\sigma_i(x_n,\ldots ,x_0; g_{n-1}, \ldots , g_0)$ as
\[
\left(\begin{matrix}\sigma_ix_n,\sigma_{i-1}x_{n-1},\ldots, \sigma_0x_{n-i}, x_{n-i}, \ldots, x_0;\\
      
  \sigma_{i-1}g_{n-1},\ldots, \sigma_0g_{n-i}, \id_{x_{n-i}} ,  g_{n-i-1}, \ldots, g_0  
      \end{matrix}\right).
\]
\end{definition}

\begin{proposition}\label{repfin}
 The functor $\bar{W}\uline{\cG}:dg_+\cN^{\flat}_{\Q}\to \bS$ is representable by an almost finitely presented derived geometric $1$-stack.
\end{proposition}
\begin{proof}
We verify the conditions of Theorem \ref{mylurieprerep} for $B\cG:dg_+\cN^{\flat}_{\Q}\to \bS $. Homogeneity follows immediately, because both $\mc(L\ten -)$ and $\GL_r$ preserve finite limits. Lemma \ref{obsdgla} implies that $\mc(L\ten -)$ is pre-homotopic, since for any tiny acyclic extension $A \to B$ with kernel $I$, it gives the obstruction space as 
\[
\H^2(L\ten I, d+[\omega,-])= \bigoplus_n \H^{2+n}( L\ten \H_0B, d+[\omega,-])\ten_{\H_0B}\H_n(I)=0.
\]
It follows immediately that $B\cG$ is pre-homotopic, and formal quasi-presmoothness is a consequence of the smoothness of $\GL_r$.
 
Now, for $A \in \Alg_{\Q}$, Lemma \ref{localfin} implies that $\cG(A)$ is equivalent to the groupoid of rank $r$ commutative algebras over $A$. This implies that $\pi^0\cG$ is a stack, so $\pi^0B\cG$ is a hypersheaf, and it also guarantees that the other conditions relating to $\cG$ hold, so we need only verify the cohomological conditions.

 For an $A$-algebra $B$  corresponding to an object $[B]$ of $\cG(A)$, the results of \cite{Hilb} \S 2 imply that
\[
 \DD^i_{[B]}(B\cG,M) \cong \Ext^{i+1}_{A \oplus B}(\bL^{A \oplus B/A}_{\bt}, M\ten_AB),
\]
 which has all the properties we require. Here, $\bL^{A \oplus B/A}$ denotes the cotangent complex (in the sense of \cite{Ill1}) of the unital algebra $A \oplus B$ over $A$. This corresponds to the cotangent complex  $\bL^{B/A}$ in the category of non-unital commutative rings, defined using the formalism of \cite{Q}.
\end{proof}

\begin{remark}
 Alternatively, we can describe the associated derived geometric $1$-stack explicitly. The functor  $A \mapsto \mc(L\ten A)$ is an affine dg scheme, and  $\bar{W}\uline{\cG}$ is just the hypersheafification of the quotient $B[\mc(L)/\GL_r]$ in the homotopy-Zariski (and indeed homotopy-\'etale) topologies. In the terminology of \cite{stacks2}, the simplicial affine dg scheme $B[\mc(L)/\GL_r]$ is a derived Artin $1$-hypergroupoid representing  $\bar{W}\uline{\cG}$.
\end{remark}

\begin{proposition}\label{finconsistent}
For $A \in dg_+\cN^{\flat}_{\Q}$, the space $\bar{W}\uline{\cG}(A)$ is functorially weakly equivalent to the nerve $\bar{W}\fG(A)$ of the $\infty$-groupoid  $\fG(A)$ of non-unital graded-commutative chain  $A$-algebras $B$ in non-negative degrees for which $B\ten_A^{\oL}\H_0A$ is weakly equivalent to a locally free module rank $r$ over $\H_0A$.
\end{proposition}
\begin{proof}
The data  $(\omega,g)\in \cG(A)$ amount to giving a locally free $A_0$-module $M$ of rank $r$ (defined by the descent datum $g$), and a closed degree $-1$ differential $\delta$ on the free chain Lie $A$-coalgebra $CL_{A_0}(M[-1])\ten_{A_0}A$.  Note that in the notation of \cite{Hilb} 3.5, $RCA(\Q^r)$ is the dg scheme representing $\mc(L\ten -)$.

The functor $\beta^*$ from Definition \ref{betastar} maps from dg Lie $A$-coalgebras to non-unital graded-commutative chain $A$-algebras, giving us a  chain algebra
\[
 \beta^*(CL_{A_0}(M[-1])\ten_{A_0}A, \delta).
\]
over $A$.
Thus we have defined a functor $\beta^*: \cG(A) \to \fG(A)$, and Lemma \ref{localfin} implies that this is a weak equivalence when $A \in \Alg_{\Q}$, so $\pi^0\cG \simeq \pi^0\fG(A)$. 

Now, for $A \in \Alg_{\Q}$, $\omega \in \mc(L\ten A)$ and an $A$-module $N$, a standard calculation gives
\[
 \DD^i_{(\omega,\id)}(B\cG,N) \cong \H^{i+1}(L\ten A, d+[\omega, ]),
\]
which by \cite{Hilb} \S 2 is just $\Ext^{i+1}_{A \oplus B}(\bL^{A \oplus B/A}, N\ten_AB)$, where $B$ is the non-unital $A$-algebra corresponding to $\omega$. By faithfully flat descent, we deduce that if an $A$-algebra $B$ is associated to  $(\omega,g)\in \cG(A)$, then
\[
  \DD^i_{(\omega,\id)}(B\cG,N) \cong \Ext^{i+1}_{A \oplus B}(\bL^{A \oplus B/A}, N\ten_AB).
\]
Adapting  \cite{dmsch} Corollary \ref{dmsch-representaffine} and Example \ref{dmsch-fineg} to non-unital algebras, 
\[
 \DD^i_{[B]}(\bar{W}\fG,N) \cong \Ext^{i+1}_{A \oplus B}(\bL^{A \oplus B/A}, N\ten_AB),
\]
so $f$ induces isomorphisms on the cohomology groups $\DD^i$. 

As \cite{dmsch}  Example \ref{dmsch-fineg} adapts to non-unital algebras, the functor $\bar{W}\fG$ is also representable by a derived geometric $1$-stack, so the weak equivalence follows from Remark \ref{predetectweakh}.
\end{proof}

\subsection{Derived moduli of polarised projective schemes}\label{polsn}

Fix a numerical polynomial $h \in \Q[t]$, with $h(i)\ge 0$ for $i\gg 0$. We will now study the moduli of polarised projective schemes $(X,\O_X(1))$ over an affine base, with $\O_X(1)$ ample, for which $\Gamma(X, \O_X(n))$ is locally free of rank $h(n)$ for $n \gg 0$. As in \cite{Mum} Lecture 7 Corollary 3, such a polynomial $h$ exists for every flat projective scheme  over a connected Noetherian base.

Note that a $\bG_m$-representation $M$ in  $A$-modules is equivalent to an $A$-linear decomposition 
\[
 M= \bigoplus_{n \in \Z} M\{n\},
\]
with $\lambda \in \bG_m(A)$ acting on $M\{n\}$ as multiplication by $\lambda^n$. The functors $\beta^*$ and $\beta$ of the previous section both extend naturally to $\bG_m$-equivariant objects.

\begin{definition}
 Given  $p\gg 0$ and $q \ge p$,
 define a  DGLA $L_{[p,q]}$ over $\Q$ by
\[
 L^n_{[p,q]}:= \Hom_{\Q}^{\bG_m}(CL_{n+1}( \bigoplus_{q \ge r\ge p} \Q^{h(r)}\{r\}  [-1]),  \bigoplus_{q \ge r\ge p} \Q^{h(r)}\{r\} [-1]);       
\]
this can be identified with the space of $\bG_m$-equivariant degree $-n$ Lie coalgebra derivations of $ \beta(\bigoplus_{q \ge r\ge p} \Q^{h(r)}\{r\})$, and this latter description allows us to define differential and bracket as
\[
 d_{L_{[p,q]}}(f)= d_{\beta}\circ f \pm f\circ d_{\beta} \quad [f,g]= f\circ g \mp g \circ f.       
\]
\end{definition}

\begin{definition}\label{poldgla}
Given  $p\gg 0$, define the pro-DGLA $L_p$ over $\Q$  to be the inverse system $L_p = \{L_{[p,q]}\}_q$, so the underlying DGLA is $\Lim_q L_{[p,q]}$, given by
\[
 L^n_p= \Hom_{\Q}^{\bG_m}(CL_{n+1}( \bigoplus_{ r\ge p} \Q^{h(r)}\{r\}  [-1]),  \bigoplus_{r\ge p} \Q^{h(r)}\{r\} [-1]),      
\]
which can be identified with the space of $\bG_m$-equivariant degree $-n$ Lie coalgebra derivations of $ \beta(\bigoplus_{r\ge p} \Q^{h(r)}\{r\})$, the latter regarded as the colimit $\LLim_q \beta(\bigoplus_{q \ge r\ge p} \Q^{h(r)}\{r\})$. 

Given a $\Q$-vector space $V$, we then define $L_p\hat{\ten}V$ to be the completed tensor product 
\[
L_p\hat{\ten}V:= \Lim_q (L_{[p,q]}\ten V),        
\]
so
\[
 (L_p\hat{\ten}V)^n:= \Hom_{\Q}^{\bG_m}(CL_{n+1}( \bigoplus_{ r\ge p} \Q^{h(r)}\{r\}  [-1]),  \bigoplus_{r\ge p} V^{h(r)}\{r\} [-1]).      
\]
\end{definition}

\begin{definition}
 Given $A \in \Alg_{\Q}$, define a gauge group for $L_p\hat{\ten}A$ by setting $G_{L_p}(A):= \prod_{r \ge p}\GL_{h(r)}(A)$. This has a canonical action on $\beta(\bigoplus_{r\ge p} \Q^{h(r)}(r))\ten_{\Q}A$, so we set  $\ad: G_{L_p} \to \GL(L_p^n)$ to be the adjoint action on derivations.  Finally, $D:G_{L_p} \to L^1$ is given by $D(g)= d_{\beta} - \beta(g) \circ d_{\beta} \circ  \beta(g)^{-1}$.    
\end{definition}

\begin{lemma}\label{localpol}
 For any commutative $\Q$-algebra $A$, the groupoid $[\mc(L_p\hat{\ten} A)/G_{L_p}(A)]$ is naturally equivalent to the groupoid of $\bG_m$-equivariant non-unital commutative $A$-algebra structures on
\[
 \bigoplus_{r\ge p} A^{h(r)}\{r\}.
\]
\end{lemma}
\begin{proof}
 This is just a  graded version of Lemma \ref{localfin}.
\end{proof}

\begin{definition}
 Given $A \in dg_+\cN^{\flat}_{\Q}$, define $L_p(A)$ to be the DGLA
\[
L_p(A)^n:= \bigoplus_i L^{n+i}_p\hat{\ten}A_i                                                                      
\]
with differential $d_L \pm d_A$ and bracket given by $[v\ten a, w\ten b] = \pm [v,w]\ten (ab)$, where signs follow the usual graded conventions.
\end{definition}

\begin{definition}
 For the DGLA $L_p$ of Definition \ref{poldgla}, 
define  the  groupoid-valued functor $\cG_p:dg_+\cN^{\flat}_{\Q}\to \gpd$ to be the stackification of the groupoid presheaf
\[
 A \mapsto  [\mc(L_p(A))/G_{L_p}(A_0)] 
\]
 in the strict pro-Zariski topology of Definition \ref{strictcover}. 

Making use of Lemma \ref{univZarcover}, we can describe this explicitly by setting 
$A'_0:=(A_0/\prod_{\m} A_0/\m)^{\loc}$, where $\m$ runs over all maximal ideals of $A_0$, and then setting $A':= A_0'\ten_{A_0}A$. Objects of $\cG_p(A)$ are then pairs $(\omega, g) \in \mc(L_p(A'))\by G_{L_p}(A'_0\ten_{A_0}A'_0)$, satisfying the following conditions:
\begin{enumerate}
 \item $g\star (\pr_1^*\omega) = (\pr_0^*\omega) \in \mc(L(A'\ten_{A}A'))$,
\item $\pr_{02}^*g = (\pr_{01}^*g)\cdot  (\pr_{12}^*g) \in G_{L_p}(A'_0\ten_{A_0}A'_0\ten_{A_0}A'_0)$.
\end{enumerate}

An isomorphism from $(\omega_1, g_1)$ to $(\omega_2, g_2)$ is  an element $\alpha \in G_{L_p}(A'_0)$ such that $\alpha\star \omega_1 = \omega_2 \in \mc(L_p(A'))$, with $(\pr_0^*\alpha)\cdot g_1 = g_2\cdot (\pr_1^*\alpha) \in G_{L_p}(A'\ten_AA')$
\end{definition}

\begin{lemma}\label{globalpol}
 For $A \in \Alg_{\Q}$, the groupoid  $\cG_p(A)$ is canonically equivalent to the groupoid of non-unital $\bG_m$-equivariant commutative $A$-algebras
\[
 B= \bigoplus_{r \ge p} B\{r\},
\]
with each $A$-module $B\{r\}$ locally free of rank $h(r)$.
\end{lemma}
\begin{proof}
 If we set $A':= (A/\prod_{\m} A/\m)^{\loc}$, where $\m$ runs over all maximal ideals of $A$, then Lemma \ref{univZarcover} shows that any Zariski cover $\Spec B \to \Spec A'$ must have a section. Hence locally free $A'$-modules are free. Lemma \ref{localpol} then implies that $[\mc(L_p(A'))/G_{L_p}(A')]$ is equivalent to the groupoid of non-unital $\bG_m$-equivariant commutative $A'$-algebras
$
 B'= \bigoplus_{r \ge p} B'\{r\},
$
with each $A'$-module $B'\{r\}$ locally free of rank $h(r)$.

Given an object $(\omega,g)$ of $\cG_p(A)$, it thus follows that $\omega$ corresponds to such an $A'$-algebra $B'$, while  
$g$ is a descent datum. This determines a unique $A$-algebra $B$ with $B'= B\ten_AA'$, and isomorphisms behave as required.
\end{proof}

\begin{definition}
For $A \in \Alg_{\Q}$, define $\cM_p(A)$ to be the full subgroupoid of $\cG_p(A)$ whose objects correspond under Lemma \ref{globalpol} to finitely generated commutative $A$-algebras.
\end{definition} 

\begin{lemma}\label{fgfet}
 The morphism $\cM_p \to \pi^0\cG_p$ of groupoid-valued functors on $\Alg_{\Q}$ is formally \'etale.
\end{lemma}
\begin{proof}
 For any square-zero extension $A \onto B$ of commutative $\Q$-algebras, we need to show that
\[
 \cM_p(A)\to \cM_p(B)\by_{\cG_p(B)}\cG_p(A)
\]
is an isomorphism. This follows because any flat $A$-algebra $C$ is finitely generated if and only if $C\ten_AB$ is finitely generated as a $B$-algebra, since any lift of a generating set for $ C\ten_AB$ must give a generating set for $C$.
\end{proof}

\begin{lemma}\label{fglfp}
 The functor $\cM_p:\Alg_{\Q}\to \gpd $ is locally of finite presentation, in the sense that for any filtered direct system $\{A_i\}_i$ of commutative $\Q$-algebras with $A=\LLim_iA_i $, the map
\[
\LLim_i \cM_p(A_i)\to \cM_p(A)       
\]
is an equivalence of groupoids.        
\end{lemma}
\begin{proof}
 We first show essential surjectivity. Take an object $B \in \cM_p(A)$. Since $B$ is finitely generated, we can choose homogeneous generators $x_j$ of degree $d_j$ for $1\le j\le n$, giving us a surjection
\[
f:A[x_1, \ldots, x_n] \onto B.       
\]

If we let $S:= \Z[x_1, \ldots, x_n] $, then  $I:= \ker f$ is a graded ideal of $S\ten A$. In the notation of \cite{HaimanSturmfels}, we have a degree functor $\deg: \N^n \to \N$ given by $(a_1, \ldots, a_n) \mapsto \sum_i a_id_i$, and the Hilbert polynomial is given by  $h_I=h$. By  \cite{HaimanSturmfels} Corollary 1.2, there is a projective scheme $H_S^h$ over $\Z$ with $H_S^h(A)$ the set of all graded ideals of $S$ with Hilbert function
\[
h_p(i):= \left\{ \begin{matrix} h(i) & i \ge p \\  0 & i< p.\end{matrix}\right.      
\]
 [Note that we do not use Grothendieck's  construction from \cite{GroHilb}, since that only  describes $A$-valued points of the Hilbert scheme for $A$ Noetherian.] 

In particular, $H_S^{h_p}$ is of finite presentation, so $H_S^{h_p}(A)= \LLim_i H_S^{h_p}(A_i)$. Therefore there exists $B_i \in H_S^{h_p}(A_i)$ with $B\cong B_i\ten_{A_i}A$. The forgetful functor $H_S^{h_p} \to \cM_p$ then ensures that $B_i \in \cM_p(A_i)$.

It only remains to show that     $\LLim_i \cM_p(A_i)\to \cM_p(A)$ is full and faithful. Now, \cite{HaimanSturmfels} Proposition 3.2 shows that the ideal $I$ above is finitely generated, so $B$ is finitely presented over $A$. Likewise, any objects $B_i, B_i' \in \cM_p(A_i)$ will be finitely presented, which implies that
\[
 \Hom_{\cM_p(A)}(B_i\ten_{A_i}A, B_i'\ten_{A_i}A)\cong \LLim_j \Hom_{\cM_p(A_j)}(B_i\ten_{A_i}A_j, B_i'\ten_{A_i}A_j),
\]
 completing the proof.
\end{proof}

\begin{definition}\label{cGpoldef}
 Define $\cM: \Alg_{\Q} \to \gpd$ by $\cM(A):= \LLim_p \cM_p(A)$. Likewise, define 
\[
 \cG:= \LLim_p \cG_p: dg_+\cN^{\flat}_{\Q} \to \gpd
\]
 and 
$\tilde{\cM}: dg_+\cN^{\flat}_{\Q}\to \gpd$ by 
\[
 \tilde{\cM}(A):= \cG(A)\by_{\cG(\H_0A)}\cM(\H_0A).
\]
\end{definition}

%
\begin{proposition}\label{gradedpol}
For $A \in \Alg_{\Q}$, $\cM(A)$ is equivalent to the groupoid of flat 
polarised schemes $(X,\O_X(1))$ of finite type over $A$, with $\O_X(1)$ ample and the $A$-modules $\Gamma(X, \O_X(n))$ locally free of rank $h(n)$ for all $n\gg 0$. 
\end{proposition}
\begin{proof}
This is fairly standard --- the analogue for subschemes is \cite{HaimanSturmfels} Lemma 4.1. Given an object $B \in \cM(A)$, there exists $p$ with $B$ lifting to $B \in \cM_p(A)$. Therefore we can define
\[
 (X,\O_X(1) ) := \Proj(A \oplus B).
\]
Replacing  $B$ with its image in $\cM_{q}(A)$  (for $q>p$) does not affect $\Proj(A \oplus B)$, so we have a functor $\Proj(A \oplus -)$ from
$
  \cM(A)  
$
to polarised projective schemes over $A$.

For the quasi-inverse functor, take a polarised scheme $(X, \O_X(1))$ and some $p$ for which $\Gamma(X, \O_X(n))$ is locally free of rank $h(n)$ for all $n\ge p$. Then define $B \in \cM_p(A)$ by
\[
 B:= \bigoplus_{n \ge p} \Gamma(X, \O_X(n)).
\]
\end{proof}

\begin{remark}
 Note that the hypothesis that $\Gamma(X, \O_X(n))$ be locally free for $n \gg 0$ ensures that $X$ is flat over $A$.  If $A$ is Noetherian, then the proof of \cite{Ha} III.9.9 shows that the converse holds, and indeed that if $A$ is connected, then  there exists a Hilbert polynomial $h$ with  $\Gamma(X, \O_X(n))$ locally free of rank $h(n)$ for all $n\gg 0$.
\end{remark}

\begin{proposition}\label{gradedcot}
If $A \in \Alg_{\Q}$ and  $X= \Proj(A\oplus C)$ for $C \in \cM_p(A)$, then for any $A$-module $M$, there are canonical isomorphisms
\[
 \DD^i_{[C]}(B\tilde{\cM},M)\cong \EExt^{i+1}_{X}(\bL^{X/B\bG_m\ten A}, \O_{X}\ten_AM),
\]
where $\bL$ is the cotangent complex of \cite{Ill1}.
\end{proposition}
\begin{proof}
Given $C \in \cM_{p_0}(A)$, first let $\tilde{X}:= \Spec (A\oplus C) -\{0\}$, where $\{0\}$ denotes the copy of $\Spec A$ defined by the ideal $C$. $\tilde{X}$ inherits a $\bG_m$-action from $C$ (with trivial action on $A$), and in fact 
\[
 \tilde{X}= \oSpec_X (\bigoplus_{n \in \Z} \O_X(n)),
\]
with $X=\tilde{X}/\bG_m$ and $\tilde{X}= X\by_{B\bG_m\ten A}^h\Spec A$. Writing $\pi:\tilde{X} \to X$ for the projection, base change gives $\pi^*\bL^{X/B\bG_m\ten A} \simeq  \bL^{\tilde{X}/A}$. Since $j: \tilde{X} \to \Spec (A\oplus C)$ is an open immersion, it is \'etale, so $\bL^{\tilde{X}/A} \simeq j^*\bL^{(A\oplus C)/A} $. 

For any $C$-module $N$, the associated quasi-coherent sheaf $N^{\sharp}$ on $X$ is given by $N^{\sharp}=(\pi_*j^*N)^{\bG_m} $. Now, Lemma \ref{fglfp} implies that there exists a finitely generated $\Q$-subalgebra $A^0 \subset A$ and an object $C^0 \in \cM_{p_0}(A^0)$ with $C= C^0\ten_{A^0}A$. Since both $\DD^i$ and $\EExt$ are compatible with base change (the former by \cite{drep} Lemma \ref{drep-tantrans}), it suffices to show that 
\[
 \DD^i_{[C^0]}(B\tilde{\cM},M)\cong \EExt^{i+1}_{X_0}(\bL^{X_0/B\bG_m\ten A^0}, \O_{X_0}\ten_{A^0}M),
\]
where $X_0= \Proj(A^0\oplus C^0)$. Replacing $A$ and $C$ with $A^0$ and $C^0$, we may therefore reduce to the case where $A$ is a finitely generated $\Q$-algebra (hence Noetherian). Because both expressions above commute with filtered colimits of the modules $M$, we may assume that $M$ is a finitely generated $A$-module.
 
Since $\cM \to \pi^0\cG$ is formally \'etale by Lemma \ref{fgfet},
\[
 \DD^i_{[C]}(B\tilde{\cM},M) \cong\DD^i_{[C]}(B\cG,M).
\]
As $A\oplus \beta^*\beta(B)$ is a cofibrant resolution of $A\oplus B$, we have 
\[
  \DD^i_{[C]}(B\cG,M) \cong \LLim_{p\ge p_0} \Ext^{i+1}_{A \oplus B\{\ge p\}}( \bL^{(A\oplus B\{\ge p\})/A}, B\{\ge p\}\ten_AM)^{\bG_m}.
\]

Now, the proof of Serre's Theorem (\cite{Se2} \S 59)  still works over any Noetherian base,  so shows that for a finitely generated $C$-module $N$ and any $n \in \Z$,
\[
 \Ext^i_X(\O_X(n), N^{\sharp}) \cong \LLim_p \Ext^i_C( C(n)\{\ge p\}, N\{\ge p\})^{\bG_m}.
\]
Indeed, a spectral sequence argument shows that the same is true if we replace $C(n)$  with any  finitely generated $C$-module $L$, since $L$ will then admit a resolution by finite sums of $C(n)$'s. In fact, another   spectral sequence argument allows us to take a chain complex $L$  whose homology groups $\H_i(L)$ are finite and  bounded below, giving
\[
  \EExt^i_X(L^{\sharp}, N^{\sharp}) \cong \LLim_p \EExt^i_C( L\{\ge p\}, N\{\ge p\})^{\bG_m}.
\]
 
Thus for any fixed $p \ge p_0$,
\[
 \Ext^i_X(\bL^{\tilde{X}/B\bG_m\ten A}, \O_X\ten_AM)  \cong \LLim_{q \ge p} \Ext^i_{A \oplus B\{\ge p\}} ( (\bL^{(A\oplus B\{\ge p\})/A})\{\ge q\}, B\{\ge p\}\ten_AM)^{\bG_m}.
\]
We can then take the colimit over the poset of pairs $(p,q)$ with $q \ge p\ge p_0$. Since the set of pairs $(p,p)$ is cofinal in this poset, we get
\[
 \Ext^i_X(\bL^{\tilde{X}/B\bG_m\ten A}, \O_X\ten_AM)  \cong \LLim_{p\ge p_0} \Ext^{i}_{A \oplus B\{\ge p\}}( \bL^{(A\oplus B\{\ge p\})/A}, B\{\ge p\}\ten_AM)^{\bG_m},
\]
  as required.
\end{proof}

\begin{proposition}\label{reppol}
The functor $B\tilde{\cM}: dg_+\cN^{\flat}_{\Q}\to \bS$ satisfies the conditions of Theorem \ref{mylurieprerep}, so
the associated functor
$\bar{W}\uline{\tilde{\cM}}:dg_+\cN^{\flat}_{\Q}\to \bS $
is representable by an almost finitely presented derived geometric $1$-stack.
\end{proposition}
\begin{proof}
We apply Theorem \ref{mylurieprerep} to $\tilde{\cM}$. First, note that $\pi^0\tilde{\cM}= \cM$, which is a stack, locally of finite presentation by Lemma \ref{fglfp}. Proposition \ref{gradedpol} and Grothendieck's formal existence theorem (\cite{EGA3.1} 5.4.5) ensure that for any  complete local Noetherian $\Q$-algebra $\L$, the map
\[                                                                                                                        \cM(\L) \to \Lim_n\cM(\L/\m^n)                                                                                           \]
 is surjective on objects. That the map is an isomorphism is then a consequence of (\cite{EGA3.1} 5.1.4).

Now, homogeneity of $B\tilde{\cM}$ is immediate, and Lemma \ref{obsdgla} gives pre-homotopicity. All the remaining conditions follow from Proposition \ref{gradedcot}, with the same reasoning as for Proposition \ref{repfin}.
\end{proof}

\begin{proposition}\label{polconsistent}
For $A \in dg_+\cN^{\flat}_{\Q}$, the space $\bar{W}\uline{\tilde{\cM}}(A)$ is functorially weakly equivalent to the nerve $\bar{W}\fM(A)$ of the $\infty$-groupoid  $\fM(A)$ of derived geometric $0$-stacks $\fX$ over $B\bG_m\by \Spec A$ for which $X:=\fX\ten^{\oL}_A\H_0A$ is weakly equivalent to a flat projective scheme over $\H_0A$, with the polarisation $X \to B\bG_m\ten \H_0A$ ample with Hilbert polynomial $h$.
\end{proposition}
\begin{proof}
We adapt the proof of Proposition \ref{finconsistent}. An object  of $\cG_p(A)$ corresponds to a locally free $\bG_m$-equivariant $A_0$-module $N\{\ge p\}$, with $N\{r\}$ locally free of rank $h(r)$, together with a  closed degree $-1$ differential $\delta$ on the free chain Lie coalgebra $CL_{A_0}(N[-1])\ten_{A_0}A$. We may therefore form the DG-scheme 
\[
\fX:= \Proj(A\oplus \beta^* (CL_{A_0}(N[-1])\ten_{A_0}A, \delta)).      
\]

As in \cite{stacks2} \S \ref{stacks-dgstacks}, there is a canonical derived geometric $1$-stack associated to  $\fX$. To give an explicit map from this to $B\bG_m$, we first let $\tilde{\fX}:= \Spec( A\oplus \beta^* (CL_{A_0}(N[-1])\ten_{A_0}A, \delta))-\{0\}$, and then form the simplicial scheme
\[
 \tilde{\fX}\by^{\bG_m}E\bG_m,       
\]
 which is a simplicial resolution of $X$, and has a canonical map to the simplicial scheme $B\bG_m$. Here, $E\bG_m$ is the universal $\bG_m$-space over $B\bG_m$, given by the simplicial $0$-coskeleton $E\bG_m= \cosk_0\bG_m$, so $(E\bG_m)_n= (\bG_m)^{n+1}$. 
For an explicit Artin hypergroupoid representation of $X$, we could go further and replace $\tilde{\fX}$ with its \v Cech nerve associated to  any open affine cover.

Now, if our object $C$ lies in $\tilde{\cM}_p\subset\cG_p(A)$, then $C\ten_A\H_0A $ lies in $\cM_p(\H_0A)$, so $C\ten_A\H_0A = \beta(B)$, for a finitely generated commutative algebra structure $B$ on $N$. Since the map
\[
 \beta^*\beta(B) \to B       
\]
is a quasi-isomorphism, this means that 
\[
  \fX \ten_A^{\oL}\H_0A \simeq \Proj(\H_0A\oplus B),    
\]
which is  a polarised projective scheme with Hilbert polynomial $h$.

Since $\Proj$ is unchanged on replacing $N$ with $N\{\ge q\}$ for $q>p$, we have defined a functor
\[
 \alpha_A:\tilde{\cM}(A) \to \fM(A).       
\]
By \cite{dmsch}  Example \ref{dmsch-modpolsch}, the functor $\bar{W}\fM$ is also representable by a derived geometric $1$-stack, so
we just need to check that $\bar{W}\tilde{\cM}\to \bar{W}\fM$  satisfies the conditions of Remark \ref{predetectweakh}.  

 If $A \in \Alg_{\Q}$, then Proposition \ref{gradedpol} implies that  $\alpha_A$ is an equivalence of groupoids. 
Combining Proposition \ref{gradedcot} with  \cite{dmsch} Corollary \ref{dmsch-representdaffine} and Example \ref{dmsch-modpolsch}, we have isomorphisms
\[
\DD^i_{[C]}(B\tilde{\cM},M)  \cong \Ext^{i+1}_{X}(\bL^{X/\bG_m\ten A}, \O_X\ten_AM) \cong   \DD^i_{[C]}(\bar{W}\tilde{\fM},M),    
\]
so Remark \ref{predetectweakh} applies.
\end{proof}

\begin{remark}\label{stabilise}
Replacing the DGLA $L_p$ with the finite-dimensional DGLA
$L_{[p,q]}$ in the definitions above gives us a functor $\tilde{\cM}_{[p,q]}$.
Since $L_p = \Lim_q L_{[p,q]}$, we will have $ \tilde{\cM}_p = \Lim_q \tilde{\cM}_{[p,q]}$, and hence
\[
 \tilde{\cM}= \LLim_p \Lim_q \tilde{\cM}_{[p,q]}.     
\]
 
It is natural to seek an open substack of $\cM$ on which these limits stabilise. If we define $\cM^{(k)} \subset \cM$ to be the open substack consisting of polarised schemes $(X, \O_X(1))$ for which $\O_X(k)$ is very ample, then we may regard $X$ as a subscheme of $\bP^{h(k)}$, so  \cite{Quot} Theorem 1.2.3(b) and Theorem 1.4.1 imply that for $q \gg p \gg 0$, the maps
\[
   \cM^{(k)} \la  \cM^{(k)}_p \to \cM^{(k)}_{[p,q]}    
\]
are equivalences of underived stacks.

Moreover, for fixed $i$,  \cite{Hilb} Theorem 4.1.1 implies that for $q \gg p \gg 0$, the maps
\[
 \DD^i_{[C]}(B\tilde{\cM}^{(k)},M)   \la  \DD^i_{[C]}(B\tilde{\cM}^{(k)}_p,M) \to \DD^i_{[C]}(B\tilde{\cM}^{(k)}_{[p,q]},M)    
\]
are isomorphisms for all $[C]$ and $M$. This does not give  give suitable $p,q$ for all $i$ simultaneously. 

However, if we restrict further to the open substack $\cM^{(k),\mathrm{LCI}} \subset \cM^{(k)}$ of local complete intersections, then the cotangent complex $\bL^{X/\bG_m\ten A}$ will be concentrated in chain degrees $[0,1]$. Thus $ \DD^i_{[C]}(B\tilde{\cM}^{(k), \mathrm{LCI}},M)=0$ for $i \notin [-1, \deg h]$, so for $q \gg p \gg 0$, we have weak equivalences
\[
  \bar{W}\uline{\tilde{\cM}}^{(k), \mathrm{LCI}}   \la  \bar{W}\uline{\tilde{\cM}}^{(k), \mathrm{LCI}}_p \to \bar{W}\uline{\tilde{\cM}}^{(k), \mathrm{LCI}}_{[p,q]}     
\]
of derived stacks, by applying Remark \ref{predetectweakh}.
\end{remark}

\section{Moduli from cosimplicial groups}\label{cgpmoduli}

Since suitable DG Lie algebras can usually only be constructed in characteristic $0$, we now work with  cosimplicial groups, which form the  first step towards a more general construction.

\subsection{Cosimplicial groups}

\begin{definition}
Let  $c\Gp$ be the category of cosimplicial groups, and $cs\Gp$ the category of cosimplicial simplicial groups.
\end{definition}

\subsubsection{Maurer--Cartan}

\begin{definition}\label{mcdefgp1}
 Define $\mc:c\Gp \to \Set$ by  
$$
\mc(G):= \z^1(G)= \{ \omega \in G^1 \,:\, \sigma^0\omega =1\,\, \pd^1 \omega = \pd^2\omega \cdot \pd^0\omega\}.
$$ 
\end{definition}

\begin{definition}\label{mcdefgp}
Define $\mmc: cs\Gp \to \bS$ by 
setting $\mmc(G) \subset \prod_{n\ge 0} (G^{n+1})^{\Delta^n}$ to consist of elements  $(\omega_n )_{n \ge 0}$ satisfying
\begin{eqnarray*}
\pd_i\omega_n &=& \left\{\begin{matrix} \pd^{i+1}\omega_{n-1}  & i>0 \\ (\pd^1\omega_{n-1})\cdot(\pd^0\omega_{n-1})^{-1} & i=0,\end{matrix} \right.\\
\sigma_i\omega_n &=& \sigma^{i+1}\omega_{n+1},\\
\sigma^0\omega_n&=& 1.
\end{eqnarray*}
Define $\mc :cs\Gp \to \Set$ by $\mc(G)=\mmc(G)_0$, noting that this agrees with Definition \ref{mcdefgp1} when $G \in c\Gp$. 
\end{definition}

\begin{remark}
Note that by the proof of \cite{htpy} Lemma \ref{htpy-maurercartan},
\[
 \mc(G)\cong \Hom_{c\bS}(\Delta, \bar{W}G),       
\]
for $\bar{W}$ as in Definition \ref{barwdef2}, where the cosimplicial simplicial set $\Delta$ is given by the $n$-simplex $\Delta^n$ in cosimplicial level $n$. Thus $\mc(G)= \Tot_0 \bar{W}G$, for 
 $\Tot:c\bS \to \bS$ the total space functor of \cite{sht} Ch. VIII, originally defined in \cite{bousfieldkan} Ch. X.

In fact, $\bar{W}$ has a left adjoint $G$ (the loop group functor), and $\mmc(G)\cong \HHom_{c\Gp}(G(\Delta), \bar{W}G)$. However, $\bar{W}$ is not simplicial right Quillen, so this does not equal $\HHom_{c\bS}(\Delta, \bar{W}G)= \Tot (\bar{W}G)$.
\end{remark}

\begin{definition}
 Given a cosimplicial group $G$, define the $n$th matching object $M^nG$ to be the group
\[
M^nG=\{(g_0,g_1,\ldots,g_{n-1}) \in (G^{n-1})^{n}\,|\, \sigma^ig_j=\sigma^{j-1}g_i\, \forall i<j \}.
\]
The Reedy matching map $G^n \to M^nG$ sends $g$ to $(\sigma^0g,\sigma^1g,\ldots,\sigma^{n-1}g) $.

There is then a Reedy model structure on $sc\Gp$ (analogous to  \cite{sht} \S VII.4) in which a morphism $f:G \to H$ is a (trivial) fibration whenever the canonical maps
\[
 G^n \to H^n\by_{M^nH} M^nG
\]
are (trivial) fibrations in $s\Gp$ for all $n \ge 0$. 
\end{definition}

\begin{definition}\label{normcdef}
 Given $G \in c\Gp$, define the cosimplicial normalisation by $N^n_cG:= G^n \cap \bigcap_{i = 0}^{n-1} \ker \sigma^i$. If $G$ is abelian, then we make $N_cG$ into a cochain complex by setting 
\[
 d_c:= \sum_{i=0}^n(-1)^i\pd^i:N_c^{n-1}G \to N_c^nG.
\]
 \end{definition}

\begin{lemma}\label{gplevelwise}
A morphism   $f:G \to H$ in $sc\Gp$ is a (trivial) fibration whenever the maps
\[
 f^n:G^n \to H^n
\]
are all (trivial) fibrations in $s\Gp$.
\end{lemma}
\begin{proof}
 First note that $N^n_cG= \ker (G^n \to M^nG)$. Given $(g_0,g_1,\ldots,g_{n-1}) \in M^nG $, we can functorially construct a pre-image. 

First, set $g(1):= \pd^1g_0$; this has $\sigma^0g(1)=g_0$. Proceeding by induction, assume that we have constructed $g(r) \in G^n$ with $\sigma^ig(r)= g_i$ for all $i<r$. Set  $g_i(r):= \sigma^ig(r)^{-1}\cdot g_i$, so $(g_0(r),g_1(r),\ldots,g_{n-1}(r)) \in M^nG $, with $g_i(r)=1$ for all $i<r$. Now let $g(r+1):= g(r)\cdot \pd^{r+1}g_r(r)$, noting that this satisfies the inductive hypothesis.

Thus we have an isomorphism $G^n \cong N^nG \by M^nG$ as simplicial sets, and $f:G \to H$ is therefore a (trivial) fibration whenever $N^nf: N^nG \to N^nH$ is a (trivial) fibration in $\bS$ for all $n$. Since $N^nf$ is a retraction of $f^n$, the result follows.
\end{proof}

\begin{lemma}\label{gpmcrQ}
 If  $f:G \to H$ is a  (trivial) fibration in $sc\Gp$, then the map
\[
\mmc(f): \mmc(G) \to \mmc(H)
\]
is  a (trivial) fibration in $\bS$. In particular, if $f:G \to H$ is a trivial fibration, then $\mc(f)$ is surjective.
\end{lemma}
\begin{proof}
 In the proof of \cite{monad} Proposition \ref{monad-rhommc2}, a cofibrant object $\Phi$ is constructed in $sc\Gp$, with the property that 
\[
 \mmc(G) \cong \HHom(\Phi, G),
\]
where the simplicial sets $\HHom$ come from a   simplicial model structure.
Since $\Phi$ is cofibrant, $\HHom(\Phi,-)$ is right Quillen, so has the properties claimed.
\end{proof}

\begin{definition}\label{totprod}
Define the total complex  functor $\Tot^{\Pi}$ from chain cochain complexes (i.e. bicomplexes) to chain complexes by
$$
(\Tot^{\Pi} V)_n := \prod_{a-b=n} V^b_a,
$$
with differential $d:=d^s +(-1)^a d_c$ on $V^b_a$.
\end{definition}

\begin{lemma}\label{cgpcotcoho}
If  $A \in cs\Gp$ is abelian, then 
$$
\mc(A)  \cong \z_{-1}(\Tot^{\Pi}N^sN_cA)
$$
and
\[
\pi_n\mmc(A)\cong\H_{n-1}(\Tot^{\Pi}\sigma^{\ge 1}N^sN_cA),
\]
where $\sigma^{\ge 1}$ denotes brutal truncation in cochain degrees $\ge 1$.
\end{lemma}
\begin{proof}
This is a fairly straightforward application of the simplicial and cosimplicial Dold--Kan correspondences. Alternatively, we could appeal to Proposition \ref{cfexp}, noting that $A = \exp(DN_cA)$. 
\end{proof}

\subsubsection{The gauge action}

\begin{definition}\label{gaugedefcgp}
For $G \in sc\Gp$, there is an action of the simplicial group $G^0$ on the simplicial set $\mmc(G)$, called the \emph{gauge action}, and  given by writing 
$$
(g\star \omega)_n= ( (\pd^1)^{n+1}(\sigma_0)^{n}g) \cdot \omega_n \cdot (\pd^0 (\pd^1)^n(\sigma_0)^ng^{-1}),
$$
as in \cite{htpy} Definition \ref{htpy-defdef}, with $(\sigma_0)^n$ denoting the canonical map $G \to G^{\Delta^n}$. 
\end{definition}

\begin{definition}\label{deldefcgp}
 Given $G \in sc\Gp$, define the Deligne groupoid by  
$
 \Del(G):= [\mc(G)/G^0_0]
$
In other words, $\Del(G)$ has  objects $\mc(G)$, and morphisms from $\omega$ to $\omega'$ consist of $\{g \in G^0_0\,:\, g\star \omega=\omega'\}$. 

Define the derived Deligne groupoid to be the simplicial object in groupoids given by $\uline{\Del}(G):= [\mmc(G)/G^0] $, so $\Del(G)= \uline{\Del}(G)_0$. 

Define  the simplicial sets $\ddel(G), \uline{\ddel}(G)\in \bS$ to be the nerves  $B\Del(G)$ and $\bar{W}\uline{\Del}(G)$, respectively.
\end{definition}

\begin{lemma}\label{cgpcot3}
If $A \in cs\Gp$ is abelian, then 
$$
\pi_n\uline{\ddel}(A)  \cong \H_{n-1}(\Tot^{\Pi}N^sN_cA),
$$
whereas $\pi_1\ddel(A)\cong  \H^0(A_0)$, with 
\[
 \pi_0\ddel(A)\cong \z_{-1}(\Tot^{\Pi}N^sN_cA)/ d_c(A_0^0).
\]
\end{lemma}
\begin{proof}
This is a straightforward consequence of Lemma \ref{cgpcotcoho}.
\end{proof}

\subsection{Moduli functors from cosimplicial groups}

\begin{proposition}\label{mcgpnice}
 If $G: Alg_R \to c\Gp$ is a homogeneous functor, with each $G^n$ formally smooth, then the functor 
\[
 \mmc(\uline{G}):d\cN^{\flat}_R \to \bS
\]
is homogeneous and formally quasi-smooth, so
\[
 \mc(\uline{G}):d\cN^{\flat}_R \to \Set
\]
is homogeneous and pre-homotopic.
\end{proposition}
\begin{proof}
 Homogeneity is automatic, as $\mmc$ preserves arbitrary limits. 

We can extend $G^n$ to a functor $G^n: d\cN_R^{\flat} \to \Gp$, given by $G^n(A):= G^n(A_0)$. Then formal smoothness of $G^n$ implies that the extended $G^n$ is  pre-homotopic. It is automatically formally quasi-presmooth, as all discrete morphisms are fibrations. Thus Proposition \ref{settotop} implies that  $\uline{G}^n:d\cN^{\flat}_R \to \bS $ is formally smooth, and hence formally quasi-smooth.
  
Lemma \ref{gplevelwise} therefore implies that $\uline{G}(A) \to \uline{G}(B)$ is a (trivial) fibration in $sc\Gp$ for all (acyclic) square-zero extensions $A \to B$, and  Lemma \ref{gpmcrQ} then implies that $\mmc(\uline{G})(A) \to \mmc(\uline{G})(B)$ is a (trivial) fibration, as required.
\end{proof}

\begin{proposition}\label{delgpnice}
 If $G: Alg_R \to c\Gp$ is a homogeneous functor, with each $G^n$ formally smooth, then the functor 
\[
 \uline{\ddel}(\uline{G}):d\cN^{\flat}_R \to \bS
\]
is homogeneous and formally quasi-smooth, while
\[
 \ddel(\uline{G}):s\cN^{\flat}_R \to \bS
\]
is homogeneous, pre-homotopic and formally quasi-presmooth.
\end{proposition}
\begin{proof}
Homogeneity is immediate, combining Proposition \ref{mcgpnice} with the fact that $\uline{G}^0$ is homogeneous. Now, take a square-zero (acyclic) extension $f:A \to B$ in $d\cN^{\flat}_R$. Since $\mmc(\uline{G} )$ is formally quasi-smooth, the map $\mmc(\uline{G}(A))\to \mmc(\uline{G}(B))$ is a (trivial) fibration.

Combining \cite{sht} Lemma IV.4.8 with \cite{CRfib}, this means that
\[
 \bar{W}[\mmc(\uline{G}(A))/\uline{G}^0(A)] \to \bar{W}[\mmc(\uline{G}(B))/\uline{G}^0(A)]
\]
 is a (trivial) fibration in $\bS$. Now, the map
\[
  \bar{W}[\mmc(\uline{G}(B))/\uline{G}^0(A)] \to  \bar{W}[\mmc(\uline{G}(B))/\uline{G}^0(B)]
\]
is a pullback of $\bar{W}\uline{G}^0(A)\to \bar{W}\uline{G}^0(B)$, which is a (trivial) fibration as $\uline{G}^0$ is formally smooth. Composing the two morphisms above, we see that
\[
 \uline{\ddel}(\uline{G})(A) \to  \uline{\ddel}(\uline{G})(B) 
\]
is a (trivial) fibration, so 
    $\uline{\ddel}(\uline{G})$ is formally quasi-smooth.

Meanwhile, pre-homotopicity of $\ddel(\uline{G})$ follows immediately from  pre-homotopicity of $\mc(\uline{G})$, while formal quasi-presmoothness of $\ddel(\uline{G})$ is an immediate consequence of formal smoothness of $G^0$, by  \cite{sht} Ch.V. 
\end{proof}

\subsubsection{Cohomology}

\begin{definition}
 Given  a homogeneous, levelwise formally smooth functor    $G: Alg_R \to c\Gp$, a ring $A \in \Alg_R$, an $A$-module $M$  and $\omega\in \mc(G(A))$, define the cosimplicial module $\CC^{\bt}_{\omega}(G, M)$ to be the tangent space
\[
\CC^{n}_{\omega}(G, M):=T_{1}(G^n,M)
\]
with operations on $a \in \CC^{n}_{\omega}(G, M)$ given by   
%
\begin{eqnarray*}
        \sigma^ia&=& \sigma^i_Ga\\
\pd^ia&=& \left\{\begin{matrix}   ((\pd^2_G)^{n}\omega) (\pd^0_Ga)((\pd^2_G)^{n}\omega^{-1}) & i=0\\  \pd^i_Ga & i \ge 1. \end{matrix} \right. 
\end{eqnarray*}
\end{definition}

\begin{definition}\label{cgpcoho}
 For   $G, A,M, \omega$ as above, define
\[
 \H^i_{\omega}(G,M):=   \H^i \CC^{\bt}_{\omega}(G, M).    
\]
     
\end{definition}

\begin{lemma}
 Given  a homogeneous, levelwise formally smooth functor    $G: Alg_R \to sc\Gp$, $A \in \Alg_R$, $M \in d\Mod_A$
and $\omega\in \mc(G_0(A))$, the fibre of $\mc(G(A \oplus M)) \to \mc(G(A))$ over $\omega$ is canonically isomorphic to $\mc(\CC^{\bt}_{\omega}(G, M))$.
\end{lemma}
\begin{proof}
Given $\alpha \in \mc(\CC^{\bt}_{\omega}(G, M)))$, the associated element $\beta \in \mc(G(A \oplus M))$ is given by 
\[
 \beta_n:= \alpha_n (\pd^2_G)^n\omega \in T_{ (\pd^2_G)^n\omega}(G^{n+1}_n, M).
\]
\end{proof}


\begin{lemma}\label{cgpcohomc}
If $G: Alg_R \to c\Gp$ is a homogeneous, levelwise formally smooth functor, with $A \in \Alg_R$ and $M$ an $A$-module,   then 
\[
 \DD^i_{\omega}(\mmc(\uline{G}),M)\cong  \DD^i_{\omega}(\mc(\uline{G}),M) \cong \left\{\begin{matrix} \H^{i+1}_{\omega}(G,M) & i >0 \\ \z^1\CC^{\bt}_{\omega}(G,M) & i=1. \end{matrix} \right.  
\]
\end{lemma}
\begin{proof}
By Lemma \ref{cgpcotcoho}, for any $L \in d\Mod_A$,
\[
 \DD^i_{\omega}(\mmc(\uline{G}),L) \cong \H_{i-1}(\Tot^{\Pi}\sigma^{\ge 1}N^sN_c \CC^{\bt}_{\omega}(\uline{G},L)).
\]
Thus   we have  a  spectral sequence
\[
 E_2^{i,-j}=\H^i(\sigma^{\ge 1}\pi_j\CC^{\bt}_{\omega}(\uline{G},L)) \abuts \DD^{i-j-1}_{\omega}(\mmc(\uline{G}),L);
\]
in the terminology of \cite{W} p.142, this is a second quadrant spectral sequence, so is weakly convergent.

  The simplicial abelian group $\CC^{n}_{\omega}(\uline{G},L)$ is given in level $i$ by $\CC^{n}_{\omega}(\uline{G},(L^{\Delta^i})_0)$. Moreover, $\CC^{\omega}(G,-)$ is an exact functor; left exactness follows from homogeneity, and right exactness from formal smoothness. Thus $\pi_j\CC^{n}_{\omega}(\uline{G},L) \cong \CC^{n}_{\omega}(\uline{G}, \pi_j(L^{\Delta^{\bt}})_0)$. 

Now, when $d\cN^{\flat}_R= s\cN^{\flat}_R$,  we have $(L^{\Delta^n})_0=L_n$, so  $(L^{\Delta^{\bt}})_0=L$. When  $d\cN^{\flat}_R= dg_+\cN^{\flat}_R$, $N^s(L^{\Delta^{\bt}})_0$ is weakly equivalent to $L$. In either case, $\pi_j(L^{\Delta^{\bt}})_0 \cong \H_jL$, so our spectral sequence is
\[
 \H^i(\sigma^{\ge 1}\CC^{\bt}_{\omega}(G,\H_jL)) \abuts  \DD^{i-j-1}_{\omega}(\mmc(\uline{G}),L).
\]
Taking $L=M[-n]$, the spectral sequence degenerates, giving 
\[
 \H^i(\sigma^{\ge 1}\CC^{\bt}_{\omega}(G,M))\cong \DD^{i-n-1}_{\omega}(\mmc(\uline{G}),M[-n])= \DD^{i-1}_{\omega}(\mmc(\uline{G}),M),
\]
completing the proof for $\mmc$.

Now, $T_{\omega}(\mmc,-):d\Mod_A \to \bS$ preserves fibrations and trivial fibrations. Since $M[-j]\oplus \cone(M)[1-j]$ is a path object for $M[-j]$  when $j \ge 1$ (recalling that $M$ is a discrete $A$-module), this means that $T_{\omega}(\mmc,M[-j]\oplus \cone(M)[1-j] )$ must be a path object for $T_{\omega}(\mmc,M[-j])$. Therefore for $j \ge 1$,
\begin{eqnarray*}
 \pi_0T_{\omega}(\mmc,M[-j])&=& T_{\omega}(\mc,M[-j])/ T_{\omega}(\mc, \cone(M)[1-j] )\\
\DD^j_{\omega}(\mmc, M)&=&  \DD^j_{\omega}(\mc, M).
\end{eqnarray*}
 The proof for $j=0$ is even simpler, as $M^{\Delta^n}=M$, so $T_{\omega}(\mmc,M)= T_{\omega}(\mc,M)$.
\end{proof}

\begin{lemma}\label{cgpcohodel}
If $G: Alg_R \to c\Gp$ is a homogeneous, levelwise formally smooth functor, with $A \in \Alg_R$ and $M$ an $A$-module,   then 
\[
 \DD^i_{\omega}(\uline{\ddel}(\uline{G}))\cong \DD^i_{\omega}(\ddel(\uline{G})) \cong\H^{i+1}_{\omega}(G,M).    
\]
\end{lemma}
\begin{proof}
The description of $\DD^i_{\omega}(\uline{\ddel}(\uline{G}))$ follows with the same reasoning as Lemma \ref{cgpcohomc}, substituting Lemma \ref{cgpcot3} for Lemma \ref{cgpcotcoho}.

Now, there is a morphism
\[
 \xymatrix{ \mc(\uline{G}) \ar[r] \ar[d] & \ddel(\uline{G}) \ar[r]\ar[d] & BG^0 \ar[d] \\
\mmc (\uline{G})\ar[r] & \uline{\ddel}(\uline{G}) \ar[r] & \bar{W}\uline{G}^0 
}      
\]
of fibration sequences, with the outer maps inducing isomorphisms on $\DD^i$, so Proposition \ref{longexact} gives the required isomorphisms $ \DD^i_{\omega}(\ddel(\uline{G}))\cong\DD^i_{\omega}(\uline{\ddel}(\uline{G}))$.
 \end{proof}

\subsection{Denormalisation}\label{denormsn}

\begin{definition}\label{denormdef}
Given a DGLA $L$ in non-negative degrees, let $DL$ be its cosimplicial denormalisation. Explicitly,
$$
D^nL:= \bigoplus_{\begin{smallmatrix} m+s=n \\ 1 \le j_1 < \ldots < j_s \le n \end{smallmatrix}} \pd^{j_s}\ldots\pd^{j_1}L^m,
$$
for formal symbols $\pd^j$.
We then  define operations $\pd^j$ and $\sigma^i$ using the cosimplicial identities, subject to the conditions that $\sigma^i L =0$ and $\pd^0v= dv -\sum_{i=1}^{n+1}(-1)^i \pd^i v$ for all $v \in L^n$.

We now have to define the Lie bracket $\llbracket-,-\rrbracket$ from $D^nL \ten D^nL$ to $D^n L$. Given a finite set  $I$ of distinct strictly positive integers, write $\pd^I= \pd^{i_s}\ldots\pd^{i_1}$, for $I=\{i_1, \ldots i_s\}$, with $i_1 < \ldots < i_s$. The Lie bracket is then   defined on the basis by 
$$
\llbracket \pd^Iv, \pd^J w\rrbracket:= \left\{ \begin{matrix} \pd^{I\cap J}(-1)^{(J\backslash I, I \backslash J)}[v,w] & v\in L^{|J\backslash I|}, w\in L^{|I\backslash J|},\\ 0 & \text{ otherwise},\end{matrix} \right.
$$
where for disjoint sets $S,T$ of integers, $(-1)^{(S,T)}$ is the sign of the shuffle permutation of $S \sqcup T $ which sends the first $|S|$ elements to $S$ (in order), and the remaining $|T|$ elements to $T$ (in order). 
Beware that this formula cannot be used to calculate  $\llbracket \pd^Iv, \pd^J w\rrbracket $ when $0 \in I \cup J$ (for the obvious generalisation of $\pd^I$ to finite sets $I$ of distinct non-negative integers).
\end{definition}

Of course, the denormalisation functor above extends a denormalisation functor $D$ from non-negatively graded cochain complexes to cosimplicial complexes. The latter $D$ is quasi-inverse to the normalisation functor $N_c$ of Definition \ref{normcdef}. 

\begin{definition}
Given a pro-nilpotent Lie algebra $\g$, define $\hat{\cU}(\g)$ to be the pro-unipotent completion of the universal enveloping algebra of $\g$, regarded as a pro-object in the category of algebras. As in \cite{QRat} Appendix A, this is a pro-Hopf algebra, and we define $\exp(\g) \subset  \hat{\cU}(\g)$ to consist of elements $g$ with $\vareps(g)=1$ and $\Delta(g)= g\ten g$, for $\vareps: \hat{\cU}(\g) \to k$ the augmentation (sending $\g$ to $0$), and $\Delta: \hat{\cU}(\g) \to \hat{\cU}(\g)\ten \hat{\cU}(\g)$ the comultiplication.

Since $k$ is assumed to have characteristic $0$, exponentiation gives an isomorphism from $\g$ to $\exp(\g)$, so we may regard $\exp(\g)$ as having the same elements as $\g$, but with multiplication given by the Campbell--Baker--Hausdorff formula. 
\end{definition}

If $L$ is a DGLA in strictly positive degrees, observe that we can write $L$ as the inverse limit $L= \Lim_n \sigma^{\le n}L$ of nilpotent DGLAs, where $\sigma^{\le n}$ denotes brutal truncation. We may thus regard $DL$ as the pro-nilpotent cosimplicial Lie algebra $\Lim_n D(\sigma^{\le n}L)$, so we can exponentiate to obtain $\exp(DL):= \Lim_n \exp(D(\sigma^{\le n}L))$.


\begin{lemma}\label{cfexp}
Given a simplicial  DGLA $L^{\bt}_{\bt}$ in strictly positive cochain degrees, there is a canonical isomorphism
$$
 \mc(\exp(DL)) \cong \mc(\Tot{\Pi} N^sL),  
$$
 Here, $N^s$ is simplicial normalisation (as in Definition \ref{normdef}). 
\end{lemma}
\begin{proof}
This is \cite{monad} Theorem \ref{monad-cfexp}.

\end{proof}

\begin{definition}
 Given a DGLA $L$ with gauge $G_L$, define the cosimplicial group $D(\exp(L),G_L)$ as follows. 
\[
 D^n(\exp(L),G_L):=  \exp(D^n L^{>0})\rtimes G_L ,       
\]
(with $G_L$ acting on $\exp(D^n L^{>0})$ via the adjoint action $\ad$), with operations
\begin{eqnarray*}
\sigma^i(a,g)&=& ( \sigma^ia,g)\\
\pd^i(a,g) &=& \left\{\begin{matrix}  ( \pd^ia,g) & i>0 \\          
                                        ( \pd^0a \cdot \exp((\pd^2)^nDg), g) &i=0, 
\end{matrix}\right.
 \end{eqnarray*}
for $(a,g) \in D^n(\exp(L),G_L)$, and $Dg \in L^1$.
\end{definition}


\begin{remark}
 If $L$ is a nilpotent DGLA in non-negative degrees and $G_L= \exp(L^0)$ as in Example \ref{nilpdglagauge}, then observe that $D (\exp(L),G_L)\cong \exp (DL)$, with the isomorphism given in level $n$ by 
\[
 (a,g) \mapsto a \cdot (\pd^1)^ng.
\]

The only difficult part of the comparison is checking that the isomorphism preserves $\pd^0$.
This follows because for $v \in L^0$, we have $\pd^0v= \pd^1v+dv$. Since $\llbracket dv, \pd^1v \rrbracket =0$, this gives 
\begin{eqnarray*}
        \exp(\pd^0v)&=& \exp(\pd^1v) + d\exp(v)\\
&=&  (1+  (d\exp(v))\exp(-v)) \cdot \exp(\pd^1v)\\
&=&  (1+  D\exp(v)) \cdot \exp(\pd^1v)\\
&=& \exp( D\exp(v)) \cdot \exp(\pd^1v).
\end{eqnarray*}
\end{remark}

\begin{lemma}
 Given a simplicial DGLA $L$ with gauge $G_L$ for $L_0$, the isomorphism $\mc(D(\exp(L),G_L)) \cong \mc(\Tot{\Pi} N^sL)$ of Lemma \ref{cfexp} is $G_L$-equivariant for the respective gauge actions.
\end{lemma}
\begin{proof}
 This is a consequence of the proof of \cite{htpy} Theorem \ref{htpy-defqs}, which deals with the case when   $G_L= \exp(L^0)$ is defined as in Example \ref{nilpdglagauge}.      
\end{proof}

\begin{corollary}\label{cfexp2}
 Given a  DGLA $L$ over $R$, with  gauges $G_L(A)$  for $L\ten_RA$,    functorial in $A \in \Alg_R$, there are canonical isomorphisms
\begin{eqnarray*}
 \mc(L\ten_R-) &\cong&\mc( D(\exp(L\ten_R-),G_L(-))\\
 {[\mc(L\ten_R-)/G_L(-) ]}&\cong&  \Del(D(\exp(L\ten_R-),G_L(-) )  
\end{eqnarray*}
of functors on $dg_+\Alg_R$.
\end{corollary}

Thus cosimplicial groups generalise DGLAs, with the added advantage that they can also give functors on $s\Alg_R$, and hence work in all characteristics.

\subsection{Sheafification}\label{shfcgpsn}

Another advantage of cosimplicial groups over DGLAs is that they have a good notion of sheafification.

\begin{definition}\label{shfcgp}
Given a levelwise formally smooth, homogeneous functor  $G: Alg_R \to c\Gp$ preserving finite products, 
and some class $\oP$ of covering morphisms in $\Alg_R$, define the sheafification $G^{\sharp}$ of $G$ with respect to $\oP$ by
first defining a cosimplicial commutative $A$-algebra $(B/A)^{\bt}$ for every $\oP$-covering $A\to B$, as $(B/A)^n:= \overbrace{B\ten_{A}B\ten_{A}\ldots \ten_{A}B}^{n+1}$, then setting
\[
 G^{\sharp}(A) = \diag \LLim_{B} G((B/A)^{\bt}),       
\]
where $\diag$ is the diagonal functor $(\diag X)^n= X^{nn}$ from bicosimplicial groups to cosimplicial groups. 
\end{definition}

\begin{definition}\label{shfgpd}
 Given a groupoid-valued functor $\Gamma: d\cN^{\flat}_R \to \Gpd$, and a class $\oP$ of covering morphisms in $\Alg_R$, define $\Gamma^{\sharp}:  d\cN^{\flat}_R \to \Gpd$ to be the  stackification of the groupoid presheaf $\Gamma$ in the 
strict  $\oP$-topology of Definition \ref{strictcover}. 
\end{definition}

\begin{lemma}\label{cfshfcgp}
For $G$ and $\oP$ as above, there is a canonical morphism
\[
 \Del(\uline{G})^{\sharp} \to  \Del(\uline{G}^{\sharp})       
\]
 of groupoid-valued functors on $d\cN^{\flat}_R$, inducing an equivalence
\[
 \pi^0\Del(\uline{G})^{\sharp} \to \pi^0\Del(\uline{G}^{\sharp}).
\]
\end{lemma}
\begin{proof}
   An object of $\Del(\uline{G})^{\sharp}$ is a pair $(\omega, g) \in \mc( \uline{G}(B\ten_{A_0}A))\by G^0(B\ten_{A_0}B)$, for $A_0 \to B$ a $\oP$-covering, satisfying the following conditions:
\begin{enumerate}
 \item $g\star (\pr_1^*\omega) = (\pr_0^*\omega) \in \mc(  \uline{G}(A\ten_{A_0}B\ten_{A_0}B))$,
\item $\pr_{02}^*g = (\pr_{01}^*g)\cdot  (\pr_{12}^*g) \in G^0(B\ten_{A_0}B\ten_{A_0}B)$,
\end{enumerate}
and we now describe the image of $(\omega, g)$.

First, form the cosimplicial $A_0$-algebra $(B/A_0)^{\bt}$ as in Definition \ref{shfcgp}.   We now map $(B,\omega, g)$ to the object $\omega'$ of $\mc(\diag \uline{G}(A\ten_{A_0}(B/A_0)^{\bt}))$ given by
\[
  \omega'_n:=  (\pr_{01}^*(\pd^1)^{n+1}_Gg) \cdot \pr_1^*\omega_n     \in G^{n+1}( (A^{\Delta^n})_0\ten_{A_0}(B/A_0)^{n+1}).  
\]

An isomorphism from $(B,\omega, g)$ to $(C,\nu, h)$ is a $\oP$-covering $B\ten_{A_0}C \to D$ with $A \to D$  a $\oP$-covering, together with an element $\alpha \in G^0(D)$ such that $\alpha\star \omega = \nu \in \mc(\uline{G} ( A\ten_{A_0}D))$, with $(\pr_0^*\alpha)\cdot g = h\cdot (\pr_1^*\alpha) \in G^0(D\ten_{A_0}D)$. We just map this  to $\alpha \in G^0(D)$.

To see that this induces an equivalence $ \Del(\uline{G})^{\sharp}(A) \to \Del(\uline{G}^{\sharp}(A))$ for $A \in \Alg_R$, we appeal to \cite{monad} Lemma \ref{monad-diagset}, which shows that objects of 
\[
 \mc (\diag G((B/A)^{\bt}))      
\]
correspond to pairs $(\omega,g) \in \mc(G(B))\by \mc(G^0( (B/A)^{\bt}))$ satisfying $\pd_B^1\omega \cdot \pd_G^0g= \pd^1g\cdot \pd_B^0\omega$. Since $\pd^1_B\omega= \pr_0^*\omega$ and $ \pd^0_B\omega= \pr_1^*\omega$, this amounts to saying that  $g$ and $\omega$ satisfy condition (1) above, while condition (2) is equivalent to saying that $g \in \mc(G^0((B/A)^{\bt})$. Morphisms in $\mc (\diag G((B/A)^{\bt}))$ are given by the gauge actions of $G^0(B^0)$, so the equivalence of groupoids follows.
 \end{proof}

Now recall that $\oP$-covering morphisms are all assumed faithfully flat.
\begin{lemma}\label{cfshfcgp2}
 Take $G:\Alg_R \to c\Gp$ satisfying the conditions of Definition \ref{shfcgp}, a ring $A \in \Alg_R$, an $A$-module $M$, and an object of $\Del(G)^{\sharp}(A)$ represented by $(\omega,g)$ for $\omega \in \mc(G(B))$. If the maps 
\[
 \H^*_{\omega}(G, M\ten_AB)\ten_BB' \to \H^*_{\omega}(G, M\ten_AB')
\]
are isomorphisms for all $\oP$-coverings $B \to B'$, then the morphism $\alpha$ of Lemma \ref{cfshfcgp} induces an isomorphism
\[
 \DD^*_{(\omega,g)}(B\Del(\uline{G})^{\sharp},M) \to  \DD^*_{\alpha(\omega,g)}(B\Del(\uline{G}^{\sharp}),M).       
\]
\end{lemma}
\begin{proof}
We begin by calculating the cohomology groups $ \DD^*_{(\omega,g)}(B\Del(\uline{G})^{\sharp},M)$. It follows from Lemma \ref{cgpcohodel} that these are given by first taking a cosimplicial $A$-module $K(B')$ given by  the equaliser  of
\[
 \xymatrix@1{ \CC^{\bt}_{\omega}(G, M\ten_AB') \ar@<.5ex>[r]^-{\pr_1^*} \ar@<-.5ex>[r]_-{\ad_g\pr_0^*} & \CC^{\bt}_{\pr_1^*\omega}(G, M\ten_AB'\ten_AB'),}
\]
then getting 
\[
  \DD^i_{(\omega,g)}(B\Del(\uline{G})^{\sharp},M)\cong \LLim_{B'}\H^{i+1}K(B'),
\]
where $B'$ ranges over all $\oP$-hypercoverings $B\to B'$.

Now, the requirement that  $\H^*_{\omega}$  commute with base change ensures that $\ad_g$ gives an effective descent datum on cohomology, giving an isomorphism
\[
 \H^*K(B')\ten_AB' \cong \H^*_{\omega} (G, M\ten_AB').
\]
Thus taking the colimit over $B'$ does not affect the calculation, so
\[
  \DD^i_{(\omega,g)}(B\Del(\uline{G})^{\sharp},M)\cong\H^{i+1}K(B).
\]

Meanwhile, $g$ allows us to extend the fork above to form a bicosimplicial complex $\check{\CC}^{\bt}(B'/A, \CC^{\bt}_{\omega}(G,M))$, with 
\[
 \check{\CC}^{i}(B'/A, \CC^{\bt}_{\omega}(G,M))= \CC^{\bt}_{\pr_i^*\omega}(G,M\ten_A(B'/A)^i)       
\]
and horizontal cohomology $\check{\H}^0 (B'/A, \CC^{\bt}_{\omega}(G,M))= K(B')$.
Lemma \ref{cgpcohodel} then shows that 
\[
 \DD^i_{\alpha(\omega,g)}(B\Del(\uline{G}^{\sharp}),M)  \cong \LLim_{B'} \H^{i+1} (\diag  \check{\CC}^{\bt}(B'/A, \CC^{\bt}_{\omega}(G,M))),     
\]
and the Eilenberg--Zilber Theorem allows us to replace $\diag$ with the total complex functor $\Tot$.

Now,  there are canonical maps $ K(B)\ten_A(B'/A)^n \to \check{\CC}^{n}(B'/A, \CC^{\bt}_{\omega}(G,M))$, and we know that these give isomorphisms on cohomology, so
\[
  \DD^i_{\alpha(\omega,g)}(B\Del(\uline{G}^{\sharp}),M)  \cong  \LLim_{B'} \H^{i+1}(\Tot K(B)\ten_A (B'/A)^{\bt}).    
\]
Since $\H^j((B'/A)^{\bt})=0$ for all $j>0$, this becomes
\[
  \DD^i_{\alpha(\omega,g)}(B\Del(\uline{G}^{\sharp}),M)  \cong  \LLim_{B'} \H^{i+1}K(B),     
\]
as required.
\end{proof}

\begin{remark}
In particular, this means that the groupoid-valued functor $\cG$ of Definition \ref{cGpoldef} can be replaced by a functor coming straight from a cosimplicial group-valued functor. Explicitly, set $G_{[p,q]}:= D(\exp(L_{[p,q]}),G_{L_{[p,q]}})$; then $G:= \LLim_p\Lim_qG_{[p,q]}$ satisfies the conditions of Definition \ref{shfcgp}, so Lemma \ref{cfshfcgp} gives a map $\cG \to \Del(G^{\sharp})$, inducing an equivalence on $\pi^0$, and isomorphisms on $\DD^i$ for all points of $\cM \subset \pi^0\cG$.
\end{remark}

\begin{remark}
In Definition \ref{shfcgp}, instead of just taking hypercovers $A_0 \to B^{\bt}$ coming from $\oP$-covering morphisms  $A_0 \to B$, we could have taken 
the filtered colimit  over the  category  of all simplicial $\oP$-hypercovers $\Spec \tilde{A}^{\bt} \to \Spec A$. This corresponds to a kind of hypersheafification, rather than just sheafification, and all the results above still carry over, by faithfully flat descent.
\end{remark}

\subsection{Derived moduli of $G$-torsors}

We now show how cosimplicial groups can govern derived moduli of torsors. Fix a smooth algebraic group space $G$ over $R$, and a Deligne--Mumford stack $X$ over $R$.

\begin{definition}
 Define $\CC^{\bt}_{\et}(X,G): \Alg_R \to c\Gp$ as follows. Given $A \in \Alg_R$, let $\hyp_{\et}(X,A)$ be the inverse category whose objects are  simplicial \'etale hypercovers $Y_{\bt}$ of $X \by \Spec A$,  and whose morphisms are simplicial \'etale hypercovers. Then 
set
\[
  \CC^n_{\et}(X,G)(A):= \LLim_{Y_{\bt} \in \hyp_{\et}(X,A)} \Hom(Y_n,G).     
\]
\end{definition}

\begin{remark}
 Since $G$ is finitely presented, we can work just as well with pro-\'etale hypercovers. This has the advantage that Lemma \ref{univetcover} can be applied to provide a  weakly initial object among simplicial  hypercovers, giving a smaller, non-functorial, model for   $\CC^{\bt}_{\et}(X,G)(A)$.  
\end{remark}

\begin{lemma}
 For $A \in \Alg_R$, the groupoid $\Del( \CC^n_{\et}(X,G)(A))$ is equivalent to the groupoid of $G$-torsors on $X \by \Spec A$.      
\end{lemma}
\begin{proof}
An object   of   $\mc( \CC^{\bt}_{\et}(X,G)(A))$ is a descent datum $\omega \in \Hom(Y_1,G)$, with the Maurer--Cartan relations giving the gluing conditions. Thus $\omega$ gives rise to an \'etale $G$-torsor $\Bu_{\omega}$ on $X\by \Spec A$, and it is straightforward to check that the gauge action of $\Hom(Y_0,G)$ corresponds to  isomorphisms of torsors. Every torsor is trivialised by some \'etale cover, so this functor is an equivalence.
\end{proof}

\begin{lemma}
Given   $A \in \Alg_R$, an $A$-module $M$, and  $\omega \in \mc( \CC^{\bt}_{\et}(X,G)(A))$ corresponding to a $G$-torsor $\Bu_{\omega}$ on $X \by \Spec A$, there are canonical isomorphisms
\[
      \H^i_{\omega}(\CC^{\bt}_{\et}(X,G),M)\cong  \H^{i}_{\et}(X\by \Spec A, M\ten_A\ad \Bu_{\omega}), 
\]
 where $\ad \Bu_{\omega} $ is the adjoint bundle 
\[
        (\g\ten_R\O_X\ten_RA ) \by_{G(\O_X\ten_RA)}\Bu_{\omega},
\]
 for $\g$ the Lie algebra of $G$, equipped with its adjoint $G$-action. 
\end{lemma}
\begin{proof}
First, observe that 
\[
 G( \O_X\ten_R(A\oplus M)) \cong G(\O_X\ten_RA) \ltimes (\g\ten_R\O_X\ten_RM),
\]
which gives functorial isomorphisms
\[
 \CC^{i}_{\omega}(\CC^{\bt}_{\et}(X,G),M) \cong \LLim_{Y_{\bt}} \Gamma(Y_i, M\ten_A\ad \Bu_{\omega}).
\]
Since \'etale hypercovers compute cohomology, this gives the required isomorphism.
\end{proof}

\begin{proposition}\label{torsorconsistent}
 The functor $\uline{\ddel}(\uline{ \CC^{\bt}_{\et}(X,G)})$ is canonically weakly equivalent to the derived stack of \'etale derived $G$-torsors on $X$ from \cite{dmsch} Example \ref{dmsch-Gtorsors}.
\end{proposition}
\begin{proof}
 Combining Lemma \ref{cgpcohodel} and Proposition \ref{delgpnice} with Remark \ref{predetectweakh}, it suffices to construct a functorial natural transformation from $\Del( \uline{ \CC^{\bt}_{\et}(X,G)})(A)$ to the $\infty$-groupoid $\mathrm{Tors}(X,G)(A)$ of $G$-torsors on $X \by \Spec A$, and to show that this is an equivalence for $A \in \Alg_R$, inducing isomorphisms on cohomology groups $\DD^i$.

Our first key observation is that for $A \in d\cN_R^{\flat}$, the ring $(A^{\Delta^n})_0$ is a nilpotent extension of $\H_0A$, so its \'etale site is isomorphic to that of $A_0$. In particular, every simplicial \'etale hypercover of $X \by \Spec (A^{\Delta^n})_0$ is of the form $Y_{\bt}\ten_{A_0}(A^{\Delta^n})_0$, for $Y_{\bt}$ a simplicial \'etale hypercover of $X \by \Spec A_0$.

Thus an element $\omega \in \mc(\uline{ \CC^{\bt}_{\et}(X,G)}(A))$ lies in $\mc$ applied to the simplicial cosimplicial group $\Gamma(Y_{\bt}\ten_{A_0}(A^{\Delta^{\bt}}),G)$ given by
$(i,j) \mapsto \Hom(Y_i\ten_{A_0}(A^{\Delta^j})_0,G)$, for some  simplicial \'etale hypercover $Y_{\bt}$ of $X \by \Spec A_0$. Now, 
\begin{eqnarray*}
 \mc( \Gamma(Y_{\bt}\ten_{A_0}(A^{\Delta^{\bt}}),G)) &=& \Hom_{c\bS}(\Delta, \bar{W}\Gamma(Y_{\bt}\ten_{A_0}(A^{\Delta^{\bt}}),G))\\
&\cong& \Hom_{s\Pr(d\cN_R^{\flat}) }(Y_{\bt}\by_{\Spec A_0} \Spec A, \bar{W}\uline{G}),
\end{eqnarray*}
where $s\Pr(d\cN_R^{\flat})$ denotes the category of functors $ d\cN_R^{\flat} \to \bS $.

For a simplicial group $\Gamma$, and $\bar{W}$ as in Definition \ref{barwdef2}, there is a universal principal $\Gamma $-space $W\Gamma$ over $\bar{W}\Gamma$, as  in \cite{sht} \S V.4, given by $W\Gamma= \bar{W}[\Gamma/\Gamma]$, (whereas $\bar{W}\Gamma= \bar{W}[\bt/\Gamma]$, regarding a group as a groupoid on one object). Thus $W\Gamma$ has a group structure inherited from $\Gamma$, and $\Gamma= \bar{W}[\Gamma/\{1\}]$ is a subgroup; the $\Gamma$-action on $W\Gamma$ is then given by left multiplication, with $\bar{W}\Gamma= \Gamma\backslash W\Gamma$. 

Thus we may associate the $\uline{G}$-space $ P_{\omega}:=(Y_{\bt}\by_{\Spec A_0} \Spec A)\by_{\bar{W}\uline{G}}W\uline{G}$ to $\omega$. Since $\uline{G}$ is the derived stack $\oR G$ associated to $G$, the derived stack $\oR P_{\omega}$ associated to $P_{\omega}$ is a derived $\oR G$-torsor on $\oR ( Y_{\bt}\by_{\Spec A_0} \Spec A)\simeq \oR (X\by_{\Spec A_0} \Spec A)$, so we have defined our functor on objects. 

Now, the constant group $\Gamma_0$ is a simplicial subgroup of $\Gamma$, giving a simplicial group homomorphism $W\Gamma_0\to W\Gamma$. Moreover, $W\Gamma_0$ is the $0$-coskeleton $\cosk_0\Gamma_0$ of $\Gamma_0$, so $\Hom_{\bS}(Y, W\Gamma_0) \cong \Hom_{\Set}(Y_0,\Gamma_0)$. From the proof of \cite{htpy} Proposition \ref{htpy-psidef}, the gauge action (Definition \ref{gaugedefcgp}) of $\Hom(Y_0,\Gamma_0)$ on $\mc(\CC^{\bt}(Y,\Gamma))= \Hom(Y, \bar{W}\Gamma)$ corresponds to the right multiplication by $W\Gamma_0$ on $\bar{W}\Gamma = \Gamma\backslash W\Gamma$.   

Hence, given $\omega, \omega' \in \Hom_{s\Pr(d\cN_R^{\flat}) }(Y_{\bt}\by_{\Spec A_0} \Spec A, \bar{W}\uline{G})$ and $g \in \Hom_{\Pr(d\cN_R^{\flat}) }(Y_0, G)$ with $g\star \omega= \omega'$, this means that $\omega'(y)= \omega(y)\cdot g(y)^{-1}$.
We therefore construct an isomorphism $P_{\omega} \to P_{\omega'}$ by $(y,w) \mapsto (y, w\cdot g(y)^{-1})$, for $y \in Y_{\bt}\by_{\Spec A_0} \Spec A$ and $w \in W\uline{G}$.

We have thus constructed a morphism $\Del( \uline{ \CC^{\bt}_{\et}(X,G)}) \to \mathrm{Tors}(X,G)$. The nerve of $ \mathrm{Tors}(X,G)$ is weakly equivalent to the derived stack $\cHom(X, BG)$ from \cite{dmsch} Example \ref{dmsch-Gtorsors}. Since $\g= \Hom_R(e^*\Omega_{G/R},R)$, the calculation of \cite{dmsch} Example \ref{dmsch-Gtorsors} gives
\[
 \DD^i_{\Bu}(\cHom(X, BG),M)\cong \H^{i+1}_{\et}(X\by \Spec A, M\ten_A\ad \Bu_{\omega}).
\]

On taking nerves, we thus get maps
\[
\uline{\ddel}( \uline{ \CC^{\bt}_{\et}(X,G)}) \la \ddel( \uline{ \CC^{\bt}_{\et}(X,G)})\to \bar{W}\mathrm{Tors}(X,G)\simeq \cHom(X, BG),
\]
 all of which give isomorphisms on $\DD^i$ and equivalences on $\pi^0$. Remark \ref{predetectweakh} thus shows that 
\[
\uline{\ddel}( \uline{ \CC^{\bt}_{\et}(X,G)})\simeq  \bar{W}\uline{\ddel( \uline{ \CC^{\bt}_{\et}(X,G)})}\simeq \cHom(X, BG).
\]
\end{proof}

\begin{remark}
 If $G= \GL_r$, this gives us a construction for derived moduli of rank $r$ vector bundles on $X$. If instead $G= \SL_r$, we get derived moduli of determinant $1$, rank $r$ vector bundles. 
\end{remark}

\section{Moduli from quasi-comonoids}\label{qmmoduli}

Although cosimplicial groups can be used to construct derived moduli in all characteristics for many problems, they are insufficiently flexible to arise  in the generality we need. Instead, we use the quasi-comonoids introduced in \cite{monad}

\subsection{Quasi-comonoids}

The following is a special case of \cite{monad} Lemma \ref{monad-qmlemma}:

\begin{definition}\label{qmdef}
Define a quasi-comonoid $E$ to consist of sets $E^n$ for $n \in \N_0$,  together with maps 
$$
\begin{matrix}
\pd^i:E^n \to E^{n+1} & 1\le i \le n\\
\sigma^i:E^{n}\to E^{n-1} &0 \le i <n,
\end{matrix}
$$
an associative product $*:E^m \by E^n \to E^{m+n}$, with identity $1 \in E^0$, such that:
\begin{enumerate}
\item\label{sdc1} $\pd^j\pd^i=\pd^i\pd^{j-1}\quad i<j$.
\item $\sigma^j\sigma^i=\sigma^i\sigma^{j+1} \quad i \le j$.
\item \label{sdc3}
$
\sigma^j\pd^i=\left\{\begin{matrix}
                        \pd^i\sigma^{j-1} & i<j \\
                        \id             & i=j,\,i=j+1 \\
                        \pd^{i-1}\sigma^j & i >j+1
                        \end{matrix} \right. .
$
\item\label{sdc4} $\pd^i(e)*f=\pd^i(e*f)$.
\item $e*\pd^i(f)=\pd^{i+m}(e*f)$, for $e \in E^m$.
\item $\sigma^i(e)*f=\sigma^i(e*f)$.
\item\label{sdc7} $e*\sigma^i(f)=\sigma^{i+m}(e*f)$, for $e \in E^m$.
\end{enumerate}

Denote the category of quasi-comonoids by $QM^*$.
\end{definition}

\begin{example}\label{egcfcgp}
Given a cosimplicial group $G$ (or even a cosimplicial monoid $G$), there is an  associated quasi-comonoid $\cE(G)$ given  by $\cE(G)^n= G^n$, with identity $1 \in G^0$, operations $\pd^i_{\cE(G)}=\pd^i_G, \sigma^i_{\cE(G)}=\sigma^i_G$, and Alexander-Whitney product
$$
g*h= ((\pd_G^{m+1})^ng)\cdot ((\pd^0_G)^mh),
$$  
for $g \in G^m,\, h \in G^n$.
\end{example}

\subsubsection{Maurer--Cartan}

We now construct a Maurer--Cartan functor analogous to the one for  cosimplicial groups.

\begin{definition}\label{mcdefqm1}
Define  $\mc: QM^*(\Set) \to \Set$ by 
$$
\mc(E) = \{\omega \in E^1\,:\, \sigma^0\omega =1, \,\pd^1\omega = \omega *\omega\}.
$$
\end{definition}

Now let $QM^*(\bS)$ be the category of simplicial objects in $QM^*$. 
The following is \cite{monad} Definition \ref{monad-mcdefqm}:
\begin{definition}\label{mcdefqm}
Define $\mmc: QM^*(\bS) \to \bS$ by 
$$
\mmc(E) \subset \prod_{n\ge 0} (E^{n+1})^{I^n}
$$
(where $I= \Delta^1 \in \bS$), consisting of those $\underline{\omega}$ satisfying: 
 \begin{eqnarray*}
\omega_m(s_1,\ldots, s_m)*\omega_n(t_1,\ldots, t_n)&=&\omega_{m+n+1}(s_1,\ldots, s_m,0,t_1,\ldots, t_n);\\ 
\pd^i\omega_n(t_1,\ldots,t_n)&=&\omega_{n+1}(t_1, \ldots,t_{i-1},1,t_i,\ldots,t_n);\\ 
\sigma^i\omega_n(t_1,\ldots,t_n)&=&\omega_{n-1}(t_1, \ldots,t_{i-1},\min\{t_i,t_{i+1}\},t_{i+2},\ldots, t_n);\\
\sigma^0\omega_n(t_1,\ldots,t_n)&=&\omega_{n-1}(t_2, \ldots,t_n);\\
\sigma^{n}\omega_n(t_1,\ldots,t_n)&=&\omega_{n-1}(t_1, \ldots,t_{n-1}),\\ 
\sigma^0\omega_0&=&1.
\end{eqnarray*}

Define $\mc :QM^*(\bS) \to \Set$ by $\mc(E)=\mmc(E)_0$, noting that this agrees with Definition \ref{mcdefqm1} when $E \in QM^*(\Set)$. 
\end{definition}

\begin{definition}\label{qmsmodel}
 Given a quasi-comonoid $E$, define the $n$th matching object $M^nE$ to be the set
\[
M^nE=\{(e_0,e_1,\ldots,e_{n-1}) \in (E^{n-1})^{n}\,|\, \sigma^ie_j=\sigma^{j-1}e_i\, \forall i<j \}.
\]
The Reedy matching map $E^n \to M^nW$ sends $e$ to $(\sigma^0e,\sigma^1e,\ldots,\sigma^{n-1}e) $.

Then \cite{monad} Lemma \ref{monad-qmsmodel} gives  a  model structure on $QM^*(\bS)$  in which a morphism $E \to F$ is a (trivial) fibration whenever the canonical maps
\[
 E^n \to F^n\by_{M^nF} M^nE
\]
are (trivial) fibrations in $\bS$ for all $n \ge 0$. 
\end{definition}

\begin{lemma}\label{qmmcrQ}
 If  $f:E \to F$ is a  (trivial) fibration in $QM^*(\bS)$, then the map
\[
\mmc(f): \mmc(E) \to \mmc(F)
\]
is  a (trivial) fibration in $\bS$. In particular, if $f:E \to F$ is a trivial fibration, then $\mc(f)$ is surjective.
\end{lemma}
\begin{proof}
This is a direct consequence of \cite{monad} Corollary \ref{monad-mcquillen}, which shows that $\mmc$ is a right Quillen functor.       
\end{proof}

\begin{lemma}\label{abqm}
There is an equivalence between the category of abelian group objects in $QM^*$, and the category $c\Ab$ of cosimplicial complexes of abelian groups.
\end{lemma}
\begin{proof}
 This is \cite{monad} Lemma \ref{monad-abqm}. The equivalence is given by the functor $\cE$ of Example \ref{egcfcgp}.       
\end{proof}

\begin{lemma}\label{cotcoho}
For a cosimplicial simplicial abelian group  $A$, there are canonical isomorphisms
$$
\pi_n\mmc(\cE(A))  \cong \H_{n-1}(\Tot^{\Pi}\sigma^{\ge 1}N^sN_cA),
$$
 where   $\sigma^{\ge 1}$ denotes brutal truncation in cochain degrees $\ge 1$, $\Tot^{\Pi}$ is the product total functor of Definition \ref{totprod}, and $N^s$ and $N_c$ are the normalisation functors of Definitions \ref{normdef} and \ref{normcdef}. 
\end{lemma}
\begin{proof}
 This is \cite{monad} Proposition \ref{monad-cotcoho}.       
\end{proof}

\subsubsection{The gauge action}

\begin{definition}\label{gaugedefqm}
For $E \in QM^*(\bS)$, let $(E^0)^{\by}_n \subset E^0_n$ be the submonoid of invertible elements. There is then an 
 action of the simplicial group $(E^0)^{\by}$ on the simplicial set $\mmc(E)$, called the \emph{gauge action}, and  given by setting 
$$
(g\star \omega)_n=  (\sigma_0)^ng * \omega_n *(\sigma_0)^ng^{-1},
$$
as in \cite{htpy} Definition \ref{htpy-defdef}, with $(\sigma_0)^n$ denoting the canonical map $(E^0) \to (E^0)^{\Delta^n}$. 
\end{definition}

\begin{definition}\label{deldefqm}
 Given $E \in QM^*(\bS)$, define the Deligne groupoid by  
$
 \Del(E):= [\mc(E)/ (E^0_0)^{\by}]
$
In other words, $\Del(E)$ has  objects $\mc(E)$, and morphisms from $\omega$ to $\omega'$ consist of $\{g \in (E^0_0)^{\by}\,:\, g*\omega=\omega'*g\}$. 

Define the derived Deligne groupoid to be the simplicial object in groupoids given by $\uline{\Del}(E):= [\mmc(E)/(E^0)^{\by}] $, so $\Del(E)= \uline{\Del}(E)_0$. 

Define  the simplicial sets $\ddel(E), \uline{\ddel}(E)\in \bS$ to be the nerves  $B\Del(E)$ and $\bar{W}\uline{\Del}(E)$, respectively.
\end{definition}

\begin{lemma}\label{cot3}
If $A$ is a simplicial cosimplicial abelian group, then 
$$
\pi_n\uline{\ddel}(\cE(A))  \cong \H_{n-1}(\Tot^{\Pi}N^sN_cA),
$$
whereas $\pi_1\ddel(\cE(A))\cong  \H^0(A_0)$, with 
\[
 \pi_0\ddel(\cE(A))\cong \z_{-1}(\Tot^{\Pi}N^sN_cA)/ d_c(A_0^0).
\]
\end{lemma}
\begin{proof}
 This is \cite{monad} Proposition \ref{monad-cot3}.       
\end{proof}

\subsubsection{Comparison with cosimplicial groups}

\begin{lemma}\label{cfmc}\label{cfdef}
Given a simplicial cosimplicial group $G$, and associated simplicial quasi-comonoid $\cE(G)$ as in Example \ref{egcfcgp}, there are $G^0$-equivariant weak equivalences
$$
\mmc(\cE(G))\simeq \mmc(G)
$$
and hence weak equivalences
$$ 
\ddel(\cE(G))\simeq \ddel(G)
$$
in $\bS$, functorial in  objects $G \in cs\Gp$. Here, the functors $\mmc$ on the left and right are those from Definitions \ref{mcdefqm} and \ref{mcdefgp} respectively, while the functors $\ddel$ are those from Definitions \ref{deldefqm} and \ref{deldefcgp}.
\end{lemma}
\begin{proof}
 This is \cite{monad} Propositions \ref{monad-cfmc} and \ref{monad-cfdef}.
\end{proof}

\subsection{Constructing quasi-comonoids}

\subsubsection{Monads}\label{monads}

\begin{definition}
A monad (or triple) on a category $\cB$ is a monoid in the category of endofunctors of $\cB$ (with the monoidal structure given by composition of functors). 
\end{definition}

\begin{example}
 Given an     adjunction
$$
\xymatrix@1{\cD \ar@<1ex>[r]^U_{\top} & \cB \ar@<1ex>[l]^F}
$$
with unit $\eta:\id \to UF$ and co-unit $\vareps:FU \to \id$, the associated monad on $\cB$ is given by    $\top=UF$, with unit $\eta: \id \to \top$ and multiplication $\mu:= U\vareps F: \top^2 \to \top$.
\end{example}

\begin{definition}
 
Given a monad $(\top, \eta, \mu)$ on a category $\cB$, define the category of $\top$-algebras, $\cB^{\top}$, to have objects
$$
\top B \xra{\theta} B,
$$
such that $\theta\circ \eta_B=\id$ and $\theta \circ \top \theta= \theta \circ  \mu_B: \top^2 B \to B$.
        
A morphism 
$$
g: ( \top B_1 \xra{\theta} B_1 ) \to  (\top B_2 \xra{\phi} B_2)  
$$      
of $\top$-algebras is a morphism $g:B_1 \to B_2$ in $\cB$ such that $\phi\circ \top g= g \circ \theta$.
       \end{definition}

\begin{example}
Let $\top:=\Symm_R$ be the symmetric functor on  $\Mod_R$, with unit $\eta_M: M \to \top M$ given by the inclusion of degree $1$ monomials, and $\mu_M: \top^2M \to \top M$ given by expanding out polynomials of polynomials.  Then $(\Mod_R)^{\top}$ is equivalent to the category of unital commutative $R$-algebras.
       \end{example}

Given a monad  $(\top, \mu, \eta)$ on a category $\cB$, and an object $B \in \cB$,  there is a  quasi-comonoid $E(B)$ given by
$$
E^n(B)= \Hom_{\cB}(\top^n B, B)
$$
in $(\Set, \by)$, 
with product
$g*h=g\circ\top^n h$, and  for $g \in E^n(B)$,
\begin{eqnarray*}
\pd^i(g) &=& g \circ \top^{i-1}\mu_{\top^{n-i}B}\\
\sigma^i(g) &=& g \circ \top^{i}\eta_{\top^{n-i-1}B}.
\end{eqnarray*}

Note that these constructions also all work for a comonad $(\bot, \Delta, \vareps)$, by contravariance. There is even a generalisation to bialgebras for a distributive monad-comonad pair: see \cite{monad} Proposition \ref{monad-enrichtopbot}.

\begin{lemma}\label{monadmc} Given an object $B \in \cB$, the set of $\top$-algebra structures on $B$ is
$
\mc(E(B))
$, while $\Del(E(B))$ is equivalent to the groupoid of $\top$-algebras overlying $B$.
\end{lemma}
\begin{proof}
This follows immediately from the explicit description in Definition \ref{mcdefqm1}.
\end{proof}

\subsubsection{Diagrams}\label{diagrams} 

\begin{definition}
 Given a category $\cB$ equipped with a monad $\top$, together with  $K \in \bS$ and a map $B:K_0 \to \Ob \cB$,
define the quasi-comonoid $E_K(B)$ 
 by
$$
E^n(B/K)= \prod_{x \in K_n}\Hom_{\cB}(\top^n B((\pd_0)^{n}x), B((\pd_1)^{n}x)), 
$$ 
with operations 
\begin{eqnarray*}
\pd^i(e)(x)&:=& e(\pd_i x)\circ \top^{i-1}\mu_{\top^{n-i}B((\pd_0)^{n+1}x)}\\
\sigma^j(e)(y)&:=& e(\sigma_j y)\circ \top^{i}\eta_{\top^{n-i-1}B(\pd_0)^{n-1}x)},\\ 
(f*e)(z)&:=& f((\pd_{m+1})^nz)\circ \top^m e((\pd_{0})^mz),
\end{eqnarray*}
for $f \in E^m(B/K), e \in E^n(B/K)$.
\end{definition}

\begin{definition}
Given a category $\cB$ equipped with a monad $\top$, together with  
a small category $\bI$ and a map $B:\Ob \bI \to \Ob\cB$, 
define the quasi-comonoid $E(B/\bI)$ 
 by
\[
 E(B/\bI):= E(B/B\bI),
\]
where $B\bI$ is the nerve of $\bI$.
\end{definition}

\begin{lemma}\label{sdcdiagramsub} 
Given $\cB, \top, \bI$ and   $B:\Ob \bI \to \Ob\cB$ as above, 
\[
 \mc(E(B/\bI))
\]
is isomorphic to the set of functors $\bD:\bI \to \cB^{\top}$ with $U\bD(i)= B(i)$ for all $i \in \bI$, where $U: \cB^{\top} \to \cB$ is the forgetful functor.

 Meanwhile $\Del(E(B))$ is equivalent to the groupoid of diagrams $\bD:\bI \to \cB^{\top}$ with $\bD(i)$ overlying $B(i)$ for all $i \in \bI$.
\end{lemma}
\begin{proof}
This is \cite{monad} Lemma \ref{monad-sdcdiagramsub}. Given $\omega \in \mc(E(B/\bI))$, the algebra structure $\top B(i) \to B(i)$ is given by $\omega(i \xra{\id} i) \in \Hom_{\cB}(\top^n B(i), B(i))$, while the morphism $\bD(f): \bD(i) \to \bD(j)$ is given by $ \omega(i \xra{f} j) \circ \eta_{B(i)}\in \Hom_{\cB}( B(i), B(j)) $.
\end{proof}

\begin{corollary}
 Take a category $\cB$ equipped with a monad $\top$, together with  
a small category $\bI$ and a subcategory $\bJ$. Assume that we have a functor $\bF: \bJ \to \cB^{\top}$, and a map $B: \Ob \bI \to \Ob \cB$ extending $U\Ob\bF: \Ob J \to \Ob \cB$.

For $\omega_{\bF} \in \mc(E(B|_{\bJ}/\bJ))$ corresponding to $\bF$ in Lemma \ref{sdcdiagramsub}, we can form a quasi-comonoid $E$ by 
\[
E^n:=  E^n(B/\bI)\by_{E^n(B|_{\bJ}/\bJ)}\{\overbrace{\omega_{\bF} * \omega_{\bF} * \ldots * \omega_{\bF} }^n\}.
\]

Then 
$
 \mc(E)
$
is isomorphic to 
the set of functors $\bD:\bI \to \cB^{\top}$ with $U\bD(i)= B(i)$ for all $i \in \bI$ and $\bD|_{\bJ}= \bF$, while $\Del(E)$ is the groupoid of such functors.
\end{corollary}
\begin{proof}
 This follows immediately from the observation that $\mc$ preserves limits.
\end{proof}

\begin{example}\label{morphismsqm}
 The main applications of these results are to moduli of morphisms. In that case, $\bI$ is the category $0 \to 1$,  $B\bI\cong \Delta^1$, and we define $E(B_0,B_1):= E(B/\bI)$, where $B: \Ob \bI \to \cB$ is given by $B(i):= B_i$. 

The category $\bJ$ will be $\emptyset$, $\{0\}$, $\{1\}$ or $\{0,1\}$ depending on which endpoints we wish to fix (if any), so the quasi-comonoid $E$ of the corollary is the fibre of $E(B_0,B_1)\to \prod_{j \in \bJ} E(B_j)$ over powers of $\omega_j$.
\end{example}

\begin{remark}
 In \cite{monad} Definition \ref{monad-MCdef} the construction $E(B/\bI)$ (for a simplicial category $\cB$ equipped with a monad $\top$) is used to extend the simplicial set $\mmc(E)$ to a bisimplicial  set $\cMC(E)$. Explicitly, $\cMC(E)_n$ is given by taking $\bI$ to be the category associated to the poset $[0,n]$,  and setting
\[
 \cMC(E)_n:= \coprod_{B: [0,n] \to \Ob \C }\mmc(E(B/\on)).
\]

By \cite{monad} Proposition \ref{monad-MCsegal}, $\cMC(E)$ is a Segal space whenever $\C$ satisfies suitable fibrancy conditions. Segal spaces are a model for $\infty$-categories (whereas simplicial sets are a model for $\infty$-groupoids), and \cite{monad} Propositions \ref{monad-MCNequiv} and \ref{monad-ddefmap} show that $\Del(E)$ is effectively the core of  $\cMC(E)$. 
This means that for   all of our moduli constructions based on quasi-comonoids in \S \ref{qmegs}, we  could construct derived moduli as $\infty$-categories, rather than just $\infty$-groupoids.
\end{remark}

\begin{definition}\label{fibrantmor}
 Say that an ordered pair $B,B'$ of objects in $\cB$ induces fibrant quasi-descent data if  $E(B)$ and $E(B')$ are fibrant simplicial quasi-comonoids, and
the matching maps $\HHom_{\cB}(\top^nB, B') \to M^n\HHom{\cB}(\top^{\bt}B, B')$  are also Kan fibrations for all $n \ge 0$. 
\end{definition}
Note that this is the same as regarding $\cB$  as a simplicial quasi-descent datum  (in the sense of \cite{monad} Proposition \ref{monad-enrichtop}), then restricting to  
objects $B,B'$, discarding morphisms $B' \to B$, and requiring that the resulting  simplicial quasi-descent datum $\cD$ be fibrant. 

\begin{lemma}\label{loopmorsub}
Take objects  $D_0,D_1 \in \cB^{\top}$, with $UD_0, UD_1 \in \cB$ inducing fibrant quasi-descent data.
Then there is a natural weak equivalence
$$
\mmc(UD_0, UD_1)\by_{\mmc( E(UD_0))\by \mmc(E(UD_1))}\{(D_0,D_1)\} \simeq \Tot \HHom_{\cB}( \top^{\bt}UD_0, UD_1 ),
$$
where $\Tot:c\bS \to \bS$ the total space functor of \cite{sht} Ch. VIII, and the cosimplicial structure on $\HHom_{\cB}( \top^{\bt}UD_0, UD_1 )$ is the usual cotriple resolution (\cite{W} \S 8.7)
 defined via the isomorphisms
\[
 \HHom_{\cB}( \top^nUD_0, UD_1 )\cong \HHom_{\cB^{\top}}( (FU)^{n+1}D_0, D_1).
\]
\end{lemma}
\begin{proof}
 This is \cite{monad} Proposition \ref{monad-loopmorsub}.
\end{proof}

\subsubsection{A bar construction}

Now fix a simplicial category $\cB$ equipped with a monad $(\top, \mu,\eta)$ respecting the simplicial structure.
Assume that $\cB^{\top}$ is a cocomplete simplicial category, equipped with a functor $\bS \by \cB^{\top} \xra{\ten} \cB^{\top}$ for which 
\[
\Hom_{\cB^{\top}}(K\ten D, D')\cong \Hom_{\bS}(K, \HHom_{\cB^{\top}}(D,D')).
\]

Write $F: \cB \to \cB^{\top}$ for  the free algebra functor sending $M$ to the map $(\mu_M: \top^2 M \to \top M)$, and $U: \cB^{\top} \to \cB$ for the forgetful functor sending $(\theta: \top M \to M) $ to $M$. In particular, $UF= \top$.

\begin{definition}\label{betaF}
Given  an object  $M \in \cB$ and an element  $\omega \in \mc(E(B))$ (for $\mc$ as in Definition \ref{mcdefqm}), define $\beta_F^*(M) \in \cB^{\top}$ by the property that
\[
\Hom(\beta_F^*(M), D) \subset \prod_{q\ge 0} \Hom_{\cB^{\top}}( (\Delta^1)^q\ten F\top^qM, D)
\]
consists of $\uline{\phi}$ satisfying 
\begin{eqnarray*}
\phi_q(t_1, \ldots, t_{i-1}, 1, t_i, \ldots t_{q-1})&=& \pd^i\phi_{q-1}(t_1, \ldots, t_{q-1})\\
\phi_q( t_1, \ldots, t_{i-1}, 0, t_i, \ldots t_{q-1})&=& \phi_{i-1}( t_1, \ldots, t_{i-1})\circ F\top^{i-1}\omega_{q-i}(t_i, \ldots t_{q-1})\\
\sigma^i\phi_q(t_1, \ldots, t_{q})&=&\phi_{q-1}( t_1, \ldots, \min(t_{i}, t_{i+1}), \ldots, t_{q})\quad 1\le i \le q\\
\sigma^{q}\phi_q(t_1, \ldots, t_{q})&=&\phi_{q-1}( t_1, \ldots, t_{q-1}),
\end{eqnarray*}
where $\pd^0= \lambda^*_{\top^{q-1}M}$, $\pd^i= F\top^{i-1}\mu_{\top^{q-i}M}^*$, and
$\sigma^i= F\top^{i-1}\eta_{\top^{q-i}M}^*$.

Beware that, unlike Definition \ref{mcdefqm}, there is no relation for $\sigma^0$.
\end{definition}

\begin{remarks}
Note that $\beta_F^*(M, \omega)$  is a $\top$-algebra $R$ generated by $\phi_q(\top^qM\ten I^q)$, subject to various conditions, all of which are linear on generators except for
\[
\phi_q(1, t_1, \ldots, t_{q-1})= U\vareps_R \circ \top\phi_{q-1}(t_1, \ldots, t_{q-1}).
\]
Also observe that $\beta^*_F$ defines a functor $\Del(E(M)) \to \cB^{\top}$.

This construction is inspired by Lada's bar construction in \cite{loop}, which uses similar data to define a bar construction as an object of $\cB$, then shows that it carries a canonical $\top$-algebra structure. However, Lada's construction only applies when $\top$ is an operad and $\cB$ is the category of topological spaces, since it uses special properties of both. Beware that although similar expressions arise in both $\beta_F^*$ and in Lada's bar construction, they are not directly comparable.
\end{remarks}

\begin{proposition}\label{betalemma}
Fix $\cB, \top, M,\omega$ as above, and take any $D \in \cB^{\top}$ for which the pair $M, UD$ induces fibrant quasi-descent data (Definition \ref{fibrantmor}). Then
there is a functorial weak equivalence between $\HHom_{\cB^{\top}}( \beta_F^*(M, \omega),D)$ and the fibre of 
\[
 \mmc(E(M,UD))\to \mmc(E(M)) \by \mmc(E(UD))
\]
over $(\omega, \mu_D )$, where $D= (\mu_D: \top UD \to UD)$, and $E(M,UD)$ is the quasi-comonoid defined in Example \ref{morphismsqm}.
\end{proposition}
\begin{proof}
We adapt the proof of \cite{monad} Proposition \ref{monad-loopmorsub}. In the notation of \cite{monad}, we have $\cD \in sQ\Dat_1$ given by $\cD(0,0) = E(M)$, $\cD(1,1)= E(UD)$, $\cD(1,0)=\emptyset$  and $\cD(0,1)^n= \HHom_{\cB}(\top^nM, UD)$. 
Using further notation from \cite{monad}, we  can define $\cA \in sQ\Dat_1$ by 
\[
 \cA:= (\Xi \by \alg^*\oI)\cup_{\Xi\by\{1\}}(\alg^*\oO\by \{1\}),
\]
and we then have 
\[
 \mmc(E(M,UD))\by_{ \mmc(E(UD)) } \mc(E(D)_0) \cong \HHom_{sQ\Dat_1}( \cA, \cD).
\]

Meanwhile, we can construct another object $\C \in sQ\Dat_1$ with  $\C(0,0)= \Xi$ (so $\C(0,0)^n=I^{n-1}$), $\C(1,1)=\bt$, $\C(1,0)= \emptyset$ and $\C(0,1)^n=I^n$, where $I= \Delta^1$. The multiplication operation $\C(0,0)^m\by \C(0,1)^n\to \C(0,1)^{mn}$ is given by   $I^{m-1}\by I^n \to I^{m-1} \by \{0\}\by I^n$. The operations on $\C(0,1)$ are 
\begin{eqnarray*}
\pd^i(t_1, \ldots, t_n) &=&  (t_1, \ldots, t_{i-1}, 1, t_i, \ldots t_{n}) \quad 1\le i \le n\\
    \sigma^i(t_1, \ldots, t_{n})&=&( t_1, \ldots, \min(t_{i}, t_{i+1}), \ldots, t_{n})     \quad 0\le i < n.                     
 \end{eqnarray*}
The multiplication operation $\C(0,1)^m \by \C(1,1)^n \to \C(0,1)^{m+n}$ is given by
\[
 (t_1, \ldots, t_m, \bt) \mapsto (t_1, \ldots, t_m,1, 1, \ldots, 1).
\]

Now, the inclusion $\Xi = \C(0,0) \to \C$ gives a Kan fibration
\begin{eqnarray*}
 \HHom_{sQ\Dat_1}( \C, \cD) \to&& \HHom_{sQM^*(\bS)}( \Xi, E(M))\by\HHom(\alg^*\oO, E(UD)) \\
 &=& \mmc(E(M)) \by \mc(E(D)_0)
\end{eqnarray*}
whose fibre over $(\omega,D)$ is
$  \HHom_{\cB^{\top}}( \beta_F^*(M, \omega),D)$.
It therefore suffices to show that 
\[
 \HHom_{sQ\Dat_1}( \C, \cD)\by_{\mc(E(UD)_0)}\{D\} \simeq \mmc(E(M,UD))\by_{ \mc(E(UD)_0) } \{D\}
\]
as fibrant objects over $\mmc(E(M))$, since taking the fibre over $\omega$ yields the required result.

Now, $ \alg^*\oI \in sQ\Dat_1$ is given by $(\alg^*\oI)(i,j)^n=\bt$ for all $n$ and for all $0 \le i\le j\le 1$, while  $(\alg^*\oI)(1,0)= \emptyset$. Thus the unique maps $\C \to \alg^*\oI$ and $\cA \to \alg^*\oI$ are both weak equivalences, and hence both are cofibrant replacements for $\alg^*\oI$ in the comma category $ (\alg^*\oO\by \{1\})\da sQ\Dat_1$. 

If we regard $\cD$ as an object of  $(\alg^*\oO\by \{1\})\da sQ\Dat_1 $ via the morphism $D: (\alg^*\oO\by \{1\})\to \cD$, this means that
\begin{eqnarray*}
 \HHom_{sQ\Dat_1}( \C, \cD)\by_{\mc(E(UD)_0)}\{D\} &\simeq& \HHom_{(\alg^*\oO\by \{1\})\da sQ\Dat_1}( \C,  \cD)\\
  &\simeq& \oR \HHom_{(\alg^*\oO\by \{1\})\da sQ\Dat_1}(\alg^*\oI, \cD) \\
 &\simeq&  \HHom_{(\alg^*\oO\by \{1\})\da sQ\Dat_1}( \cA, \cD) \\
&\simeq& \mmc(E(M,UD))\by_{ \mmc(E(UD)) }\{D\},
\end{eqnarray*}
which completes the proof.
\end{proof}

\begin{definition}
 Define $\bot_{\bt}: \cB^{\top}\to \cB^{\top}$ by the property that
\[
 \HHom_{\cB^{\top}}(\bot_{\bt}A,D) \cong \Tot  \HHom_{\cB^{\top}}( (FU)^{\bt+1}A,D).
\]
Explicitly, we form the simplicial diagram $n \mapsto (FU)^{n+1}A$ in $\cB^{\top}$ (the cotriple resolution), then let $\bot_{\bt}A$ be the coend 
\[
 \int^{n \in \Delta} \Delta^n \ten (FU)^{n+1}.
\]
\end{definition}

\begin{corollary}\label{cfbeta}
 If $A,D \in \cB^{\top}$, with $UA, UD$  inducing fibrant quasi-descent data, then there are functorial weak equivalences
\[
  \HHom_{\cB^{\top}}(\beta^*_F(UA,\mu_A), D) \simeq \HHom_{\cB^{\top}}(\bot_{\bt}A, D).
\]
\end{corollary}
\begin{proof}
 This just combines Proposition \ref{loopmorsub} with Proposition \ref{betalemma}.
\end{proof}

\subsection{Linear quasi-comonoids}\label{linearqm}

\begin{definition}
 Say that a quasi-comonoid $A$ is linear if each $A^n$ is an abelian group, with the operations $\pd^i, \sigma^i$ being linear, and $*: A^m \by A^n \to A^{m+n}$ bilinear. As explained in \cite{monad} \S \ref{monad-linearsn}, this corresponds to working with the  monoidal structure $\ten$ rather than $\by$.

Denote the category of linear quasi-comonoids by $QM^*(\Ab, \ten)$, and the category of simplicial objects in $QM^*(\Ab, \ten)$ by $QM^*(s\Ab,\ten)$
\end{definition}

\begin{example}
The quasi-comonoid $E(B)$ constructed in \S \ref{monads} is a linear quasi-comonoid whenever $\cB$ is a pre-additive category and $\top$ is an additive functor.
\end{example}

\begin{lemma}\label{normlinear}
 Given $A \in QM^*(\Ab,\ten)$, the normalisation $NA$ has the natural structure of a  (not necessarily commutative) DG ring.
\end{lemma}
\begin{proof}
 The normalisation is given by $(NA)^n = A^n \cap \bigcap_{i=0}^{n-1}\ker \sigma^i$. If we write $0_m$ for the additive identity in $A^m$, then define $\pd^0, \pd^{n+1}: A^n \to A^{n+1}$ by $\pd^0a:= 0_1*a$, $\pd^{n+1}a:=a*0_1$. This makes $A$ into a cosimplicial complex, so $NA$ is a chain complex, with $d:N^nA \to N^{n+1}A$ given by 
\[
da:= \sum_{i=0}^{n+1} (-1)^i\pd^i. 
\]

For $a \in N^mA$ and $b \in N^nA$, the product $a*b$ lies in $N^{m+n}B$, and $d(a*b)= (da)*b + (-1)^ma*(db)$, so $NA$ is indeed a DG ring, with unit $1 \in N^0A$.
\end{proof}

\begin{definition}\label{denormdefalg}
Given a DG ring $B$ in non-negative cochain degrees with multiplication denoted by $\wedge$, we now define the cosimplicial ring $DB$. As a cosimplicial complex, $DB$ is given by the formula of Definition \ref{denormdef}, with multiplication
 \[
 (\pd^Ia)\cdot (\pd^J b):= \left\{ \begin{matrix} \pd^{I\cap J}(-1)^{(J\backslash I, I \backslash J)} a \wedge b & a\in L^{|J\backslash I|}, b\in L^{|I\backslash J|},\\ 0 & \text{ otherwise},\end{matrix} \right.
 \]
for $(-1)^{(S,T)}$ defined as in Definition \ref{denormdef}.
\end{definition}

\begin{lemma}\label{denormalglemma}
  Given $A \in QM^*(\Ab,\ten)$, the quasi-comonoid $\cE(DNA)$ is isomorphic to $A$ 
(for $\cE$ as in Example \ref{egcfcgp}, regarding $DNA$ as a cosimplicial multiplicative monoid).
\end{lemma}
\begin{proof}
 Since $DNA \cong A$ as a cosimplicial complex, it is automatic that we have isomorphisms $\cE(DNA)^n \cong A^n$, compatible with the structural operations $\sigma^i, \pd^i$, as well as with the operations $a \mapsto 0_1*a$ and $a \mapsto a*0_1$ (corresponding to the additional operations $\pd^0, \pd^{n+1}$). Now, since Eilenberg--Zilber shuffles are  left inverse to Alexander--Whitney, it follows that for $a, b \in NA$,
$a*_{\cE}b= a*b$. Since $A$ is spanned by elements of the form $\pd^Ia$ for $a\in NA$, the defining equations of a quasi-comonoid ensure that $a*_{\cE}b= a*b$ for all $a, b \in A$.
\end{proof}

Now, any DG $R$-algebra $B$ has an underlying DGLA over $R$, with bracket $[a,b] = ab- (-1)^{\deg \deg b}ba$. If  a non-unital DG  $\Q$-algebra $B$ is pro-nilpotent (i.e. $B \cong \Lim_n B/(B)^n$), we can thus define a group $\exp(DN^{>0}B)$ as in Lemma \ref{cfexp}.

\begin{lemma}\label{expnilpalg}
 For any non-unital pro-nilpotent DG $\Q$-algebra $B$,  there is a canonical isomorphism 
\[
 \exp(DB) \cong 1+ DB
\]
of groups.
\end{lemma}
\begin{proof}
 Since $B$ is pro-nilpotent, $DB$ is as well.  The isomorphism is given by evaluating the exponential in the ring $DB$, with inverse given by $\log$.
\end{proof}

\begin{definition}
 Given $A \in QM^*(\Ab,\ten)$, make  $(A^0)^{\by}$ into a gauge on the DGLA underlying $NA$ by giving it the obvious adjoint action, and with $D: (A^0)^{\by} \to N^1A$ given by  $Da= da\cdot a^{-1}$.
\end{definition}

\begin{lemma}\label{expalg}
 For $A \in QM^*(\Ab,\ten)$ with $A^0$ a $\Q$-algebra, there is a canonical isomorphism
\[
 D(\exp(NA), (A^0)^{\by})\cong (DNA)\by_{A^0}(A^0)^{\by}
\]
of cosimplicial groups.
\end{lemma}
\begin{proof}
By applying Lemma \ref{expnilpalg} to $N^{>0}A$,  
\[
 D(\exp(NA), (A^0)^{\by})^n \cong (A^0)^{\by} \ltimes (1+ \ker(A^n \to A^0))= (A^0)^{\by} \ltimes (A^n\by_{A^0}1)\cong A^n\by_{A^0} (A^0)^{\by},
\]
and this is automatically compatible with the cosimplicial operations in higher degrees. A short calculation show that it is also compatible with $\pd^0, \pd^1: A^0 \to A^1$.
\end{proof}

\begin{definition}
 Given a simplicial DGLA $L$, define the simplicial set $\mmc(L)$ as follows. first, for any simplicial set $K$, define the simplicial DGLA $L^K$ by $(L^K)^n:= (L^n)^K$, defined with the formula of Definition \ref{salgstr}. Then $\mmc(L)$ is given by 
\[
 \mmc(L)_n:= \mc(\Tot^{\Pi}N^s(L^{\Delta^n})),
\]
 where the normalisation $N^s(L)$ has a bracket $N^s(L)^i_m \by N^s(L)^j_n\to N^s(L)^{i+j}_{m+n}$ given by the Eilenberg--Zilber shuffle product (\cite{W} 8.5.4).
\end{definition}

\begin{proposition}\label{mcqmdgla}
For  $A \in QM^*(s\Ab,\ten)$, there is a canonical $(A^0)^{\by}$-equivariant weak equivalence
\[
 \mmc(A) \simeq \mmc(NA),
\]
of simplicial sets, and hence a canonical  weak equivalence
\[
 \ddel(A) \simeq \bar{W}[\mmc(NA)/ (A^0)^{\by}].
\]
\end{proposition}
\begin{proof}
By combining Corollary \ref{cfexp} with  Lemma \ref{expalg}, we have an $(A^0)^{\by}$-equivariant isomorphism
\[
 \mmc(NA) \cong \mmc((DNA)\by_{A^0}(A^0)^{\by}).
\]
Now, Lemma \ref{cfmc} combines with Lemma \ref{denormalglemma} to give an $(A^0)^{\by}$-equivariant weak equivalence
\[
 \mmc((DNA)\by_{A^0}(A^0)^{\by}) \simeq \mmc( A\by_{A^0}(A^0)^{\by}).
\]
Since $\mmc(A\by_{A^0}(A^0)^{\by} )= \mmc(A)$ and $\ddel(A)= \bar{W}[\mmc(A)/ (A^0)^{\by}]$, this completes the proof.
\end{proof}

\subsection{Moduli functors from quasi-comonoids}

\begin{definition}\label{qmCCdef}
 Given  a homogeneous, levelwise formally smooth functor    $E: Alg_R \to QM^*$, a ring $A \in \Alg_R$, an $A$-module $M$  and $\omega\in \mc(E(A))$, define the cosimplicial $A$-module $\CC^{\bt}_{\omega}(E, M)$ by 
\[
\CC^{n}_{\omega}(E, M):=T_{\omega^n}(E^n,M)
\]
with operations on $a \in \CC^{n}_{\omega}(E, M)$ given by   
%
\begin{eqnarray*}
        \sigma^ia&=& \sigma^i_Ea\\
\pd^ia&=& \left\{\begin{matrix}   \omega*a  & i=0\\  
\pd^i_Ea & 1\le i \le n \\
a*\omega & i=n+1. \end{matrix} \right. 
\end{eqnarray*}
\end{definition}

\begin{proposition}\label{mcqmnice}
 If $E: Alg_R \to QM^*$ is a homogeneous functor, with each $E^n$ formally smooth, then the functor 
\[
 \mmc(\uline{E}):d\cN^{\flat}_R \to \bS
\]
is homogeneous and formally quasi-smooth, so
\[
 \mc(\uline{E}):d\cN^{\flat}_R \to \Set
\]
is homogeneous and pre-homotopic.
\end{proposition}
\begin{proof}
 This proceeds along the same lines as Proposition \ref{mcgpnice}. Homogeneity is automatic, as $\mmc$ preserves arbitrary limits. 

We can extend $E^n$ to a functor $E^n: d\cN_R^{\flat} \to \Set$, given by $E^n(A):= E^n(A_0)$. Then formal smoothness of $E^n$ implies that the extended $E^n$ is  pre-homotopic. It is automatically formally quasi-presmooth, as all discrete morphisms are fibrations. Thus Proposition \ref{settotop} implies that  $\uline{E}^n$ is formally smooth, and hence formally quasi-smooth.

For our next step, there is no analogue of Lemma \ref{gplevelwise} for quasi-comonoids, so we have to work a little harder. We wish to show that $\mmc(\uline{E})(A') \to \mmc(\uline{E})(A)$ is a (trivial) fibration for all (acyclic) square-zero extensions $A' \to A$. Write $C:= \H_0A$, and note that $\uline{E}(C)= E(C)$, so $\mmc(\uline{E}(C))= \mc(E(C))$ is a set. For any $\omega \in \mc(E(C))$, it thus suffices to show that the morphism $\mmc(\uline{E})(A')_{\omega} \to \mmc(\uline{E})(A)_{\omega}$ of fibres over $\omega$ is a (trivial) fibration.

Now,  on the subcategory of $d\cN^{\flat}\da C$ consisting of nilpotent extensions $B \to C$,  
define a  functor $E_{\omega}$ by $E_{\omega}^n(B) = E^n(B)\by_{E^n(C)}\omega^n$. Thus $E_{\omega}(B) \in QM^*$; since $C^{\Delta^n}=C$, we also have $\uline{E}_{\omega}(B) \in QM^*(\bS)$. Now, the crucial observation is that
\[
 \mmc(\uline{E})(B)_{\omega}= \mmc( \uline{E}_{\omega}(B)),
\]
 so by Lemma \ref{qmmcrQ}, it suffices to show that $\uline{E}_{\omega} (A') \to\uline{E}_{\omega} (A)$ is a (trivial) fibration in $QM^*(\bS)$. 

This amounts to saying that $\mu_n: \uline{E}_{\omega}^n \to M^n\uline{E}_{\omega}$ is formally quasi-smooth for all $n \ge 0$. In fact, it is formally smooth;  we prove this inductively on $n$. For $n=0$, this follows immediately from the observation above that $\uline{E}^0$ is formally smooth. If this holds up to level $n-1$, then $M^{n}\uline{E}_{\omega}$ is also formally quasi-smooth, being a pullback of formally quasi-smooth maps. Since $\uline{E}_{\omega}^n$ is formally smooth, it follows from  
  Corollary \ref{cohosmoothchar} that  $\mu_n$ will be formally smooth provided
\[
 \DD^i_{\omega}( \uline{E}^n,M) \to \DD^i_{\omega}( M^n\uline{E},M)
\]
is surjective for $i=0$, and an isomorphism for $i>0$, for all $C$-modules $M$. Now,   Lemma \ref{totcohoc} shows that $\DD^i_{\omega}( \uline{E}^n,M)= \DD^i_{\omega}(E^n,M)$, and similarly for $M^nE$. These are $0$ for $i\ne 0$, and $\DD^0_{\omega}(E^n,M) = \CC^n_{\omega}(E,M)$.
Hence we need only show that
\[
 \CC^n_{\omega}(E,M)\to M^n \CC^{\bt}_{\omega}(E,M)
\]
is surjective, which follows from Lemma \ref{gplevelwise}.
\end{proof}

\begin{corollary}\label{delqmnice}
 If $E: Alg_R \to QM^*$ is a homogeneous functor, with each $E^n$ formally smooth, then the functor 
\[
 \uline{\ddel}(\uline{E}):d\cN^{\flat}_R \to \bS
\]
is homogeneous and formally quasi-smooth, while
\[
 \ddel(\uline{E}):s\cN^{\flat}_R \to \bS
\]
is homogeneous, pre-homotopic and formally quasi-presmooth.
\end{corollary}
\begin{proof}
The proof of Proposition \ref{delgpnice} adapts, substituting Proposition \ref{mcqmnice} for  Proposition \ref{mcgpnice}.
\end{proof}

\subsubsection{Cohomology}

\begin{definition}\label{qmcoho}
Given  a homogeneous, levelwise formally smooth functor    $E: Alg_R \to QM^*$, a ring $A \in \Alg_R$, an $A$-module $M$  and $\omega\in \mc(E(A))$, define 
\[
 \H^i_{\omega}(E,M):=   \H^i \CC^{\bt}_{\omega}(E, M),    
\]
 for $\CC^{\bt}_{\omega}(E, M)$ as in Definition \ref{qmCCdef}.    
\end{definition}

The following is immediate.
\begin{lemma}
 Given  a homogeneous, levelwise formally smooth functor    $E: Alg_R \to QM^*(\bS)$, a ring $A \in \Alg_R$, an $A$-module $M$
and $\omega\in \mc(E_0(A))$, the fibre of $\mc(E(A \oplus M)) \to \mc(E(A))$ over $\omega$ is canonically isomorphic to $\mc(\CC^{n}_{\omega}(G, M))$.
\end{lemma}

\begin{lemma}\label{qmcohomc}
If $E: Alg_R \to QM^*$ is a homogeneous, levelwise formally smooth functor, with $A \in \Alg_R$ and $M$ an $A$-module,   then 
\[
 \DD^i_{\omega}(\mmc(\uline{E}),M)\cong  \DD^i_{\omega}(\mc(\uline{E}),M) \cong \left\{\begin{matrix} \H^{i+1}_{\omega}(E,M) & i >0 \\ \z^1\CC^{\bt}_{\omega}(E,M) & i=1. \end{matrix} \right.  
\]
\end{lemma}
\begin{proof}
 First, observe that there is a constant quasi-comonoid $\iota\bt$, given by the one-point set in every level. The element  $\omega$ thus defines a morphism $\omega^{\bt}:\iota\bt \to E(A)$ of quasi-comonoids, given by $\omega^n$ in level $n$, and the fibre product 
\[
  T_{\omega^{\bt}}(\uline{E},L):= \uline{E}(A\oplus L)\by_{E(A), \omega^{\bt}}\iota\bt
\]
is thus an abelian quasi-comonoid, for all $L \in d\Mod_A$. By Lemma \ref{abqm}, this corresponds to a simplicial cosimplicial abelian group, which is just given in simplicial level $n$ by
\[
  T_{\omega^{\bt}}(\uline{E},L)_n = \CC^{\bt}_{\omega}(E, (L^{\Delta^n})_0).
\]

Therefore 
\[
 T_{\omega}( \mmc(\uline{E}), L) \cong \mmc( T_{\omega^{\bt}}(\uline{E},L)),
\]
and the result is an immediate consequence of Lemma \ref{cotcoho}.
\end{proof}

\begin{lemma}\label{qmcohodel}
If $E: Alg_R \to QM^*$ is a homogeneous, levelwise formally smooth functor, with $A \in \Alg_R$ and $M$ an $A$-module,   then 
\[
 \DD^i_{\omega}(\uline{\ddel}(\uline{E}),M)\cong \DD^i_{\omega}(\ddel(\uline{E}),M) \cong\H^{i+1}_{\omega}(E,M).    
\]
\end{lemma}
\begin{proof}
 The proof of Lemma \ref{cgpcohodel} carries over.
\end{proof}

\subsection{Sheafification}\label{shfqmsn}

\begin{definition}
Given a cosimplicial diagram $\CC^{\bt}(E)$ in $QM^*$, define the quasi-comonoid $\diag \CC^{\bt}(E)$ by 
\[
 \diag \CC^{\bt}(E)^n:= \CC^n(E^n),
\]
with operations $\pd^i= \pd^i_{\CC}\pd^i_E$, $\sigma^i= \sigma^i_{\CC}\sigma^i_E$, identity $1 \in \CC^0(E)$, and multiplication
\[
 a*b:= (\pd^{m+1}_{\CC})^na *_E (\pd^0_{\CC})^mb,
\]
for $ a \in \CC^m(E^m)$, $b \in \CC^n(E^n)$.
\end{definition}

\begin{definition}\label{shfqm}
Given a levelwise formally smooth, homogeneous functor  $E: Alg_R \to QM^*$ preserving finite products, 
and some class $\oP$ of covering morphisms in $\Alg_R$, define the sheafification $E^{\sharp}$ of $E$ with respect to $\oP$ by
first defining a cosimplicial commutative $A$-algebra $(B/A)^{\bt}$ for every $\oP$-covering $A\to B$, as $(B/A)^n:= \overbrace{B\ten_{A}B\ten_{A}\ldots \ten_{A}B}^{n+1}$, then setting
\[
 E^{\sharp}(A) = \diag \LLim_{B} E((B/A)^{\bt}).       
\]
\end{definition}

\begin{lemma}\label{cfshfqm}
For $E$ and $\oP$ as above, there is a canonical morphism
\[
 \Del(\uline{E})^{\sharp} \to  \Del(\uline{E}^{\sharp})       
\]
 of groupoid-valued functors on $d\cN^{\flat}_R$, for sheafification $\Del(\uline{E})^{\sharp}$ defined as in Definition \ref{shfgpd}. This induces an equivalence
\[
 \pi^0\Del(\uline{E})^{\sharp} \to \pi^0\Del(\uline{E}^{\sharp}).
\]
\end{lemma}
\begin{proof}
   An object of $\Del(\uline{E})^{\sharp}$ is a pair $(\omega, g) \in \mc( \uline{E}(A\ten_{A_0}B))\by E^0(B\ten_{A_0}B)^{\by}$, for $A_0 \to B$ a $\oP$-covering, satisfying the following conditions:
\begin{enumerate}
 \item $g* (\pr_1^*\omega) = (\pr_0^*\omega)*g \in \mc(  \uline{E}(A\ten_{A_0}B\ten_{A_0}B))$,
\item $\pr_{02}^*g = (\pr_{01}^*g)\cdot  (\pr_{12}^*g) \in E^0(B\ten_{A_0}B\ten_{A_0}B)^{\by}$,
\end{enumerate}
and we now describe the image of $(\omega, g)$.

First, form the cosimplicial $A_0$-algebra $(B/A_0)^{\bt}$ as in Definition \ref{shfqm}.   We now map $(B,\omega, g)$ to the object $\omega'$ of $\mc(\diag \uline{E}(A\ten_{A_0}(B/A_0)^{\bt}))$ given by
\[
  \omega'_n:=  (\pr_{0,n+1}^*g) * \pr_1^*\omega_n     \in E^{n+1}( (A^{\Delta^n})_0\ten_{A_0}(B/A_0)^{n+1}).  
\]

The remainder of the proof now follows exactly as for that of Lemma \ref{cfshfcgp}.
 \end{proof}

Now recall that $\oP$-covering morphisms are all assumed faithfully flat.
\begin{lemma}\label{cfshfqm2}
 Take $E:\Alg_R \to QM^*$ satisfying the conditions of Definition \ref{shfqm}, a ring $A \in \Alg_R$, an $A$-module $M$, and an object of $\Del(E)^{\sharp}(A)$ represented by $(\omega,g)$ for $\omega \in \mc(E(B))$. If the maps 
\[
 \H^*_{\omega}(E, M\ten_AB)\ten_BB' \to \H^*_{\omega}(E, M\ten_AB')
\]
are isomorphisms for all $\oP$-coverings $B \to B'$, then the morphism $\alpha$ of Lemma \ref{cfshfqm} induces an isomorphism
\[
 \DD^*_{(\omega,g)}(B\Del(\uline{E})^{\sharp},M) \to  \DD^*_{\alpha(\omega,g)}(B\Del(\uline{E}^{\sharp}),M).       
\]
\end{lemma}
\begin{proof}
The proof of Lemma \ref{cfshfcgp2} carries over verbatim.
\end{proof}

\section{Examples of derived moduli via quasi-comonoids}\label{qmegs}

We now show how to apply the machinery of the previous section to construct derived moduli stacks for several specific examples. The approach is the same as that used to construct derived deformations in \cite{paper2} and \cite{ddt2}, by finding suitable monads to construct a quasi-comonoid, and then taking the Deligne groupoid. 

Throughout this section, $R$ will be a G-ring admitting  a dualising complex in the sense of \cite{HaRD} Ch. V. Examples are $\Z$, any field,  or any  Gorenstein local ring.

\subsection{Finite schemes}\label{finqmsn}

For a fixed $r \in \N$, we now study a quasi-comonoid governing finite schemes of rank $r$. For any commutative $R$-algebra $A$, our moduli groupoid consists of algebra homomorphisms $A\to B$, with $B$ a locally free  $A$-module of rank $r$.

In order to construct a quasi-comonoid, we take the approach of  \S \ref{monads}. Let $\cB(A)$ be the category of $A$-modules, and define the monad $\top_A$ on $\cB(A)$ to be $\Symm_A$, so $\cD(A):= \cB(A)^{\top_A}$ is the category of commutative $A$-algebras.

\begin{definition}
 Working in $\cB(A)$, define $E_r: \Alg_R \to QM^*$ by setting $E_r(A)$ to be the quasi-comonoid $ E_r(A):=E(A^r)$, given by 
\[
  E^n_r(A)= \Hom_{A}(\top^n_A(A^r), A^r).
\]
\end{definition}

The following is then just Lemma \ref{monadmc} in this context:
\begin{lemma}\label{localfinqm}
For any commutative $R$-algebra 
$A$,
$\Del(E_r(A))$ is canonically isomorphic to the groupoid of commutative $A$-algebra structures on the $A$-module $A^r$.
 \end{lemma}

For our next step, note that $\Del(E_r) $ is not a stack, although the core of $\cD$ is. However,  the stackification $\Del(E_r)^{\sharp}$ in the Zariski topology is equivalent to the subgroupoid of $\cD(A)$ consisting of commutative $A$-algebras $B$ which are locally free $A$-modules of rank $r$. 

\begin{definition}
Define  $\Del(\uline{E_r})^{\sharp}: d\cN^{\flat}_{R}\to \gpd$ to be the stackification of $\Del(\uline{E_r})$ in the strict Zariski topology of Definition \ref{strictcover}. Likewise, define the simplicial groupoid-valued functor 
$\uline{\Del}(\uline{E_r})^{\sharp}$ on $d\cN^{\flat}_{R}$ by stackifying levelwise, so $(\uline{\Del}(\uline{E_r})^{\sharp})_n = (\uline{\Del}(\uline{E_r})_n)^{\sharp}$.
\end{definition}

\begin{proposition}\label{repfinqm}
The functor $\bar{W}\uline{\Del(\uline{E_r})^{\sharp}} \to \bS$ is representable by an almost finitely presented derived geometric $1$-stack. Moreover, 
\[
 \bar{W}\uline{\Del(\uline{E_r})^{\sharp}} \simeq \bar{W} \uline{\Del}(\uline{E_r})^{\sharp} \simeq \uline{\ddel}(\uline{E_r}^{\sharp}),
\]
where the last is defined using the quasi-comonoid sheafification of Definition \ref{shfqm}.
\end{proposition}
\begin{proof}
 For the first statement, we just show that $B\Del(\uline{E_r})^{\sharp}$ satisfies  the conditions of Theorem \ref{mylurieprerep}. Corollary \ref{delqmnice} ensures that  $B\Del(\uline{E_r})^{\sharp}$ is homogeneous and pre-homotopic. It follows from  Lemma \ref{qmcohodel} and the cotriple characterisation of Andr\'e-Quillen cohomology (\cite{W} \S 8.8)
 that
\[
 \DD^i_C(B\Del(\uline{E_r})^{\sharp}, M) \cong \Ext^{i+1}(\bL^{C/A}_{\bt}, C\ten_AM),
\]
 so the finiteness conditions of Theorem \ref{mylurieprerep} all hold, making $\bar{W}\uline{\Del(\uline{E_r})^{\sharp}}$ representable.

Similar arguments show that $ \bar{W} \uline{\Del}(\uline{E_r})^{\sharp}$ satisfies the conditions of Theorem \ref{mylurierep3}, so is  representable by a derived geometric $1$-stack. For the first equivalence, we thus apply Corollary \ref{predetectweak} to the morphism
\[
 B\Del(\uline{E_r})^{\sharp} \to \bar{W} \uline{\Del}(\uline{E_r})^{\sharp}
\]
coming from the map $\Del(\uline{E_r})^{\sharp} \to   \uline{\Del}(\uline{E_r})^{\sharp}$ of simplicial diagrams of groupoids.

For the second equivalence, we just use Lemmas \ref{cfshfqm} and \ref{cfshfqm2} to show that the composite  map
\[
 B\Del(\uline{E_r})^{\sharp} \to B\Del(\uline{E_r}^{\sharp}) \to \bar{W} \uline{\Del}(\uline{E_r})^{\sharp}
\]
satisfies the conditions of Corollary \ref{predetectweak}.
\end{proof}

\begin{remark}
 Alternatively, we can describe the associated derived geometric $1$-stack explicitly. The functor  $A \mapsto \mc(E_r( A))$ is an affine dg scheme, $(E^0_r)^{\by}= \GL_r$,  and  $\bar{W}\uline{ \Del(\uline{E_r})^{\sharp} }$ is just the hypersheafification of the quotient $B[\mc(E_r)/ (E^0_r)^{\by}]$ in the homotopy-Zariski (and indeed homotopy-\'etale) topologies. In the terminology of \cite{stacks2}, the simplicial affine dg scheme $B[\mc(E_r)/ (E^0_r)^{\by} ]$ is a homotopy derived Artin $1$-hypergroupoid representing  $\bar{W}\uline{ \Del(\uline{E_r})^{\sharp}}$.
\end{remark}

\begin{proposition}\label{finconsistentqm}
For $A \in s\cN^{\flat}_{R}$, 
the space $\uline{\ddel}(\uline{E_r}^{\sharp})(A)$ is functorially weakly equivalent to the nerve $\bar{W}\fG(A)$ of the $\infty$-groupoid  $\fG(A)$ of  simplicial commutative $A$-algebras $B$ 
for which $B\ten_A^{\oL}\pi_0A$ is weakly equivalent to a  locally free module of rank $r$ over $\pi_0A$.
\end{proposition}
\begin{proof}
For $F: s\Mod(A) \to s\Alg(A)$ the free commutative algebra functor on simplicial $A$-modules, 
the functor  $\beta_F^*$ of Definition \ref{betaF} maps from $\Del(E_r^{\sharp})(A)$ to $s\Alg(A)$, functorially in $A \in s\cN^{\flat}_{R}$, and all objects in its image are cofibrant.

When $A \in \Alg_R$, Corollary \ref{cfbeta} implies that $\beta_F^*(C)$ is   homotopy equivalent to the cotriple resolution $\bot_{\bt}C$ of $C$.
Thus $\pi_0(\beta_F^*(C) )\cong C$, and $ \pi_i(\beta_F^*(C) ) )=0$ for all $i>0$, so we indeed have a functor
\[
 \beta_F^*: \Del(E_r^{\sharp})(A) \to \fG(A)
\]
for all $A \in s\cN^{\flat}_{R}$ (using compatibility with base change). Moreover, this functor is an equivalence of $\infty$-groupoids when $A \in \Alg_R$.

Now, \cite{dmsch} Corollary \ref{dmsch-representaffine} and Example \ref{dmsch-fineg} give 
\[
 \DD^i_{[B]}(\bar{W}\fG,N) \cong \Ext^{i+1}_B(\bL^{B/A}, N\ten_AB),
\]
so $\beta_F^*$ satisfies the conditions of Remark \ref{predetectweakh}, giving an equivalence $\bar{W}\uline{\Del(\uline{E_r})^{\sharp}} \to \bar{W}\fG $. Combining with Proposition \ref{repfinqm} gives the required equivalence $\uline{\ddel}(\uline{E_r}^{\sharp}) \simeq \bar{W}\fG$.
\end{proof}

\subsection{Coherent sheaves}\label{cohmodsn}

Take a projective scheme $X$ over  $R$, and fix a numerical polynomial $h$. We now consider moduli of coherent sheaves on $X$ with Hilbert polynomial $h$. In other words, our underived moduli functor $\cM_h$ will set $\cM_h(A)$ to be the groupoid of coherent sheaves $\sF$ on $X \by \Spec A$ with $\Gamma( X \by \Spec A, \sF(n))$ locally free of rank $n$ for $n \gg 0$. 

\begin{definition}
     For $A \in \Alg_R$, we now define our base category $\cB(A)$ as follows. First form the category $\C(A)$ of 
graded $A$-modules $M = \bigoplus_{n \ge 0} M\{n\}$ in non-negative degrees, then let $\cB(A):= \pro(\C(A))$. Explicitly, objects of $\cB(A)$ are inverse systems $\{M^{\alpha}\}_{\alpha}$ in $\C(A)$, with
\[
 \Hom_{\cB(A)}( \{M^{\alpha}\}_{\alpha} , \{N^{\beta}\}_{\beta})= \Lim_{\beta}\LLim_{\alpha}\Hom_{\C(A)}( M^{\alpha},   N^{\beta}).    
\]
\end{definition}

If we let $S:= \bigoplus_{n\ge 0} \Gamma(X,\O_X(n))$, then there is a monad $\top$ on $\C(A)$ given by $\top M= S\ten_{R}M$, with multiplication $\mu:\top^2\to \top$ coming from the multiplication $S\ten_{R}A \to S$, and unit $\id \to \top$ coming from $R \to S$.

This extends naturally to a monad on $\cB(A)$, and we  
are now in the setting of \S \ref{monads}, since the category $\cD(A):=\cB(A)^{\top}$ of $\top$-algebras is just the pro-category of graded $S\ten_RA$-modules. We will be interested in  pro-modules of the form $M = \{\bigoplus_{n \ge p} M\{n\}\}_p$; for $M, M'$ any two such, we get
\begin{eqnarray*}
        \Hom_{\cD(A)} (M,M') &=& \Lim_q\LLim_p \Hom_{S\ten_RA}^{\bG_m}( M\{\ge p\}, M'\{\ge q\})\\
 &\cong& \Lim_q\LLim_{p\ge q} \Hom_{S\ten_RA}^{\bG_m}( M\{\ge p\}, M'\{\ge q\})\\
&\cong& \Lim_q\LLim_{p\ge q} \Hom_{S\ten_RA}^{\bG_m}( M\{\ge p\}, M')\\
&\cong& \LLim_{p} \Hom_{S\ten_RA}^{\bG_m}( M\{\ge p\}, M'),
\end{eqnarray*}
which ties in with \cite{Se2}.
      
\begin{definition}
Let $R^h \in \cB(R)$ be the inverse system $\{\bigoplus_{n \ge p} R^{h(n)}(n)\}_p$ of  graded modules, and  
 form the quasi-comonoid $ E_h(A):=E(R^h\ten_RA)$ given by 
\[
 E^n_h(A)= \Hom_{\cB(A)}(\top^n(R^h\ten_RA), R^h\ten_RA).     
\]
\end{definition}

Lemma \ref{monadmc} then implies that $\Del(E_h(A))$ is the subgroupoid of $\cD(A)$ consisting of $S$-module structures on $R^h\ten_RA$, and all isomorphisms between them. 

For our next step, note that $\Del(E_h) $ is not a stack, although the core of $\cD$ is. However, Lemma \ref{globalpol} adapts to show that the stackification $\Del(E_h)^{\sharp}$ in the pro-Zariski topology is equivalent to the subgroupoid of $\cD(A)$ consisting of pro-$(S\ten_RA)$-modules $M = \{ \bigoplus_{n \ge p} M\{n\} \}_p$,  with $M\{n\}$ locally free of rank $h(n)$ over $A$ for $n \gg 0$. 

Now, there is a  functor $\sF \mapsto \bigoplus_{n\ge 0} \Gamma(X, \sF(n))$ from coherent sheaves to $\cD(A)$, and the essential image just consists of finitely generated $(S\ten_RA)$-modules. 

\begin{definition}\label{fdefinitions}
Define the  functor $\mc_f(E_h):\Alg_R \to \Set$ by letting $\mc_f(E_h,A) \subset \mc(E_h,A)$ consist of $(S\ten_RA)$-modules isomorphic  to finitely generated modules. Next, define $\mmc_f(\uline{E}_h):d\cN^{\flat}_R \to \bS $ by 
\[
 \mmc_f(\uline{E}_h,A):= \mmc(\uline{E}_h,A)\by_{\mc(E_h,\H_0A)}\mc_f(E_h,\H_0A),
\]
with $\mc_f(\uline{E}_h,A)=  \mmc_f(\uline{E}_h,A)_0$.

Likewise, define $\Del_f(E_h):= [\mc_f(E_h)/(E^0)^{\by}]$,   $\Del_f(\uline{E}_h):= [\mc_f(\uline{E}_h)/ (E^0)^{\by}]$ and $\uline{\Del}_f (\uline{E}_h):= [\mmc_f(\uline{E}_h)/ (\uline{E}^0)^{\by}]$.
\end{definition}

\begin{definition}
 Let $E_h^{\sharp}$ be the quasi-comonoid sheafification (see Definition \ref{shfqm}) of $E_h$ in the pro-Zariski topology , and recall from Lemma \ref{cfshfqm} that $\Del(E_h^{\sharp}) \simeq \Del(E_h)^{\sharp}$. Define $\mc_f(E_h^{\sharp}) \subset \mc(E_h^{\sharp})$ to be the essential image of $\Del_f(E_h)^{\sharp} $  (which consists of modules isomorphic  to finitely generated modules).

Then define $\mmc_f(E_h^{\sharp}),  \mmc(\uline{E}_h^{\sharp}),\Del_f(E_h^{\sharp}), \Del_f(\uline{E}_h^{\sharp})$ and   $\uline{\Del}_f (\uline{E}_h^{\sharp})$ by adapting the formulae of  Definition \ref{fdefinitions}.
\end{definition}

 Lemma \ref{globalpol} adapts to show that the substack
 $\Del_f(E_h)^{\sharp}$ of  $\Del_f(E_h)^{\sharp}$  is equivalent  to the stack $\cM_h(A)$ defined at the beginning of the section. Note that $\mc_f(E_h) \to \mc(E_h)$ is formally \'etale, so  $ \uline{\Del}_f (\uline{E}_h)\to \uline{\Del} (\uline{E}_h) $ is also formally \'etale, as is $ \uline{\Del}_f (\uline{E}_h^{\sharp})\to \uline{\Del} (\uline{E}_h^{\sharp}) $.

\begin{proposition}\label{cohmod}
The functor $\bar{W}\uline{\Del_f(\uline{E_h})^{\sharp}} \to \bS$ is representable by an almost finitely presented derived geometric $1$-stack. Moreover, 
\[
 \bar{W}\uline{\Del_f(\uline{E_h})^{\sharp}} \simeq \bar{W} \uline{\Del}_f(\uline{E_h})^{\sharp} \simeq \uline{\ddel}_f(\uline{E_h}^{\sharp}).
\]
\end{proposition}
\begin{proof}
By exploiting the fact that $\mc_f(E_h) \to \mc_f(E_h)$ is formally \'etale, the proof of  Proposition \ref{repfinqm} carries over. The only substantial differences lie in the calculation of cohomology:
\[
 \DD^i_L(B\Del(\uline{E_h})^{\sharp}, M) \cong \Ext^{i+1}_{S\ten_RA}(L, L\ten_AM),
\]
and in establishing local finite presentation, which comes from adapting Lemma \ref{fglfp}.
\end{proof}
 
\begin{proposition}\label{cohconsistent}
 For $A \in s\cN^{\flat}_{R}$, 
the space $\uline{\ddel}_f(\uline{E_h}^{\sharp})(A)$  is functorially weakly equivalent to the nerve of the $\infty$-groupoid $\fM_h$ of derived quasi-coherent sheaves $\sF$ on $X \by \Spec A$ for which 
\begin{enumerate}
 \item  $\sF\ten_A^{\oL}\pi_0A$ is weakly equivalent to a quasi-coherent sheaf $\bar{\sF}$ on $X \by \Spec \pi_0A$, and
\item for all $n \gg 0$,  $\Gamma(X \by \Spec \pi_0A, \bar{\sF}(n))$ is a locally free $A$-module of rank $h(n)$.
\end{enumerate}
\end{proposition}
\begin{proof}
 Taking $F: s\cB(A) \to s\cD(A)$ to be the functor $M \mapsto S\ten_RM$ from simplicial graded $A$-modules to simplicial graded  $S\ten_RA$-modules, Definition \ref{betaF} gives us a functor 
\[
 \beta_F^*: \Del(E_h^{\sharp})(A)\to s\cD(A),
\]
preserving cofibrant objects. In particular, if $A \in \Alg_R$, and $L\in \cD(A)$ has $L\{n\}$ locally free for all $n$,  then $\beta_F^*(L)$ is cofibrant, and   Corollary \ref{cfbeta} implies that it is  homotopy equivalent to the cotriple resolution $S^{\ten_R \bt+1}\ten_RL $ of $L$ (as in \cite{W} \S 8.7.1). 

This ensures that for any $A \in s\cN^{\flat}_{R}$, $\beta_F^*$ maps objects of  $\Del_f(E_h^{\sharp})(A)$ to modules associated to objects of $\fM_h(A)$. Explicitly, let $\tilde{X}:= (\Spec S)-\{0\}$ be the canonical $\bG_m$-bundle over $X = \Proj S$, with $\pi: \tilde{X} \to X$ the projection and $j: \tilde{X} \to \Spec S$ the open immersion. Then our functor $\Del_f(E_h^{\sharp})(A) \to \fM_h(A) $ is 
\[
 (M, \omega) \mapsto  \beta_F^*(M, \omega)^{\sharp}= (\pi_*j^{-1}  \beta_F^*(M, \omega))^{\bG_m}.
\]

Now, for any $A \in \Alg_R$, we have
\[
 \DD^i_L(B\Del(E_h^{\sharp}),M )\cong \Ext^{i+1}_{\cD(A)}(L, L\ten_AM);
\]
adapting Serre's Theorem 
(\cite{Se2} \S 59) as in the proof of Proposition \ref{gradedcot}, this is isomorphic to $\Ext^{i+1}_{X\by \Spec A}(L^{\sharp}, L^{\sharp}\ten_AM)$. 

The remainder of the proof follows as for that of Proposition \ref{finconsistentqm}, replacing \cite{dmsch}  Example \ref{dmsch-fineg} with \cite{dmsch} Theorem \ref{dmsch-representdmod}.
\end{proof}

\begin{remark}[Associated DGLAs]
Since the functor $\top= S\ten_R$ is linear, our quasi-comonoid $E_h(A)$ is linear in the sense of
 \S \ref{linearqm}, so its normalisation $NE_h$ has the natural structure of a DGLA.   We could thus use Proposition \ref{mcqmdgla} to rewrite all these results in terms of the groupoid $[\mc(NE_h)/(E^0_h)^{\by}]$,   thereby making them consistent with  the approach of \cite{Quot}.
\end{remark}

\begin{remark}[Quot schemes]\label{quot}
If we wished to work with quotients of a fixed coherent sheaf $\sM$ on $X$, there are two equivalent approaches we could take. One is to choose $M \in \cD(R)$ with $M^{\sharp}=\sM$, and then to   replace $\cD(A)$ with the comma category $(M\ten_RA) \da \cD(A)$. The monad $\top$ would then be given by $\top(L) = (M\ten_RA) \oplus (S\ten_AL)$, with unit $L \to S\ten_AL$, and multiplication $\top^2(L) \to \top(L)$ being the map
\begin{eqnarray*}
 (M\ten_RA) \oplus (S\ten_RM\ten_RA) \oplus (S\ten_RS\ten_RL) &\to& (M\ten_RA) \oplus (S\ten_AL)\\
(m_1, s_1\ten m_2, s_2\ten s_3\ten l) &\mapsto& (m_1+s_1m_2, (s_2s_3)\ten l). 
\end{eqnarray*}
We would also have to replace $\mc_f(E_h)$ with the subset of $\mc(E_h)$ consisting of finitely generated $S\ten_RA$-modules under $M\ten_RA$ for which  the map $M\ten_RA \to L$ is surjective. 

The alternative approach would be to work with the construction of Example \ref{morphismsqm}, first forming the quasi-comonoid
\[
E:= E(A^{h'}, A^h)\by_{E(A^{h'})}\{M\},
\]
where $h'$ is the Hilbert polynomial of $M$, then taking the  subset of $\mc(E)$ cut out by the finiteness and surjectivity conditions above.

Beware, however, that in both of these approaches, the resulting quasi-comonoid is no longer linear, so cannot naturally be replaced with a DGLA.
\end{remark}

\subsection{Polarised projective schemes}

For a fixed  numerical polynomial $h \in \Q[t]$, we will now study the moduli of polarised projective schemes $(X,\O_X(1))$, with $\O_X(1)$ ample, for which $\Gamma(X, \O_X(n))$ is locally free of rank $h(n)$ for $n \gg 0$. 

We take $\cB(A)$ to be the pro-category of graded $A$-modules from \S \ref{cohmodsn}, with monad
\[
 \top= \Symm_A: \cB(A) \to \cB(A).       
\]
Setting $\cD(A):= \cB(A)^{\top}$ gives us the pro-category of $\bG_m$-equivariant commutative  $A$-algebras in non-negative degrees.

\begin{definition}
Let $R^h \in \cB(R)$ be the inverse system $\{\bigoplus_{n \ge p} R^{h(n)}(n)\}_p$ of  graded modules, and  
 form the quasi-comonoid $ E_h(A):=E(R^h\ten_RA)$ given by 
\[
 E^n_h(A)= \Hom_{\cB(A)}(\top^n(R^h\ten_RA), R^h\ten_RA).     
\]
\end{definition}

Lemma \ref{monadmc} then implies that $\Del(E_h(A))$ is the subgroupoid of $\cD(A)$ consisting of commutative ring structures on $R^h\ten_RA$, and all isomorphisms between them. By Lemma \ref{globalpol},  the stackification $\Del(E_h)^{\sharp}$ in the pro-Zariski topology is equivalent to
the subgroupoid of $\cD(A)$ consisting of commutative pro-$A$-algebras $B = \{ \bigoplus_{n \ge p} B\{n\} \}_p$,  with $B\{n\}$ locally free of rank $h(n)$ over $A$ for $n \gg 0$. 

We now proceed as in Definition \ref{fdefinitions}, letting $\mc_f(E_h) \subset \mc(E_h)$ consist of finitely generated $A$-algebras, and so on for $ \mmc_f(\uline{E}_h)$ etc. Note that $\mc_f(E_h) \to \mc(E_h)$ is formally \'etale. 

\begin{proposition}\label{gradedpolqm}
For $A \in \Alg_{\Q}$, $\Del_f(E_h(A)) $ is equivalent to the groupoid of flat 
polarised schemes $(X,\O_X(1))$ of finite type over $A$, with $\O_X(1)$ ample and the $A$-modules $\Gamma(X, \O_X(n))$ locally free of rank $h(n)$ for all $n\gg 0$. 
\end{proposition}
\begin{proof}
 This is essentially just Proposition \ref{gradedpol}.       
\end{proof}

\begin{proposition}\label{reppolqm}
The functor $\bar{W}\uline{\Del_f(\uline{E_h})^{\sharp}} \to \bS$ is representable by an almost finitely presented derived geometric $1$-stack. Moreover, 
\[
 \bar{W}\uline{\Del_f(\uline{E_h})^{\sharp}} \simeq \bar{W} \uline{\Del}_f(\uline{E_h})^{\sharp} \simeq \uline{\ddel}_f(\uline{E_h}^{\sharp}).
\]
\end{proposition}
\begin{proof}
 The proofs of Propositions  \ref{repfinqm} and \ref{cohmod} carry over, substituting the relevant finiteness properties from Proposition \ref{reppol}. In particular, Proposition \ref{gradedcot} adapts to show that
\[
  \DD^i_{[C]}(\uline{\ddel}_f(\uline{E_h}^{\sharp}),M)\cong \EExt^{i+1}_{X}(\bL^{X/B\bG_m\ten A}, \O_{X}\ten_AM),      
\]
 where      $X= \Proj(A\oplus C)$. 
\end{proof}

\begin{proposition}\label{polconsistentqm}
For $A \in s\cN^{\flat}_{\Q}$, the space $ \uline{\ddel}_f(\uline{E_h}^{\sharp}) $ is functorially weakly equivalent to the nerve $\bar{W}\fM(A)$ of the $\infty$-groupoid  $\fM(A)$ of derived geometric $0$-stacks $\fX$ over $B\bG_m\by \Spec A$ for which $X:=\fX\ten^{\oL}_A\H_0A$ is weakly equivalent to a flat projective scheme over $\H_0A$, with the polarisation $X \to B\bG_m\ten \H_0A$ ample with Hilbert polynomial $h$.
\end{proposition}
\begin{proof}
 This is essentially the same as Proposition \ref{polconsistent}, replacing $\beta^*$ with the  functor   $\beta_F^*: \Del(E_h(A))\to s\cD(A)$ on simplicial objects from     Definition \ref{betaF} (constructed similarly to those of Propositions \ref{finconsistentqm}  and \ref{cohconsistent}).
\end{proof}

\begin{remark}
 Note that the constructions of \S \ref{diagrams} immediately allow us to adapt $E_h$ to work with moduli of diagrams of polarised projective schemes, and in particular with morphisms  of such schemes. For moduli over a fixed base $\Proj S$, an alternative approach is to replace $\top$ with the monad $M \mapsto S\ten_R\Symm_AM$. Either of these approaches can be used to construct derived Hilbert schemes (by taking  $\mc_f(E)$ to be the subset of $\mc(E)$ consisting of finitely generated $A$-algebras $B$ for which $S\ten_RA \to B$ is surjective).  Propositions \ref{polconsistentqm}  and \ref{polconsistent} ensure that these approaches give equivalent derived stacks, as does \cite{Hilb}.
\end{remark}

\subsection{Finite group schemes}\label{fingpqmsn}

In order to study  moduli of finite group schemes, we follow the approach of \cite{dmsch} Example \ref{dmsch-modgpsch}, by noting that the nerve functor gives a full and faithful inclusion of the category of group schemes into the category of pointed simplicial schemes.

Given a finite group scheme $G$ over $\Spec A$, with $O(G):= \Gamma(G, \O_G)$ locally free of rank $r$, we thus look at the simplicial group scheme  $BG$ (for an explicit description, note that this is the same as $\bar{W}G$ from Definition \ref{barwdef2}). If we write $O(BG)^n:= \Gamma(BG_n, \O_{BG_n})$, then $O(BG)$ is a commutative cosimplicial augmented $A$-algebra, with $O(BG)^n$ locally free of rank $r^n$. 

\begin{lemma}\label{gpfet}
 The functor $G \mapsto BG$ from group schemes to pointed simplicial schemes is  formally \'etale.      
\end{lemma}
\begin{proof}
A simplicial scheme $X_{\bt}$ over $A$ is of the form $BG$ if and only if
\begin{enumerate}
        \item $X_0= \Spec A$, and
\item for all $n>1$ and all $0\le k \le n$, the maps $X_n \to \Hom_{\bS}(\L^{n,k}, X)$ are isomorphisms, where $\L^{n,k} \subset \Delta^n$ is the $k$th horn, obtained by removing the $k$th face from $\pd \Delta^n$.  
\end{enumerate}
Since any deformation of an isomorphism is an isomorphism, the result follows.
\end{proof}

We could now combine \S \ref{diagrams} with \S \ref{finqmsn} to obtain a quasi-comonoid functor governing moduli of such diagrams, but there is a far more efficient choice. If $A$ is a local ring, then  not only are the modules $O(BG)^n$ independent of $G$: we can also describe all operations except $\pd^0$. 

The following is adapted from \cite{htpy} Definition \ref{htpy-vdef}:
\begin{definition}\label{barvdef}
Define $\bar{V}\co  s\gp \to\bS$ by setting 
$$
\bar{V}G_n:=G_{n-1}\by G_{n-2}\by \ldots \by G_0
$$
with operations
\begin{eqnarray*}
\pd_0(g_{n-1},\ldots,g_0)&=& ((\pd_0g_{n-1})^{-1}g_{n-2}, \ldots, (\pd_0^{\phantom{0}n-1}g_{n-1})^{-1}g_0)\\
\pd_i(g_{n-1},\ldots,g_0)&=&(\pd_{i-1}g_{n-1},\ldots, \pd_1g_{n-i+1}, g_{n-i-1},g_{n-i-2},\ldots, g_0)\\
\sigma_i(g_{n-1},\ldots,g_0)&=&(\sigma_{i-1}g_{n-1},\ldots, \sigma_0g_{n-i}, g_{n-i},g_{n-i-1},\ldots, x_0).
\end{eqnarray*}
\end{definition}

\begin{lemma}
 There is a natural isomorphism $\bar{\phi}: \bar{V} \to \bar{W}$.        
\end{lemma}
\begin{proof}
As in \cite{htpy} Proposition \ref{htpy-psidef},  
the map $\bar{\phi}_G:  \bar{V}G \to \bar{W}G$ is given by 
\[
 \bar{\phi}(g_{n-1},\ldots,g_0)= (g_{n-1}, (\pd_0g_{n-1})^{-1}g_{n-2}\ldots, (\pd_0g_1)^{-1}g_0).       
\]
\end{proof}

We can therefore replace $B$ with the functor $\bar{V}$, and consider the simplicial scheme $\bar{V}G$, which has the property that $\pd_0$ is the only simplicial operation to depend on the group structure of $G$. We now proceed along the same lines as \cite{higher} \S \ref{higher-calgsn}.

\begin{definition}\label{delta*}
Define $\Delta_*$ to be the subcategory of the ordinal number category $\Delta$ containing only  those morphisms fixing $0$.  Given a category $\C$, define the category  $c_+\C$ of  almost cosimplicial diagrams in $\C$ to consist of  functors  $\Delta_* \to \C$. Thus an almost cosimplicial object $X^*$ consists of objects $X^n \in \C$, with all of the operations $\pd^i, \sigma^i$ of a cosimplicial complex except $\pd^0$, satisfying the usual relations. 
\end{definition}

\begin{definition}
 Define the functor  $F_{\pd}: c_+\Mod(A) \to c\Mod(A)$ from     almost cosimplicial $A$-modules to $A$-modules by
\[
    (F_{\pd} M^*)^n = M^n\oplus M^{n-1} \oplus \ldots \oplus M^0,    
\]
with operations
   \begin{eqnarray*}
\pd^i(v_n,\ldots,v_0)&=& (\pd^iv_n,\pd^{i-1}v_{n-1}, \ldots, \pd^1v_{n-i+1}, 0, v_{n-i}, \ldots, v_1,v_0)\\
\sigma^i(v_n,\ldots,v_0)&=& (\sigma^iv_n, \ldots, \sigma^{1}v_{n-i+1}, \sigma^0v_{n-i} + v_{n-i-1}, v_{n-i-2},  \ldots,v_0).
\end{eqnarray*}
\end{definition}

By the argument of \cite{higher} Lemma \ref{higher-Gdef}, $F_{\pd}$ is left adjoint to the  functor $U_{\pd}: c\Mod(A) \to c_+\Mod(A)$ given by forgetting $\pd^0$. Moreover, this adjunction is monadic, so for the monad $\top_{\pd}:= F_{\pd}U_{\pd}$, there is a natural equivalence
\[
 c\Mod(A) \simeq  c_+\Mod(A)^{\top_{\pd}}.      
\]

In fact, we can go further than this. By \cite{higher} \S \ref{higher-calgsn}, the monad $\Symm$ distributes over $\top_{\pd}$, so the composite monad $ \Symm \circ \top_{\pd}$ is another monad. Moreover,
\[
     c\Alg(A) \simeq  c_+\Mod(A)^{ (\Symm \circ \top_{\pd})}.   
\]

We wish to modify this slightly, since we are only interested in augmented cosimplicial $A$-algebras, or equivalently non-unital cosimplicial $A$-algebras (taking augmentation ideals).
We thus replace $\Symm$ with $\Symm^+:= \bigoplus_{n>0} S^n$, and set $\top := \Symm^+ \circ \top_{\pd}$. Then $ c_+\Mod(A)^{\top}$ is equivalent to the category  $c\NAlg(A)$  of non-unital commutative cosimplicial $A$-algebras.

\begin{definition}\label{barv+def}
Define a functor $E_r:\Alg_R \to QM^*$ by first forming the almost cosimplicial module $\bar{V}(A^r) \in  c_+\Mod(A)$ as
\[
 \bar{V}(A^r)^n:= A^{r^n}= \overbrace{(A^r)\ten_A(A^r) \ten_A\ldots \ten_A(A^r)}^n,      
\]
with operations dual to those of Definition \ref{barvdef} (without $\pd^0$). We then let $\bar{V}_+(A^r):= \ker(\bar{V}(A^r) \to A)$, where the $A$ is the constant diagram $\bar{V}(A^0)$, so 
\[
 \bar{V}_+(A^r)^n = \ker( (\pd^1)^n: \bar{V}(A^r)^n \to \bar{V}(A^r)^0).
\]
 
We now set $E_r(A)$ to be the quasi-comonoid $ E_r(A):=E( \bar{V}_+(A^r))$ of \S \ref{monads}, given by
\[
  E^n_r(A)= \Hom_{c_+\Mod(A) }(\top^n_A \bar{V}_+(A^r), \bar{V}_+(A^r)).
\]
 \end{definition}

\begin{definition}
Define  $\Del(\uline{E_r})^{\sharp}: d\cN^{\flat}_{R}\to \gpd$ to be the stackification of $\Del(\uline{E_r})$ in the strict Zariski topology of Definition \ref{strictcover}. Likewise, define the simplicial groupoid-valued functor 
$\uline{\Del}(\uline{E_r})^{\sharp}$ on $d\cN^{\flat}_{R}$ by stackifying levelwise, so $(\uline{\Del}(\uline{E_r})^{\sharp})_n = (\uline{\Del}(\uline{E_r})_n)^{\sharp}$.
\end{definition}

\begin{definition}
    By Lemma \ref{monadmc}, $\mc(E_r(A))$ is equivalent to the groupoid of pointed simplicial affine schemes $X$ over $A$ for which $U_{\pd}O(X) \cong \bar{V}(A^r) \in  c_+\Mod(A)$. We then define
\[
 \mc_g(E_r(A)) \subset \mc(E_r(A))       
\]
to consist of simplicial schemes of the form $BG$, for $G$ a group scheme. By Lemma \ref{gpfet}, this inclusion of functors is formally \'etale.
 
We define $\Del_g(E_r)$, $\mmc_g(\uline{E}_r)$, $\mc_g(E_r^{\sharp})$ and so on similarly. 
\end{definition}

\begin{proposition}\label{repfingpqm}
The functor $\bar{W}\uline{\Del_g(\uline{E_r})^{\sharp}} \to \bS$ is representable by an almost finitely presented derived geometric $1$-stack. 
Moreover, 
\[
 \bar{W}\uline{\Del_g(\uline{E_r})^{\sharp}} \simeq \bar{W} \uline{\Del}_g(\uline{E_r})^{\sharp} \simeq \uline{\ddel}_g(\uline{E_r}^{\sharp}),
\]
where the last is defined using the quasi-comonoid sheafification of Definition \ref{shfqm}.
\end{proposition}
\begin{proof}
 The proofs of Propositions \ref{repfinqm} and \ref{reppolqm}  carry over. The only differences lie in a straightforward check that $\Del_g(E_r)^{\sharp}$ is locally of finite presentation, and in the calculation of cohomology groups.

For the comonad  $\bot:= \Symm^+ \circ F_{\pd}U_{\pd}$ on 
$c\NAlg(A)$,
we get a canonical simplicial resolution  $\bot_{\bt}S$, given by $\bot_nS:= \bot^{n+1}S$. 
For $A \in \Alg_R$, the proof of \cite{higher} Lemma \ref{higher-lalg} then shows that $A \oplus (\bot_{\bt}S)^m$ is a cofibrant resolution of 
 $A \oplus S^m$ for all $m$, whenever $S^m$ is projective as an $A$-module. If we set $\bL^{\bot}_{\bt}(S):= \Omega((A\oplus  \bot_{\bt}S)/A)$, this means that
  the simplicial complex $\bL^{\bot}_{\bt}(S)^m$ is a model for the cotangent complex of $A \oplus S^m$. 

If $S$ is levelwise projective as an $A$-module, then  by \cite{higher} Lemma \ref{higher-lproj},
$\bL^{\bot}_n(S)$ is a projective object of $c\Mod(S)$. Thus, for $S \in E_r^{\sharp}(A)$, Lemma \ref{qmcohodel} gives
 that
\[
 \DD^i_{[S]}(\uline{\ddel}_f(\uline{E_r}^{\sharp}),M)\cong \EExt^{i+1}_S( \bL^{\bot}_{\bt}(S), S\ten_AM),
\]
which satisfies the  finiteness conditions of Theorem \ref{mylurierep3}.
\end{proof}

\begin{definition}
 Given a flat group scheme over a ring $A$, follow \cite{illusiethesis} \S 2.5.1 in defining the $A$-complex $\chi^{G/A}$ by
\[
 \chi^{G/A}:= \oL e^* \bL^{G/A},
\]
where $e: \Spec A \to G$ is the unit of the group structure. This has a canonical $G$-action, and we write $\pmb{\chi}^{G/A}$ for the associated complex on $BG$.

As in \cite{illusiethesis} \S 4.1, $\chi^{G/A}$ is perfect and concentrated in chain degrees $0,1$. 
We 
set $\omega^{G/A}:= \H_0(\chi^{G/A})$, 
with $\pmb{\omega}^{G/A} $ 
the associated sheaf on $BG$.
\end{definition}

\begin{definition}
 For a cosimplicial ring $S$, we make $c\Mod(S)$  into a simplicial category by setting (for $K \in \bS$)
$$
(M^K)^n: = (M^n)^{K_n}, 
$$
as an $S^n$-module. This has a left adjoint, which we denote by  $M \mapsto M\ten K$. Given a cofibration $K \into L$ in $\bS$, we write $M \ten (L/K):= (M\ten L)/(M\ten K)$.

Given $M \in c\Mod(S)$, define  $\underline{M}$  to be the bicosimplicial complex given in horizontal level $i$ by  $\underline{M}^{i\bt}= M\ten \Delta^i$. Let $N_c\uline{M}$ we the cochain complex in $c\Mod(S)$ given by taking the horizontal cosimplicial normalisation of Definition \ref{normcdef}.
\end{definition}

\begin{lemma}\label{fingpcoho}
 Given an affine group scheme $G$ over $A$, with $\Gamma(G, \O_G)$ locally free of rank $r$ over $A$, let $S \in \Del(E_r)^{\sharp}(A)$ be the associated cosimplicial ring
\[
 S^n:= \ker(\Gamma((BG)_n, \O_{(BG)_n)}) \to A).
\]
Then for $M \in \Mod_A$,
\[
 \DD^i_{[S]}(\uline{\ddel}_g(\uline{E_r}^{\sharp}),M)\cong \EExt^i_{BG}(\bL^{BG/A}, \ker(\O_{BG}\ten_AM \to M)).
\]

In particular,
\[
 \DD^i_{[S]}(\uline{\ddel}_g(\uline{E_r}^{\sharp}),M)\cong\EExt^{i+2}_{BG}(\pmb{\chi}^{G/A} , \O_{BG}\ten_AM)
\]
for $i \ge 1$. For low degrees, there is  an exact sequence 
\begin{eqnarray*}
0 \to &\Hom_{BG}(\pmb{\omega}^{G/A} , \O_{BG}\ten_AM)& \to \Hom_A(\omega^{G/A} ,M)\to \\
\DD^{-1}_{[S]}( \uline{\ddel}_g(\uline{E_r}^{\sharp}, M)\to   &\Ext^{1}_{BG}(\pmb{\chi}^{G/A} , \O_{BG}\ten_AM)&\to  \Ext^1_A(\chi^{G/A},M) \to \\
\DD^{0}_{[S]}( \uline{\ddel}_g(\uline{E_r}^{\sharp}, M)\to   &\Ext^{2}_{BG}(\pmb{\chi}^{G/A} , \O_{BG}\ten_AM)&\to  0.
\end{eqnarray*}
\end{lemma}
\begin{proof}
Write $L:= \bL^{\bot}_{\bt}(S)$ , defined as in the proof of Proposition \ref{repfingpqm}, and recall that this is a projective object of $c\Mod(S)$. Observe that in the terminology of \cite{stacks2}, $\Spec(A \oplus S)$ is a derived fppf $1$-hypergroupoid, and a derived Artin $1$-hypergroupoid whenever $G$ is smooth.

If $G$ is smooth, then  \cite{stacks2} Proposition \ref{stacks-cfolsson} shows that
\[
 \EExt_{S}^i(\Tot N_c\uline{L}, S\ten_A M) \cong  \EExt^i_{BG}(\bL^{BG/A}, \ker(\O_{BG}\ten_AM \to M)),
\]
where $\Tot$ is the total complex functor.
For general $G$, the same formula holds, 
  since \cite{stacks2} Proposition \ref{stacks-cotgood} only uses the Artin hypothesis to prove that $\Tot N_cN^s\uline{L}$ is projective, while the descent argument from the proof of \cite{stacks2} Proposition \ref{stacks-cfhagcot} works for all faithfully flat morphisms.

Now, \cite{stacks2} Lemmas \ref{stacks-powers} and \ref{stacks-omegacalc} combine to show that $N_c^i\uline{L}$ is acyclic for $i>1$, while $N_c^1\uline{L} $ is the pullback to $BG$ of the cotangent complex of a trivial relative derived $1$-hypergroupoid. It then follows from \cite{stacks2} Lemma \ref{stacks-ttruncate} that there are canonical isomorphisms
\[
 \Ext^*_S(N_c^1\uline{L}, P) \cong \Ext^*_{S^0}( (N_c^1\uline{L})^0, P^0)
\]
for all $P \in c\Mod(S)$. Since $S^0=0$, this means that
\[
 \EExt_{S}^i(\Tot N_cN^s\uline{L}, S\ten_A M) \cong \EExt^i(L,  S\ten_A M).
\]
Thus, combined with the proof of Proposition \ref{repfingpqm}, we get
\[
 \DD^i_{[S]}(\uline{\ddel}_g(\uline{E_r}^{\sharp}),M)\cong \EExt^i_{BG}(\bL^{BG/A}, \ker(\O_{BG}\ten_AM \to M)).
\]

Finally, $\bL^{BG/A} \simeq  \pmb{\chi}^{G/A}[1]$, so the exact sequence $0 \to S \to A \oplus S \to A \to 0$ of $S$-modules gives   the required long exact sequence.
\end{proof}

\begin{proposition}\label{fingpconsistentqm}
For $A \in s\cN^{\flat}_{R}$, 
the space $\uline{\ddel}_g(\uline{E_r}^{\sharp})(A)$ is functorially weakly equivalent to the nerve $\bar{W}\fM(A)$ of the $\infty$-groupoid  $\fM(A)$ of  pointed derived geometric $1$-stacks $X$ over $A$
for which $X\ten_A^{\oL}\pi_0A$ is weakly equivalent to the nerve of a flat rank $r$ group scheme over $\pi_0A$.
\end{proposition}
\begin{proof}
We work along the same lines as Proposition \ref{finconsistentqm}. As a consequence of Proposition \ref{repfingpqm} and Remark \ref{predetectweakh}, it suffices to construct a natural transformation
\[
 \Phi:\Del_g(\uline{E_r})^{\sharp}\to \fM
\]
of $\infty$-groupoids, inducing equivalences on $\pi^0$ and isomorphisms on $\DD^i$ of the nerves.

Given an object of $ \Del_g(\uline{E_r})^{\sharp}(A)$, we get $M \in c_+\Mod(A_0)$, locally isomorphic (over $A_0$) to $\bar{V}_+(A^r_0)$, together with elements 
\[
 \omega_n \in \Hom_{c_+\Mod(A_0)}(\top^{n+1}M, M\ten_{A_0}(A^{I^n})_0),
\]
satisfying the Maurer--Cartan relations of Definition \ref{mcdefqm}. For the free functor $F: sc_+\Mod(A) \to sc\NAlg(A) $ from simplicial almost cosimplicial $A$-modules to non-unital 
simplicial  cosimplicial commutative $A$-algebras,  Definition \ref{betaF} thus gives us a functor 
\[
 \beta_F^*: \Del(E_r)^{\sharp}(A)\to sc\NAlg(A). 
\]

We therefore get a functor $\phi: \Del(E_h(A)) \to sc\Alg(A)\da A$ to the category of augmented simplicial cosimplicial commutative $A$-algebras, given by
$\omega \mapsto A \oplus \beta_F^*(\omega)$. Moreover, it follows from Definition \ref{betaF} that all objects in the image of  $\phi$ are  Reedy cofibrant. If $A \in \Alg_R $ and $\omega$ corresponds to a group scheme $G$, then arguing as in the proof of Proposition \ref{finconsistentqm}, $\phi(\omega)$ is a cofibrant resolution of $O(BG)$ as a simplicial augmented cosimplicial commutative $A$-algebra. Therefore for arbitrary $A \in s\cN^{\flat}_R$, 
\[
\Spec (\phi(\omega)\ten_A^{\oL}\pi_0A) 
\]
 is a pointed fppf $1$-hypergroupoid, so $\Spec \phi(\omega)$ is a pointed derived  fppf $1$-hypergroupoid.

We therefore define $\Phi(\omega)$ to be the homotopy-fppf hypersheafification of $\Spec \phi(\omega)$. By \cite{stacks2} Theorem \ref{stacks-bigthm}, $\Phi(\omega)$ is a pointed derived geometric fppf $1$-stack whenever $\omega \in \Del_g(\uline{E_r})^{\sharp}(A)$. By \cite{toenflat} Theorem 0.1, this is the same as  a derived geometric Artin $1$-stack.
 
When $A \in \Alg_R$, with $\omega$ corresponding to $G$, we have seen  that $\Phi(\omega)$ is just the classifying stack $BG$. For arbitrary $A \in s\cN^{\flat}_R$, this means that $ \Phi(\omega)\ten_A^{\oL}\pi_0A$ is of the form $BG$ for some flat  rank $r$ group scheme $G$ over $\pi_0A$. Thus $\Phi$ indeed gives a functor $\Phi: \Del_g(\uline{E_r})^{\sharp}\to \fM$.

The arguments above have shown that $\pi^0\Phi$ is an equivalence, since the space of group homomorphisms $G\to G'$ corresponds to the space of pointed morphisms $BG \to BG'$. To see that $\Phi$ gives isomorphisms 
\[
 \DD^i_{\omega}(B\Del(E_r)^{\sharp}, M) \to \DD^i_{\Phi(\omega)}(\bar{W}\fM,M), 
\]
we combine
 Lemma  \ref{fingpcoho} with \cite{dmsch} Theorem \ref{dmsch-augrepresentdaffine}.
\end{proof}

\bibliographystyle{alphanum}
\addcontentsline{toc}{section}{Bibliography}
\bibliography{references}

\begin{thebibliography}{CFK2}

\bibitem[Ane]{anel}
Mathieu Anel.
\newblock Grothendieck topologies from unique factorisation systems.
\newblock arXiv:0902.1130v2 [math.AG], 2009.

\bibitem[BK]{bousfieldkan}
A.~K. Bousfield and D.~M. Kan.
\newblock {\em Homotopy limits, completions and localizations}.
\newblock Lecture Notes in Mathematics, Vol. 304. Springer-Verlag, Berlin,
  1972.

\bibitem[CFK1]{Quot}
Ionu{\c{t}} Ciocan-Fontanine and Mikhail Kapranov.
\newblock Derived {Q}uot schemes.
\newblock {\em Ann. Sci. {\'E}cole Norm. Sup. (4)}, 34(3):403--440, 2001.

\bibitem[CFK2]{Hilb}
Ionu{\c{t}} Ciocan-Fontanine and Mikhail~M. Kapranov.
\newblock Derived {H}ilbert schemes.
\newblock {\em J. Amer. Math. Soc.}, 15(4):787--815 (electronic), 2002.

\bibitem[CLM]{loop}
Frederick~R. Cohen, Thomas~J. Lada, and J.~Peter May.
\newblock {\em The homology of iterated loop spaces}.
\newblock Springer-Verlag, Berlin, 1976.
\newblock Lecture Notes in Mathematics, Vol. 533.

\bibitem[CR1]{CRfib}
A.~M. Cegarra and J.~Remedios.
\newblock Diagonal fibrations are pointwise fibrations.
\newblock {\em J. Homotopy Relat. Struct.}, 2(2):81--92, 2007.

\bibitem[CR2]{CRdiag}
A.~M. Cegarra and Josu{\'e} Remedios.
\newblock The relationship between the diagonal and the bar constructions on a
  bisimplicial set.
\newblock {\em Topology Appl.}, 153(1):21--51, 2005.

\bibitem[DK]{pathgpd}
W.~G. Dwyer and D.~M. Kan.
\newblock Homotopy theory and simplicial groupoids.
\newblock {\em Nederl. Akad. Wetensch. Indag. Math.}, 46(4):379--385, 1984.

\bibitem[GJ]{sht}
Paul~G. Goerss and John~F. Jardine.
\newblock {\em Simplicial homotopy theory}, volume 174 of {\em Progress in
  Mathematics}.
\newblock Birkh{\"a}user Verlag, Basel, 1999.

\bibitem[Gro1]{EGA3.1}
A.~Grothendieck.
\newblock {\'E}l{\'e}ments de g{\'e}om{\'e}trie alg{\'e}brique. {III}.
  {\'e}tude cohomologique des faisceaux coh{\'e}rents. {I}.
\newblock {\em Inst. Hautes {\'E}tudes Sci. Publ. Math.}, (11):167, 1961.

\bibitem[Gro2]{GroHilb}
Alexander Grothendieck.
\newblock Techniques de construction et th{\'e}or{\`e}mes d'existence en
  g{\'e}om{\'e}trie alg{\'e}brique. {IV}. {L}es sch{\'e}mas de {H}ilbert.
\newblock In {\em S{\'e}minaire {B}ourbaki, {V}ol.\ 6}, pages Exp.\ No.\ 221,
  249--276. Soc. Math. France, Paris, 1995.

\bibitem[Har1]{HaRD}
Robin Hartshorne.
\newblock {\em Residues and duality}.
\newblock Lecture notes of a seminar on the work of A. Grothendieck, given at
  Harvard 1963/64. With an appendix by P. Deligne. Lecture Notes in
  Mathematics, No. 20. Springer-Verlag, Berlin, 1966.

\bibitem[Har2]{Ha}
Robin Hartshorne.
\newblock {\em Algebraic geometry}.
\newblock Springer-Verlag, New York, 1977.
\newblock Graduate Texts in Mathematics, No. 52.

\bibitem[HS]{HaimanSturmfels}
Mark Haiman and Bernd Sturmfels.
\newblock Multigraded {H}ilbert schemes.
\newblock {\em J. Algebraic Geom.}, 13(4):725--769, 2004.

\bibitem[Ill1]{Ill1}
Luc Illusie.
\newblock {\em Complexe cotangent et d{\'e}formations. {I}}.
\newblock Springer-Verlag, Berlin, 1971.
\newblock Lecture Notes in Mathematics, Vol. 239.

\bibitem[Ill2]{illusiethesis}
Luc Illusie.
\newblock Cotangent complex and deformations of torsors and group schemes.
\newblock In {\em Toposes, algebraic geometry and logic ({C}onf., {D}alhousie
  {U}niv., {H}alifax, {N}.{S}., 1971)}, pages 159--189. Lecture Notes in Math.,
  Vol. 274. Springer, Berlin, 1972.

\bibitem[LMB]{champs}
G{\'e}rard Laumon and Laurent Moret-Bailly.
\newblock {\em Champs alg{\'e}briques}, volume~39 of {\em Ergebnisse der
  Mathematik und ihrer Grenzgebiete. 3. Folge. A Series of Modern Surveys in
  Mathematics [Results in Mathematics and Related Areas. 3rd Series. A Series
  of Modern Surveys in Mathematics]}.
\newblock Springer-Verlag, Berlin, 2000.

\bibitem[Lur]{lurie}
J.~Lurie.
\newblock {\em Derived Algebraic Geometry}.
\newblock PhD thesis, M.I.T., 2004.
\newblock www.math.harvard.edu/$\sim$lurie/papers/DAG.pdf or
  http://hdl.handle.net/1721.1/30144.

\bibitem[Mum]{Mum}
David Mumford.
\newblock {\em Lectures on curves on an algebraic surface}.
\newblock With a section by G. M. Bergman. Annals of Mathematics Studies, No.
  59. Princeton University Press, Princeton, N.J., 1966.

\bibitem[Pri1]{paper2}
J.~P. Pridham.
\newblock Deformations of schemes and other bialgebraic structures.
\newblock {\em Trans. Amer. Math. Soc.}, 360(3):1601--1629, 2008.

\bibitem[Pri2]{htpy}
J.~P. Pridham.
\newblock Pro-algebraic homotopy types.
\newblock {\em Proc. London Math. Soc.}, 97(2):273--338, 2008.
\newblock arXiv math.AT/0606107 v8.

\bibitem[Pri3]{higher}
J.~P. Pridham.
\newblock Derived deformations of {A}rtin stacks.
\newblock arXiv:0805.3130v2 [math.AG], submitted, 2009.

\bibitem[Pri4]{monad}
J.~P. Pridham.
\newblock The homotopy theory of strong homotopy algebras and bialgebras.
\newblock {\em Homology, Homotopy Appl.}, 12(2):39--108, 2010.
\newblock arXiv:0908.0116v2 [math.AG].

\bibitem[Pri5]{ddt2}
J.~P. Pridham.
\newblock Derived deformations of schemes.
\newblock {\em Comm. Anal. Geom.}, 20(3):529--563, 2012.
\newblock arXiv:0908.1963v1 [math.AG].

\bibitem[Pri6]{dmsch}
J.~P. Pridham.
\newblock Derived moduli of schemes and sheaves.
\newblock {\em J. K-Theory}, 10(1):41--85, 2012.
\newblock arXiv:1011.2189v3 [math.AG].

\bibitem[Pri7]{drep}
J.~P. Pridham.
\newblock Representability of derived stacks.
\newblock {\em J. K-Theory}, 10(2):413--453, 2012.
\newblock arXiv:1011.2742v4 [math.AG].

\bibitem[Pri8]{stacks2}
J.~P. Pridham.
\newblock Presenting higher stacks as simplicial schemes.
\newblock {\em Adv. Math.}, to appear.
\newblock arXiv:0905.4044v3 [math.AG].

\bibitem[Qui1]{QRat}
Daniel Quillen.
\newblock Rational homotopy theory.
\newblock {\em Ann. of Math. (2)}, 90:205--295, 1969.

\bibitem[Qui2]{Q}
Daniel Quillen.
\newblock On the (co-) homology of commutative rings.
\newblock In {\em Applications of Categorical Algebra (Proc. Sympos. Pure
  Math., Vol. XVII, New York, 1968)}, pages 65--87. Amer. Math. Soc.,
  Providence, R.I., 1970.

\bibitem[Ser]{Se2}
Jean-Pierre Serre.
\newblock Faisceaux alg{\'e}briques coh{\'e}rents.
\newblock {\em Ann. of Math. (2)}, 61:197--278, 1955.

\bibitem[To{\"e}]{toenflat}
Bertrand To{\"e}n.
\newblock Flat descent for {A}rtin $n$-stacks.
\newblock {\em Compos. Math.}, to appear.
\newblock arXiv:0911.3554v2 [math.AG].

\bibitem[TV]{hag2}
Bertrand To{\"e}n and Gabriele Vezzosi.
\newblock Homotopical algebraic geometry. {II}. {G}eometric stacks and
  applications.
\newblock {\em Mem. Amer. Math. Soc.}, 193(902):x+224, 2008.
\newblock arXiv math.AG/0404373 v7.

\bibitem[Wei]{W}
Charles~A. Weibel.
\newblock {\em An introduction to homological algebra}.
\newblock Cambridge University Press, Cambridge, 1994.

\end{thebibliography}
\end{document}